\documentclass[notitlepage,11pt,reqno]{amsart}
\usepackage[foot]{amsaddr}
\RequirePackage[OT1]{fontenc}
\RequirePackage[normalem]{ulem}
\RequirePackage{amssymb,nicefrac,bm,upgreek,mathtools,verbatim,enumerate}
\RequirePackage[mathscr]{eucal}
\RequirePackage{dsfont}

\RequirePackage[final]{hyperref}

\usepackage[margin=1in]{geometry}
\newcommand{\stkout}[1]{\ifmmode\text{\sout{\ensuremath{#1}}}\else\sout{#1}\fi}
\newcommand{\ttup}[1]{\textup{(}#1\textup{)}}
\usepackage{color}
\definecolor{dmagenta}{rgb}{.6,.0,.8}
\definecolor{dblue}{rgb}{.0,.0,.5}
\definecolor{mblue}{rgb}{.0,.0,.8}
\definecolor{ddblue}{rgb}{.0,.0,.4}
\definecolor{dred}{rgb}{.6,.0,.0}
\definecolor{dgreen}{rgb}{.0,.5,.0}
\definecolor{Eeom}{rgb}{.0,.0,.5}

\newtheorem{lemma}{Lemma}[section]
\newtheorem{theorem}{Theorem}[section]
\newtheorem{proposition}{Proposition}[section]
\newtheorem{corollary}{Corollary}[section]

\theoremstyle{definition}
\newtheorem{definition}{Definition}[section]
\newtheorem{assumption}{Assumption}[section]

\theoremstyle{remark}
\newtheorem{remark}{Remark}[section]
\numberwithin{equation}{section}

\newcommand{\compcent}[1]{\vcenter{\hbox{$#1\circ$}}}
\newcommand{\comp}{\mathbin{\mathchoice
{\compcent\scriptstyle}{\compcent\scriptstyle}
{\compcent\scriptscriptstyle}{\compcent\scriptscriptstyle}}}

\hypersetup{
  colorlinks=true,
  citecolor=mblue,
  linkcolor=mblue,
  frenchlinks=false,
  pdfborder={0 0 0},
  naturalnames=false,
  hypertexnames=false,
  breaklinks}
\usepackage[capitalize,nameinlink]{cleveref}
\usepackage[abbrev,msc-links,nobysame,citation-order]{amsrefs} %backrefs
\crefname{section}{Section}{Sections}
\crefname{subsection}{Section}{Sections}
\crefname{hypothesis}{Hypothesis}{Conditions}
\crefname{assumption}{Assumption}{Assumptions}
\crefname{lemma}{Lemma}{Lemmas}
\Crefname{figure}{Figure}{Figures}

\crefformat{equation}{\textup{#2(#1)#3}}
\crefrangeformat{equation}{\textup{#3(#1)#4--#5(#2)#6}}
\crefmultiformat{equation}{\textup{#2(#1)#3}}{ and \textup{#2(#1)#3}}
{, \textup{#2(#1)#3}}{, and \textup{#2(#1)#3}}
\crefrangemultiformat{equation}{\textup{#3(#1)#4--#5(#2)#6}}%
{ and \textup{#3(#1)#4--#5(#2)#6}}{, \textup{#3(#1)#4--#5(#2)#6}}%
{, and \textup{#3(#1)#4--#5(#2)#6}}

\Crefformat{equation}{#2Equation~\textup{(#1)}#3}
\Crefrangeformat{equation}{Equations~\textup{#3(#1)#4--#5(#2)#6}}
\Crefmultiformat{equation}{Equations~\textup{#2(#1)#3}}{ and \textup{#2(#1)#3}}
{, \textup{#2(#1)#3}}{, and \textup{#2(#1)#3}}
\Crefrangemultiformat{equation}{Equations~\textup{#3(#1)#4--#5(#2)#6}}%
{ and \textup{#3(#1)#4--#5(#2)#6}}{, \textup{#3(#1)#4--#5(#2)#6}}%
{, and \textup{#3(#1)#4--#5(#2)#6}}

\crefdefaultlabelformat{#2\textup{#1}#3}
\newcommand{\Act}{{\mathbb{U}}}                       % Action space
\newcommand{\Uadm}{{\mathfrak{U}}}                    % Admissible Controls
\newcommand{\Ussm}{{{\mathfrak{U}}_{\mathrm{ssm}}}}   % Stable Stationary Controls
\newcommand{\RR}{\mathds{R}}      % Reals
\newcommand{\NN}{\mathds{N}}      % Natural numbers
\newcommand{\ZZ}{\mathds{Z}}      % Integers
\newcommand{\DD}{\mathds{D}}
\newcommand{\Rd}{{\mathds{R}^{d}}}
\DeclareMathOperator{\Exp}{\mathbb{E}}
\DeclareMathOperator{\Prob}{\mathbb{P}}
\newcommand{\D}{\mathrm{d}}
\newcommand{\E}{\mathrm{e}}
\newcommand{\Ind}{\mathds{1}}     % indicator function
\newcommand{\Ag}{{\mathcal{A}}}  % non local generator (diffusion)
\newcommand{\sA}{\mathscr{A}}    % generator in Ikeda--Watanabe
\newcommand{\fA}{{\mathfrak{A}}} % used
\newcommand{\Cc}{\mathcal{C}}    % Continuous Functions
\newcommand{\fD}{{\mathfrak{D}}} % extended state space
\newcommand{\cF}{{\mathcal{F}}}  % sigma field 
\newcommand{\cG}{{\mathcal{G}}}  % sigma field
\newcommand{\eom}{{\mathscr{G}}} % set of ergodic occupation measures 
\newcommand{\cH}{{\mathcal{H}}}  % Hamiltonian
\newcommand{\sH}{{\mathscr{H}}}  % Generator renewal
\newcommand{\cI}{{\mathcal{I}}}  % integral operator
\newcommand{\sI}{{\mathscr{I}}}  % classes of jobs
\newcommand{\cK}{{\mathcal{K}}}  % auxiliary set
\newcommand{\sK}{{\mathscr{K}}}  % generator of alternating renewal
\newcommand{\Lg}{\mathcal{L}}    % generator n-system
\newcommand{\cL}{{\mathscr{L}}}  % local part of generator (diffusion)
\newcommand{\cN}{{\mathcal{N}}}  % null sets
\newcommand{\cP}{{\mathcal{P}}}  % Probability measures
\newcommand{\cQ}{{\mathcal{Q}}}  % part of generator
\newcommand{\cR}{{\mathcal{R}}}  % residual process
\newcommand{\rc}{{\mathscr{R}}}  % running cost
\newcommand{\cS}{{\mathcal{S}}}  % simplex
\newcommand{\Lyap}{{\mathcal{V}}}  % Lyapunov function
\newcommand{\sV}{\mathscr{V}}    % Lyapunov function
\newcommand{\fX}{{\mathfrak{X}}}
\newcommand{\cZ}{\mathcal{Z}}    % Action Set
\newcommand{\fZ}{{\mathfrak{Z}}} % admissible Policies
\newcommand{\fZsm}{\mathfrak{Z}_{\mathrm{sm}}} % Markov Policies
\newcommand{\abs}[1]{\lvert#1\rvert}
\newcommand{\norm}[1]{\lVert#1\rVert}
\newcommand{\babs}[1]{\bigl\lvert#1\bigr\rvert}

\newcommand{\df}{\coloneqq}

\DeclareMathOperator*{\diag}{diag}

\newcommand{\order}{{\mathscr{O}}}

\newcommand{\grad}{\nabla}

\usepackage{accents}
\newlength{\dhatheight}

\newcommand{\ttl}{\Large
Optimal scheduling of critically loaded multiclass 
\textsl{GI/M/n+M} \\[3pt] queues in an alternating renewal environment}
\begin{document}

\title[Optimal scheduling of multiclass $GI/M/n+M$ queues with interruptions]{\ttl}
%\date{\today}
%%%%%%%%%%%%%%%%%%%%%%%%%%%%%%%%%%%%%%%%%%%%%%%%%%%%%%%%%%%%%%%%%%%%%%%%%%%%%%%
\author[Ari Arapostathis]
{Ari Arapostathis$^\dag$%\,\protect\orcidicon{0000-0003-2207-357X}
}
\address{$^\dag$ Department of Electrical and Computer Engineering\\
The University of Texas at Austin\\
2501 Speedway, EERC 7.824\\
Austin, TX~~78712}
\email{ari@utexas.edu}

\author[Guodong Pang]
{Guodong Pang$^\ddag$%\,\protect\orcidicon{0000-0001-6468-0429}
}
\author[Yi Zheng]{Yi Zheng$^\ddag$}
\address{$^\ddag$ The Harold and Inge Marcus Department of Industrial and
Manufacturing Engineering,
College of Engineering,
Pennsylvania State University,
University Park, PA 16802}
\email{$\lbrace$gup3,yxz282$\rbrace$@psu.edu}
%%%%%%%%%%%%%%%%%%%%%%%%%%%%%%%%%%%%%%%%%%%%%%%%%%%%%%%%%%%%%%%%%%%%%%%%%%%%%%%
\begin{abstract} 
In this paper, we study optimal control problems 
for multiclass $GI/M/n+M$ queues 
in an alternating renewal (up-down) random environment 
in the Halfin--Whitt regime.
Assuming that the downtimes are asymptotically negligible and 
only the service processes are affected, 
we show that the limits of the diffusion-scaled state processes under
non-anticipative, preemptive,
work-conserving scheduling policies, are controlled jump diffusions 
driven by a compound Poisson jump process. 
We establish the asymptotic optimality of the infinite-horizon discounted
and long-run average (ergodic) problems for the queueing dynamics. 

Since the process counting the number of customers in each class
is not Markov, the usual martingale arguments for
convergence of mean empirical measures cannot be applied.
We surmount this obstacle by demonstrating
the convergence of the generators of an augmented Markovian model  
which incorporates the age processes of the renewal interarrival times and downtimes. 
We also establish long-run average moment bounds 
of the diffusion-scaled queueing processes
under some (modified) priority scheduling policies.
This is accomplished via Foster--Lyapunov equations
for the augmented Markovian model. 
\end{abstract}



%%%%%%%%%%%%%%%%%%%%%%%%%%%%%%%%%%%%%%%%%%%%%%%%%%%%%%%%%%%%%%%%%%%%%%%%%%%%%%%
\maketitle

%%%%%%%%%%%%%%%%%%%%%%%%%%%%%%%%%%%%%%%%%%%%%%%%%%%%%%%%%%%%%%%%%%%%%%%%%%%%%%%
\section{Introduction}

There has been a lot of research activity on scheduling control problems for
queueing networks in the Halfin--Whitt regime.
The discounted problem for multiclass many-server queues was first studied in
 \cite{AMR04}.
See also the  work in \cite{Atar05,AMS09}.
For the ergodic control problem in the case of Markovian queueing networks
see \cite{ABP15,AP18,AP19}.
Scheduling control problems for queueing networks in random environments 
have also attracted much attention recently \cite{BGL14,RMH13,LQA17,ADPZ19}.
It is worth noting that in the study of asymptotic optimality in 
Markov-modulated environments, the scaling parameter depends on the
rate of the underlying Markov process;
see, for example, \cite{BGL14,ADPZ19,JMTW17}.

In this paper we consider queueing networks
operating in alternating renewal (up-down) random environments,
modeling service interruptions,
and with renewal arrivals.  
It is well known that for large-scale service systems, service interruptions
can have a dramatic impact on system performance \cite{PW09B}.
For single class queues and networks in an alternating renewal environment,
limit theorems have been studied in \cite{PW09,LPZ16,LP17,PZ16,PW09B}. 
To the best of our knowledge,
there are no studies on optimal scheduling control for multiclass many-server queues
in alternating renewal environments, or even ergodic control
in the Halfin--Whitt regime with arrivals that are renewal processes.

Specifically, we consider multiclass ($d$ classes) $GI/M/n + M$ queues 
with service interruptions in the Halfin--Whitt regime, 
where the arrival rate in each class and the number of servers in the pool are large, 
with a scaling parameter $n$, and
the service interruptions are asymptotically negligible of order $n^{\nicefrac{-1}{2}}$.
The service interruption is modeled as an alternating renewal process
constructed by regenerative `up' and `down' cycles.
In the `down' state, all servers stop functioning,
and new customers arrive, which may abandon the queue.
In the `up' state, the queueing system functions normally.
We assume that at least one class of customers has
a strictly positive abandonment rate.
The scheduling policy determines the allocation of servers to different classes
of customers.
We approximate the scheduling problem via the corresponding control problem of
the limiting jump diffusion in the heavy-traffic regime,
for which a sharp characterization of optimal Markov controls is available
\cite{APZ19a}, and use this to exhibit matching upper and lower bounds on the
optimal scheduling performance for the queueing dynamics.

In \cref{T3.1}, we establish a functional central limit theorem (FCLT) 
for the $d$-dimensional diffusion-scaled state processes
under work-conserving scheduling policies.
The limiting controlled processes are jump diffusions 
with piecewise linear drift and compound Poisson jumps. 
The proof of weak convergence relies on the construction of
a modified diffusion-scaled state process,
where we add the cumulative downtime to a diffusion-scaled state process 
without interruptions.
We show that the modified and original diffusion-scaled state processes
have the same weak limits,
which are governed by the jump diffusions described above.

The discounted and ergodic control problems for a large class of
jump diffusions arising from queueing networks in the Halfin--Whitt
regime have been studied in \cite{APZ19a}, and these results are essential
for establishing asymptotic optimality in the present paper.
In \cref{T3.2},
we show that the optimal value functions of the discounted problem for the
diffusion-scaled processes
converge to the corresponding function for the limiting jump diffusion.
The proof of asymptotic optimality follows the approach in \cite{AMR04},
which deals with
the discounted problem for multiclass $GI/M/n + M$ queues. 
An essential part of this proof involves moment bounds
for the diffusion-scaled state process, and
 the cumulative downtime process.

Asymptotic optimality for the ergodic control problem is more challenging.
The result is stated in \cref{T3.3}.
Here, long-run average moment bounds for the diffusion-scaled
state processes play a crucial role (see \cref{T4.1}). 
Typically, such bounds are obtained in the literature
via Foster-Lyapunov inequalities \cite{Atar11,AP18,ABP15,AP19,ADPZ19}.
However, since the process counting the number of customers in each class,
referred to as the queueing process, or state process, is not Markov,
we first construct a
sequence of auxiliary diffusion-scaled processes 
by adding the scaled residual time process of the alternating renewal process 
in the `down' state to the original process,
taking advantage of the fact that the long-run average moments of
the scaled residual time process are negligible as
the scaling parameter $n$ tends to infinity. 
We then consider the joint Markov process comprised of the
auxiliary diffusion-scaled state process and
the age processes of renewal arrival and alternating renewal processes,
and construct Foster--Lyapunov functions, which bear a resemblance to
the Lyapunov functions in \cite{Takis-99}.
In this, we assume that the mean residual life functions
are bounded, and use the criterion 
in \cite[Theorem~4.2]{MT-III} to show that the joint Markov processes
are positive Harris recurrent for all large enough $n$
under some (modified) priority scheduling policy. 
We apply a two-step scheduling: first,
the servers are allocated to the classes of customers with zero abandonment rate 
in such a manner that the servers used for each class do not exceed a certain proportion
dictated by the traffic intensity;   
second, a static priority rule is applied to allocate the remaining servers.
We show that the long-run average moments of the auxiliary
diffusion-scaled state processes are bounded under this scheduling policy.
We then establish a moment estimate for the difference between the auxiliary and original
diffusion-scaled processes, and proceed to show that the analogous
moment bounds
hold for the original diffusion-scaled processes.

To prove asymptotic optimality for the ergodic control problem,
we establish lower and upper bounds for the limits of the value functions 
(see equations \cref{ET3.3A,ET3.3B}).
For the proof of the lower bound, we show that the sequence of mean empirical measures of
the diffusion-scaled state processes is tight (see \cref{L5.2}), and 
any limit of mean empirical measures is an ergodic occupation measure
for the limiting jump diffusion.
This is analogous to the technique used in \cite{ABP15,AP18,AP19,ADPZ19}.
However, characterizing the limits of mean empirical measures (see \cref{T5.3})
is quite challenging here. 
Since we consider the diffusion-scaled processes with renewal arrivals
in an alternating renewal environment,
the martingale arguments in the above papers cannot be applied here.
Instead, we develop a new approach. 
Following the technique of the proof of ergodicity under
the specific scheduling policy described in the preceding paragraph,
we consider the generator of the joint Markov process of the
auxiliary diffusion-scaled state process, which incorporates
the residual time process, 
and the associated age processes of the renewal arrivals and 
the alternating renewal environment. 
We construct suitable test functions (see \cref{E-phi})
which involve the coefficients of variation of interarrival times, 
and proceed to show the convergence of generators.

For the proof of the upper bound, we adopt the spatial truncation technique
developed in \cite{ABP15}, which is also used in \cite{AP18,AP19,ADPZ19},
and is extended to jump diffusions in \cite{APZ19a}.
This involves a concatenated scheduling policy.
We first construct a continuous precise $\epsilon$-optimal control for the
ergodic control problem for the limiting jump diffusion (see \cref{T5.2}). 
Then, inside a compact set, 
we map this control to a scheduling policy for the diffusion-scaled process.
On the complement of this set, we apply the (modified) priority scheduling policy.
We show that the long run average moments of the diffusion-scaled state process
are bounded
under this concatenated scheduling policy (see \cref{C4.1}),
and the limit of mean empirical measures is the ergodic occupation measure 
of the limiting jump diffusion governed by the $\epsilon$-optimal control
(see \cref{L5.3}).
Here, the techniques used in establishing the long-run average moment bounds
under the (modified) priority scheduling policy, and the
convergence of mean empirical measures, play an important role.

%%%%%%%%%%%%%%%%%%%%%%%%%%%%%%%%%%%%%%%%%%%%%%%%%%%%%%%%%%%%%%%%%%%%%%%%%%%%%%%%
\subsection{Organization of the paper}
The notation used in the paper is summarized in the next subsection. In \cref{S2},
we describe the model of multiclass many-server queues with service interruptions.
In \cref{S3}, we define the diffusion-scaled processes and associated
control problems, and
state the main results on weak convergence and asymptotic optimality.
In \cref{S4}, we summarize the ergodic properties of the limiting controlled
jump diffusion, and state the results concerning
long-run average moment bounds for the diffusion-scaled processes.
The proofs of \cref{T3.2,T3.3} are given in \cref{S5}. \cref{AppA} is devoted
to the proofs of \cref{L3.1,T3.1}. \cref{AppB} contains the proofs of \cref{L4.1,L5.2}.

%%%%%%%%%%%%%%%%%%%%%%%%%%%%%%%%%%%%%%%%%%%%%%%%%%%%%%%%%%%%%%%%%%%%%%%%%%%%%%%%
\subsection{Notation}
We let $\abs{\,\cdot\,}$ and $\langle\,\cdot\,,\cdot\,\rangle$ denote
the standard Euclidean norm and the inner product in $\RR^{d}$,
respectively.
For $x\in\Rd$, we let $\norm{x} \df \sum_i\abs{x_i}$,
and $x'$ denote the transpose of $x$.
The symbols $\RR_+$, $\ZZ_+$, $\NN$, denote the
set of nonnegative real numbers, nonnegative integers, and
the set of natural numbers, respectively.
The indicator function of a set $A\in\Rd$ is denoted by  $\Ind_A$.
%The set of $d$-dimensional vectors of nonnegative integers 
%is denoted by $\mathbb{Z}^{d}_{+}$.
Given $a,b\in\RR$, the minimum (maximum) is denoted by $a\wedge b$ 
($a\vee b$), respectively,
$\lfloor a\rfloor$ denotes the integer part of $a$,
and $a^\pm \df (\pm a)\vee 0$.
The complement and closure
of a set $A\subset\Rd$ are denoted
by $A^{c}$ and $\Bar{A}$, respectively.
We use the notation 
$e_{i}$ to denote the vector with $i$-th entry equal to $1$
and all other entries equal to $0$.
%We also let $e\df (1,\dotsc,1)\transp$.
We let $B_{r}$ denote the open ball of radius $r$ in $\RR^{d}$, 
centered at the origin.
For a process $\{X_t\}_{t\ge0}$, $\uptau(A)$ denotes the first exit time from
the set $A\subset\RR^{d}$, 
defined by $\uptau(A) \df \inf\;\{t>0\;\colon\, X_{t}\not\in A\}$,
and we let $\uptau_{r}\df \uptau(B_{r})$.

%The term \emph{domain} in $\RR^{d}$
%refers to a nonempty, connected open subset of the Euclidean space $\RR^{d}$. 
For a domain $D\subset\RR^{d}$,
the space $\Cc^{k}(D)$ ($\Cc^{\infty}(D)$), $k\ge 0$,
stands for the class of all real-valued functions on $D$ whose partial
derivatives up to order $k$ (of any order) exist and are continuous.
$\Cc^{k,r}(D)$ stands for the set of functions that are
$k$-times continuously differentiable and whose $k^{\mathrm{th}}$ derivatives are locally
H\"{o}lder continuous with exponent $r$.
We let $\Cc^k_c(D)$ denote the space of functions in $\Cc^k(D)$ with compact support,
and $\Cc_b^k$ the set of functions in $\Cc^k(D)$ whose partial derivatives
up to order $k$ are bounded.
For a nonnegative function $g\in\Cc(\RR^{d})$, 
$\order(g)$ denotes the space of functions
$f\in\Cc(\RR^{d})$ satisfying
$\sup_{x\in\RR^{d}}\;\frac{\abs{f(x)}}{1+g(x)}<\infty$.
By a slight abuse of notation, $\order(g)$ also denotes
a generic member of these spaces.

For $k\in\NN$, we 
let $\DD^k \df \DD(\RR_+,\RR^k)$ denote the space of 
$\RR^k$-valued c\'adl\'ag functions on $\RR_+$.
When $k=1$, we write $\DD$ for $\DD^k$.
Given a Polish space $E$, by $\cP(E)$ we denote the space
of probability measures on $E$, endowed with the Prokhorov metric.

%%%%%%%%%%%%%%%%%%%%%%%%%%%%%%%%%%%%%%%%%%%%%%%%%%%%%%%%%%%%%%%%%%%%%%%%%%%%%%%%
\section{Multiclass \texorpdfstring{$GI/M/N+M$}{} queues with service interruptions}
\label{S2}

%%%%%%%%%%%%%%%%%%%%%%%%%%%%%%%%%%%%%%%%%%%%%%%%%%%%%%%%%%%%%%%%%%%%%%%%%%%%%%%%
\subsection{The model and assumptions}%\label{S2.1}

We consider a sequence of $GI/M/n + M$ queueing models 
with $d$ classes of customers.
Let $\sI\df\{1,\dotsc,d\}$.
For the $n^{\mathrm{th}}$ system, let $\{A_i^n(t)\}_{t\ge0}$ denote the
arrival process of class-$i$ customers.
We assume that the arrivals are mutually independent renewal
processes defined as follows. 
Let $\{{G}_{i,j}\colon j\in\NN\}$, $i\in\sI$,
be an i.i.d.\ sequence of strictly positive random variables with mean
$\Exp [{G}_{i}] = 1$ and finite (squared) coefficient of variation
$c_{a,i}^2\df \nicefrac{\mathrm{Var}({G}_{i})}{(\Exp [{G}_i])^2}$,
where $G_{i} \equiv G_{i,1}$.
Then, we define
\begin{equation}\label{E-arrival}
A_i^n(t) \,\df\, \max\,\Biggl\{m\ge 0 \colon \sum_{j=1}^{m} G_{i,j}
\le \lambda^n_i t\Biggr\}\,,\quad t\ge0\,,\ i\in\sI\,,
\end{equation}
where $\lambda^n_i >0$ denotes the arrival rate.
For each $n\in\NN$, 
the service and patience times of the class-$i$ customers are exponentially
distributed with parameters $\mu^n_i$ and $\gamma^n_i$, respectively.

We adopt the following standard assumption on the parameters
(see \cite{AMR04,ABP15,PW09}).

%%%%%%%%%%%%%%%%%%%%%%%%%%%%%%%%%%%%%%%%%%%%%%%%%%%%%%%%%%%%%%%%%%%%%%%%%%%%%%%%
\begin{assumption}\label{A2.1}
(\emph{The Halfin--Whitt regime})
The parameters satisfy the following limits
for each $i\in\sI$ as $n\rightarrow\infty$:
\begin{equation*}
\begin{aligned}
& n^{-1}\lambda^n_i \;\rightarrow\; \lambda_i \,>\, 0\,, \quad
\mu_i^n \;\rightarrow\; \mu_i\, > \,0\,, \quad
\gamma_i^n \;\rightarrow\; \gamma_i \,\ge\, 0\,,\\
& 
n^{-\nicefrac{1}{2}}(\lambda_i^n-n\lambda_i) 
\;\rightarrow\; \Hat{\lambda}_i \,,
\quad n^{\nicefrac{1}{2}}(\mu_i^n-\mu_i) 
\;\rightarrow\; \Hat{\mu}_i\,,  \\
&
\frac{\lambda_i^n}{n\mu_i^n} 
\;\rightarrow\; \rho_i \,\df\, \frac{\lambda_i}{\mu_i}<1\,,
\quad \sum_{i=1}^{d}\rho_i\,=\, 1\,.
\end{aligned}
\end{equation*}
\end{assumption}

We assume that
$\inf_{n\in\NN}\gamma^n_d > 0$.
\cref{A2.1}, which is also known as the Quality-and-Efficiency-Driven regime, 
implies that the system is critically loaded and 
\begin{equation*}
\rho^n \;\rightarrow\; \Hat{\rho} \,\df\, 
\sum_{i=1}^{d} \frac{\rho_i\Hat{\mu}_i - \Hat{\lambda}_i}{\mu_i}
\in\RR\,, \quad \text{where}
\quad \rho^n\,\df\,\sqrt{n}\biggl(1 - \sum_{i=1}^{d}\frac{\lambda_i^n}{n\mu^n_i}\biggr)\,.
\end{equation*}

All queues are in the same up-down alternating renewal random environment.
Waiting customers may abandon at any time.
In the `up' state,
the system functions normally, and in the `down' state
 all servers stop, while
customers keep joining the queues and any jobs that have started service will wait
for the system to resume.   
For this reason, we also refer to this model
 as multiclass queues with service interruptions.  
Let  $\bigl\{(u_k^n,d_k^n)\colon k\in\NN\bigr\}$
be a sequence of i.i.d.\ positive random vectors denoting the up-down cycles, 
and define the \emph{counting process of downtimes} by
\begin{equation}\label{ES2.1A}
N^n(t)\,\df\, \max\,\bigl\{k\ge0\colon T_k^n\le t\bigr\}\,,\quad
\text{with\ \ } T_k^n\,\df\,\sum_{i=1}^{k}(u_i^n + d_i^n)\,,\ \ k\in\NN\,,
\end{equation}
 and $T_0^n\equiv0$. At time $0$, the system is in the `up' state.

%%%%%%%%%%%%%%%%%%%%%%%%%%%%%%%%%%%%%%%%%%%%%%%%%%%%%%%%%%%%%%%%%%%%%%%%%%%%%%%%
\begin{assumption}\label{A2.2}
For each $n$ and $k$ in $\NN$,  $u^n_k$ and $d^n_k$ are independent,
$u^n_k$ is exponentially distributed with parameter $\beta_u^{n}$,
which converges to $\beta>0$ as $n\rightarrow\infty$.
We assume that $d^n_1 = \frac{1}{\vartheta^n}d_1$, 
with $d_1$ some nonnegative random variable satisfying $\Exp[d_1] = 1$,
and $\frac{\vartheta^n}{\sqrt{n}} \rightarrow \vartheta > 0$ as 
$n\rightarrow\infty$.
\end{assumption}

For $k\in\NN$, we
let $(\DD^k,M_1)$ and $(\DD^k, J_1)$
denote the space $\DD^k$ endowed with
the Skorokhod $M_1$ and $J_1$ topologies, respectively 
(see, for example, \cite{Patrick-99,WW-02}).
\cref{A2.2} implies that the service interruptions are asymptotically negligible, and
\begin{equation*}
N^n \;\Rightarrow\; N \;\quad\; \text{in} \quad (\DD,J_1) 
\quad \text{as\ }  n\to\infty\,, 
\end{equation*} 
where the limiting process $N$ is a Poisson process with rate $\beta$. 
Define the \emph{server availability process}
$\Psi^n \df \{\Psi^n(t)\colon t\ge0\}$ by
\begin{equation}\label{ES2.1B}
\Psi^n(t) \,=\, \left\{\begin{aligned} 
&1, \quad T^n_k \le t < T^n_k + u^n_{k+1} \,, \\
&0, \quad T^n_k + u^n_{k+1}\le t < T^n_{k+1} \,,
\end{aligned}\right.
\end{equation}
for $k\in\NN$.
We also define the cumulative up-time process
$C^n_{\mathsf{u}} = \{C^n_{\mathsf{u}}(t)\}_{t\ge0}$ by
$C^n_{\mathsf{u}}(t) \df \int_0^{t}\Psi^n(s)\,\D{s}$,
and the cumulative down-time process by
$C^n_{\mathsf{d}}(t) \df t - C^n_{\mathsf{u}}(t)$.
Let $F^{d_1}$ denote the distribution function of $d_1$.
By Lemma~2.2 in \cite{PW09}, we have
\begin{equation}\label{ES2.1C}
\sqrt{n}C^n_{\mathsf{d}} \,\Rightarrow\, {L} \quad \text{in\ }
 (\DD,M_1) \quad \text{as\ }  n\rightarrow\infty\,,
\end{equation}
where $\{{L}_t\}_{t\ge 0}$ is a compound Poisson process
with intensity $\Pi_L(\D{x})\D{t} = \beta\,F^{d_1}(\vartheta\D{x})\D{t}$,
where $\beta$ is given in \cref{A2.2}.

For the $n^{\mathrm{th}}$ system, we
denote the processes counting the total number of customers, those in queue, 
and those in service, by $X^n=(X^n_1,\dotsc,X^n_d)'$, 
$Q^n=(Q^n_1,\dotsc,Q^n_d)'$, and $Z^n=(Z^n_1,\dotsc,Z^n_d)'$, respectively. 
These processes satisfy the following constraints: 
\begin{equation}\label{ES2.1D}
X_i^n(t) \,=\, Q_i^n(t) + Z_i^n(t) \,, \quad
Q_i^n(t) \,\ge\, 0\,, \quad Z_i^n(t) 
\,\ge\, 0\,, \quad \text{and} \quad \langle{e},Z^n(t)\rangle \,\le\, n
\end{equation}
for each $t\ge 0$ and $i\in\sI$. 
We let 
\begin{equation}\label{ES2.1E}
\begin{aligned}
{S}_i^n(t,r) &\,\df\, 
S^n_{*,i}\biggl(\mu^n_i\int_0^{t}Z_i^n(s)\Psi^n(s)\,\D{s} + \mu^n_ir\biggr)\,, \\ 
{R}_i^n(t,r) &\,\df\, 
R^n_{*,i}\biggl(\gamma^n_i\int_0^{t}Q_i^n(s)\,\D{s} + \gamma^n_ir\biggr)\,,
\end{aligned}
\end{equation}
for $i\in\sI$, $t\ge0$, and $r\ge0$,
where $\{S^n_{*,i},R^n_{*,i} \colon i\in\sI,n\in\NN\}$ 
are Poisson processes with rate one.
We assume that for each $n\in\NN$, 
$\bigl\{X^n_i(0),A_i^n,S^n_{*,i},R^n_{*,i} \colon i\in\sI\bigr\}$
are mutually independent.
These processes are governed by the equation
\begin{equation}\label{E-dynamic}
X^n_i(t) \,=\, X_i^{n}(0) + A^n_i(t) - {S}_i^n(t)  -  {R}_i^n(t)
\end{equation}
for each $t\ge 0$, $n\in\NN$, and $i\in\sI$,
where $ {S}_i^n(t) \df {S}_i^n(t,0)$ and
${R}_i^n(t) \df {R}_i^n(t,0).$

%%%%%%%%%%%%%%%%%%%%%%%%%%%%%%%%%%%%%%%%%%%%%%%%%%%%%%%%%%%%%%%%%%%%%%%%%%%%%%%%
\subsection{Scheduling policies}%\label{S2.2}
A scheduling policy is identified with a $\ZZ_+^d$-valued
stochastic process $Z^n$
with c\'{a}dl\'{a}g sample paths, 
which satisfies \cref{ES2.1D}.
Let 
\begin{equation}\label{ES2.2A}
\Tilde{\tau}_i^n(t) \,\df\, \inf\{r \ge t\colon A_i^n(r) - A_i^n(r-)>0\}\,, \quad
\text{and} \quad 
\breve{\tau}^n(t) \,\df\, \inf\{r\ge t\colon \Psi^n(r) = 1\}\,, 
\end{equation}
for $i\in\sI$.
Recall the definitions of $C^n_{\mathsf{d}}$ in \cref{ES2.1C}, and
${S}^n$ and ${R}^n$ in \cref{ES2.1E}. 
Define the $\sigma$-fields 
\begin{equation}\label{E-filtration}
\begin{aligned}
\cF^n_t &\,\df\, 
\sigma\bigl\{X^n(0),A^n_i(t),{S}_i^n(s),
{R}_i^n(s),X^n_i(s) ,Z^n_i(s), \Psi^n(s), N^n(s)
\colon i\in\sI,0\le s\le t\bigr\}\vee\cN \,,\\[5pt]
\cG^n_t &\,\df\, \sigma\bigl\{ A_i^n(\Tilde{\tau}_i^n(t) + r)
- A_i^n\bigl(\Tilde{\tau}_i^n(t)\bigr),
{S}_i^n(\breve{\tau}^n(t),r) - S_i^n\bigl(\breve{\tau}^n(t)\bigr)\,, \\
&\mspace{100mu} {R}_i^n(\breve{\tau}^n(t),r) - {R}_i^n\bigl(\breve{\tau}^n(t)\bigr), 
C^n_{\mathsf{d}}(\breve{\tau}^n(t)+r) - C^n_{\mathsf{d}}\bigl(\breve{\tau}^n(t)\bigr)
\colon i\in\sI, r\ge 0\bigr\}\vee\cN \,,
\end{aligned}
\end{equation}
for $t\ge0$,
where 
$\cN$ is the collection of all $\Prob$-null sets.
We say that a scheduling policy $Z^n$ is \emph{non-anticipative} if
\begin{itemize}
\item[(i)]
$Z^n(t)$ is adapted to $\cF^n_t$,
\item[(ii)]
$\cF^n_t$ and $\cG_t^n$ are independent at each time $t\ge 0$,
\item[(iii)]
for each $i\in\sI$, and $t\ge 0$, the process 
${S}^n_i(\breve{\tau}^n(t),\cdot) - {S}^n_i(\breve{\tau}^n(t))$
agrees in law with $S^n_{*,i}(\mu^n_i\cdot)$, 
and the process ${R}^n_i(\breve{\tau}^n(t),\cdot) - {R}^n_i(\breve{\tau}^n(t))$
agrees in law with $R^n_{*,i}(\gamma^n_i\cdot)$.
\end{itemize}
Let $\tau^{n}_{i,k}$ denote the $k^{\mathrm{th}}$ jump time of 
$A^n_i - {S}_i^n - {R}_i^n$, for each $n\in\NN$ and $i\in\sI$.  
\Cref{E-dynamic} implies that $X_i^n(t) = X_i^n(0)$ for $0\le t \le \tau^{n}_{i,1}$, 
$X_i^n(t) = X_i^n(0) + \epsilon_1$ 
for $\tau^{n}_{i,1} \le t \le \tau^{n}_{i,2}$ and so forth, 
where $\epsilon_k$ denotes the jump size which takes values in a bounded set.
Note that 
the integrals in \cref{ES2.1E} are finite
by the definition of $\Psi^n$ in \cref{ES2.1B} and \cref{ES2.1D}.
Thus, given any non-anticipative scheduling policy $Z^n$,
and initial condition $X^n(0)$, 
there exists a unique solution to \cref{E-dynamic}.

For $x\in\ZZ^d_+$, we define the action set $\cZ^n(x)$ by
\begin{equation*}
\cZ^n(x)\,\df\,  \bigl\{z\in \ZZ_+^d\,\colon z\le x\,,\;
\langle e,z\rangle = \langle e,x\rangle\wedge n
\bigr\}\,.
\end{equation*}
A scheduling policy $Z^n$ is called \emph{admissible} 
if $Z^n(t)$ takes values in $\cZ^n\bigl(X^n(t)\bigr)$ at each $t$,
and is non-anticipative.
The set of admissible scheduling policies is denoted by $\fZ^n$.
Note that an admissible policy allows preemption, that is,
a server can interrupt service of a customer
at any time to serve some other class of customers.
In summary, given an admissible scheduling policy $Z^n$,
the process $X^n$ in \cref{E-dynamic} is well defined,
and we say that $X^n$ is governed by $Z^n$.

Next, we describe a well-known equivalent parameterization of
the set of admissible policies.
Let
\begin{equation*}
\cS \,\df\, \{u\in\Rd_{+}\,\colon \langle e,u \rangle = 1\}\,.
\end{equation*}
We also define
\begin{equation*}
\cS^n(x) \,\df\, \Bigl\{v\in\ZZ^d_{+}\,\colon v=\frac{y}{\langle e,x\rangle -n}\in\cS\,,
\; y\le x\,,\; y\in\ZZ^d_+\Bigr\}\,,\quad \text{if\ \ }\langle e,x\rangle >n \,,
\end{equation*}
and $\cS^n(x)=\{e_d\}$, if $\langle e,x\rangle \le n$.
Let $\Uadm^n$ denote the class of processes
$\{U^n(t)\}_{t\ge0}$ which are non-anticipative, in the sense of the
definition given above, and $U^n(t)$ takes values in $\cS^n\bigl(X^n(t)\bigr)$.
Then, each $U^n\in\Uadm^n$ determines a policy $Z^n\in\fZ^n$ via
\begin{equation*}
{Z}^n(t) \,=\, {X}^n(t) - {Q}^n(t)\,,\quad\text{with\ \ }
{Q}^n(t) \,=\, \bigl(\bigl\langle e,{X}^n(t) \bigr\rangle - n\bigr)^{+}{U}^n(t)\,.
\end{equation*}
This map is invertible, and its inverse is given by
\begin{equation*}
{U}^n(t)\,\df\,\begin{cases}
\frac{X^n(t) - Z^n(t)}{\langle e,X^n(t) \rangle - n} 
\quad &\text{for } \langle e,X^n(t) \rangle > n\,, \\[3pt]
e_d \quad &\text{for } \langle e,X^n(t) \rangle \le n\,.
\end{cases}
\end{equation*}
Therefore, as far as control problems are concerned, we can use
policies in $\Uadm^n$ or $\fZ^n$ interchangeably.

Next, we augment the state space, and
define the class of stationary Markov scheduling policies.
Recall the definitions of $A^n$, $N^n$, and $\Psi^n$
in \cref{E-arrival,ES2.1A,ES2.1B}, respectively.

%%%%%%%%%%%%%%%%%%%%%%%%%%%%%%%%%%%%%%%%%%%%%%%%%%%%%%%%%%%%%%%%%%%%%%%%%%%%%%%%
\begin{definition}
Let $H^n_i(t)$ denote the age process for the class-$i$ customers, that is,
\begin{equation}\label{ES2.2B}
H^n_i(t) \,\df\, t - \frac{1}{\lambda^n_i}\sum^{A^n_i(t)}_{j=1}G_{i,j}\,,
\qquad  t \ge 0\,, \qquad i\in\sI\,,
\end{equation}
and define the age process $K^n$ for the alternating renewal process in the `down'
state by
\begin{equation}\label{ES2.2C}
K^n(t) \,\df\, \Biggl( t - \sum_{k=1}^{N^n(t)}(u_k^n + d_k^n) 
- u_{N^n(t)+1}^n \Biggr)^+\,, \qquad t\ge0\,.
\end{equation}
Then, $(A^n_i,H^n_i)$, $i\in\sI$, and $(\Psi^n,K^n)$ 
are strong Markov processes (see, e.g., \cite{Dai-95}). 
We say that a scheduling policy $Z^n\in\fZ^n$ is (stationary) Markov 
if
\begin{equation*}
Z^n(t) \,=\, z^n\bigl(X^n(t),H^n(t),\Psi^n(t),K^n(t)\bigr)
\end{equation*} 
for some $z^n\colon \ZZ^d_+\times\RR^d_+\times\{0,1\}\times\RR_+ 
\to \ZZ^d_+$,
and we let $\fZsm^n$ denote the class of these policies.
Under a policy $Z^n\in\fZsm^n$, the process
$({X}^n,H^n,\Psi^n,K^n)$ is Markov
with state space
\begin{equation*}
\bigl\{(x,h,\psi,k)\in \ZZ^d_+\times\RR^d_+\times\{0,1\}\times\RR_+ 
\colon k \equiv 0\text{ if } \psi=1 \bigr\}\,.
\end{equation*} 
Abusing the notation, when $z^n$ depends only
on its first argument, we simply write $Z^n(t) = z^n\bigl(X^n(t)\bigr)$.
\end{definition} 

%%%%%%%%%%%%%%%%%%%%%%%%%%%%%%%%%%%%%%%%%%%%%%%%%%%%%%%%%%%%%%%%%%%%%%%%%%%%%%%%
\section{Diffusion-scaled processes and control problems}\label{S3}

Let $\Hat{X}^n$, $\Hat{Q}^n$, and $\Hat{Z}^n$
denote the diffusion-scaled processes defined by 
\begin{equation*}
\Hat{X}^n_i(t) \,\df\, n^{-\nicefrac{1}{2}}(X_i^n(t) - \rho_i n)\,,
\quad \Hat{Q}^n_i(t) \,\df\, n^{-\nicefrac{1}{2}}Q^n_i(t) \,,
\quad \Hat{Z}^n_i(t) \,\df\, n^{-\nicefrac{1}{2}}(Z^n_i(t)-\rho_{i} n)\,,
\end{equation*}
respectively,
for $t\ge0$ and $i\in\sI$.
It follows by \cref{E-dynamic} that the process $\Hat{X}^n_i$ takes the form
\begin{equation}\label{E-HatX}
\begin{aligned}
\Hat{X}_i^n(t) &\,=\, \;\Hat{X}_i^n(0) + \ell^n_it + \Hat{A}^n_i(t)
- \Hat{S}^n_i(t) - \Hat{R}^n_i(t) \\
&\mspace{50mu}- \mu^n_i\int_0^{t}\Hat{Z}^n_i(s)\Psi^n(s)\,\D{s}  
- \gamma^n_i\int_0^{t}\Hat{Q}^n_i(s)\,\D{s}
+ \Hat{L}^n_i(t)\,,\quad t\ge 0\,,
\end{aligned}
\end{equation}
where  $\ell^n_i \df n^{-\nicefrac{1}{2}}(\lambda^n_i - n\mu^n_i\rho_i)$,
\begin{equation*}
\begin{gathered}
\Hat{A}_i^n(t) \,\df\, 
n^{-\nicefrac{1}{2}}\bigl(A^n_i(t) - \lambda^n_it\bigr)\,, \qquad
\Hat{S}_{i}^{n}(t) \,\df\, n^{-\nicefrac{1}{2}}\biggl({S}_i^n(t) 
- \mu^n_i\int_0^{t}Z^n_i(s)\Psi^n(s)\,\D{s}\biggr)\,, \\
\Hat{R}^{n}_i(t) \,\df\, n^{-\nicefrac{1}{2}}\biggl({R}_i^n(t) 
- \gamma^n_i\int_0^{t}Q^n_i(s)\,\D{s}\biggr)\,, \quad \text{and\ \ }
\Hat{L}^n_i(t) \,\df\, \sqrt{n}\mu^n_i\rho_i C^n_{\mathsf{d}}(t)\,.
\end{gathered}
\end{equation*}
Let $\Hat{W}^n$ and $\Hat{Y}^n$, $n\in\NN$, be $d$-dimensional processes defined by
\begin{equation}\label{ES3.1A}
\Hat{W}_i^n \,\df\, \Hat{A}^n_i - \Hat{S}^n_{i} - \Hat{R}^n_{i}
\qquad \text{for }i\in\sI \,,
\end{equation}
and 
\begin{equation*}
\Hat{Y}^n_i(t) \,\df\, \ell^n_it - \mu^n_i\int_0^{t}\Hat{Z}^n_i(s)\Psi^n(s)\,\D{s} 
- \gamma^n_i\int_0^{t}\Hat{Q}^n_i(s)\,\D{s} \qquad \text{for }i\in\sI\,,\ t\ge0\,,
\end{equation*}
respectively. Then, $\Hat{X}^n_i$ in \cref{E-HatX} has
the representation
\begin{equation*}
\Hat{X}^n_i(t) \,=\, \Hat{X}_i^n(0) + \Hat{Y}^n_i(t) + \Hat{W}^n_i(t) + \Hat{L}^n_i(t)\,.
\end{equation*}
The initial condition $\Hat{X}^n(0)$, $n\in\NN$, is assumed to be deterministic
throughout the paper.

%%%%%%%%%%%%%%%%%%%%%%%%%%%%%%%%%%%%%%%%%%%%%%%%%%%%%%%%%%%%%%%%%%%%%%%%%%%%%%%%
\subsection{The limiting controlled diffusion with compound Poisson jumps}
In \cref{L3.1,T3.1} which follow,
products or powers
of the spaces $(\DD^d,J_1)$ and $(\DD^d,M_1)$ are viewed as metric
spaces endowed with the maximum metric.
The proofs of these results are given in \cref{AppA}.

%%%%%%%%%%%%%%%%%%%%%%%%%%%%%%%%%%%%%%%%%%%%%%%%%%%%%%%%%%%%%%%%%%%%%%%%%%%%%%%%
\begin{lemma}\label{L3.1}
Suppose that \cref{A2.1,A2.2} hold, and that $\{\Hat{X}^n(0)\colon n\in\NN\}$
is  bounded.
Then, under any sequence of 
${U}^n \in {\Uadm}^n$,
we have
\begin{equation*}
(n^{-1}{Q}^n, n^{-1}{Z}^n) \;\Rightarrow\; (\mathfrak{e}_0,\mathfrak{e}_{\rho}) 
\quad \text{in} \quad (\DD^d, M_1)^2\,,
\end{equation*}
where $\mathfrak{e}_0(t) \equiv (0,\dotsc,0)'$ for all $t\ge0$,
and $\mathfrak{e}_\rho(t) \equiv  (\rho_1,\dotsc,\rho_d)'$.
\end{lemma}

%%%%%%%%%%%%%%%%%%%%%%%%%%%%%%%%%%%%%%%%%%%%%%%%%%%%%%%%%%%%%%%%%%%%%%%%%%%%%%%
\begin{theorem} \label{T3.1}
Grant the assumptions in \cref{L3.1}. Then, the following hold.
\begin{itemize}
\item[\textup{(i)}] As $n\rightarrow\infty$,
\begin{equation*}
(\Hat{W}^n,\Hat{L}^n) \;\Rightarrow\; (\Sigma W,\lambda L) 
\quad \text{in}\quad (\DD^d,J_1) \times (\DD^d,M_1)\,,
\end{equation*}
where the matrix $\Sigma$ is given by $\Sigma \df 
\diag\Bigl(\sqrt{\lambda_1 (1+ c^2_{a,1})},\dotsc,
\sqrt{\lambda_d (1+ c^2_{a,d})}\Bigr)$,
${W}$ is a $d$-dimensional standard Wiener process,
$\lambda \df (\lambda_1,\dotsc,\lambda_d)'$,
and $\{{L}_t\}_{t\ge0}$ is the one-dimensional L\'evy process in \eqref{ES2.1C}, 
and is independent of ${W}$. 

\item[\textup{(ii)}]
The sequence $(\Hat{X}^n,\Hat{Y}^n,\Hat{W}^n,\Hat{L}^n)$ is tight in 
$(\DD^d,M_1)\times(\DD^d,J_1)^2\times(\DD^d,M_1)$.

\item[\textup{(iii)}]
Provided ${U}^n$ is tight in $(\DD^d,J_1)$,
any limit $X$ of $\Hat{X}^n$ is a strong solution to the stochastic differential equation
\begin{equation}\label{ET3.1A}
\D{X_t} \,=\, 
b(X_t,U_t)\,\D{t} + \Sigma\,\D{W_t} + \lambda\,\D{L_t}\,,
\end{equation}
with initial condition ${X}_0=x \in \Rd$, 
where $U$ is a limit of ${U}^n$,
and $b(x,u)\colon \Rd\times\cS\to\Rd$ takes the form
\begin{equation}\label{ET3.1B}
b(x,u) \,=\, \ell - M(x-\langle e,x \rangle^+u)-\langle e,x \rangle^+\varGamma u\,,
\end{equation}
with $\ell \df (\ell_1,\dotsc,\ell_d)'$, 
$M \df \diag(\mu_1,\dotsc,\mu_d)$, and $\varGamma\df \diag(\gamma_1,\dotsc,\gamma_d)$.
Moreover, any such limit $U$ is non-anticipative, that is, for $s<t$,  
$(W_t - W_s, L_t - L_s)$ is independent of
\begin{equation*}
\cF_s \,\df\, \text{the completion of } \sigma\{X_0, U_r, W_r,L_r\colon r\le s\}\,. 
\end{equation*}
\end{itemize}
\end{theorem}

Throughout the paper,
the time variable appears as a subscript in the
processes governing the limiting controlled diffusion 
in order to distinguish them from the processes associated
with the $n^{\mathrm{th}}$ system.

%%%%%%%%%%%%%%%%%%%%%%%%%%%%%%%%%%%%%%%%%%%%%%%%%%%%%%%%%%%%%%%%%%%%%%%%%%%%%%%%
\subsection{The control problems}
Define $\widetilde{\rc}\colon \RR^d_+\to\RR_{+}$ by
\begin{equation}\label{ES3.2A}
\widetilde{\rc}(x)\,\df\, c \abs{x}^m
\end{equation}
for some $c>0$ and $m\ge1$. 
The running cost function $\rc\colon \Rd \times \cS \to\RR_+$ is defined by
\begin{equation*}
\rc(x,u)\,\df\,\widetilde{\rc}\bigl(\langle e,x \rangle^+u\bigr)\,.
\end{equation*}

%%%%%%%%%%%%%%%%%%%%%%%%%%%%%%%%%%%%%%%%%%%%%%%%%%%%%%%%%%%%%%%%%%%%%%%%%%%%%%%%
\begin{remark}
We only choose a running cost function as in \cref{ES3.2A} 
to simplify the exposition.
One may replace \cref{ES3.2A} with a function $\widetilde{\rc}$, 
which is locally Lipschitz continuous, and satisfies
\begin{equation}\label{ER3.1A}
c_1\abs{x}^{m}  \,\le\, \widetilde{\rc}(x)  \,\le\, c_2\abs{x}^{m}
\qquad \forall\, x\in\Rd\,,
\end{equation}
for some positive constants $c_1$, $c_2$, 
and $m \ge 1$. 
All the results still hold with $\cref{ER3.1A}$.
Moreover, the lower bound in \cref{ER3.1A} is not needed for the discounted problem
(see, e.g., \cite{AMR04}).
\end{remark}

The $\alpha$-discounted control problem for the $n^{\mathrm{th}}$ system is given by
\begin{equation*}
\Hat{V}^n_\alpha\bigl(\Hat{X}^n(0)\bigr) \,\df\, 
\inf_{U^n\in\Uadm^n}\, \Hat{J}_\alpha(\Hat{X}^n(0),{U}^n) \qquad \alpha>0\,,
\ n\in\NN\,,
\end{equation*}
where the cost criterion is defined by
\begin{equation*}
\Hat{J}_\alpha(\Hat{X}^n(0),{U}^n) \,\df\, 
\Exp\biggl[\int_0^{\infty} \E^{-\alpha{t}}\,
\rc\bigl(\Hat{X}^n(s),{U}^n(s)\bigr)\,\D{s}\biggr] 
\qquad \forall\, \alpha>0\,.
\end{equation*}

For the controlled (jump) diffusion ${X}$ in \cref{ET3.1A},
we say that
a control $U$ is admissible if it takes values in $\cS$, and non-anticipative
(see \cite{APZ19a}).
We denote the set of all admissible controls by $\Uadm$ .
The corresponding $\alpha$-discounted cost criterion for the diffusion
takes the form 
\begin{equation*}
J_\alpha(x,{U}) \,\df\, 
\Exp_x^{{U}}\biggl[\int_0^{\infty}\E^{-\alpha{t}}\,
\rc({X}_s,{U}_s)\,\D{s}\biggr]
\qquad \forall\, \alpha>0\,,
\end{equation*}
and the optimal $\alpha$-discounted value function is given by
\begin{equation}\label{ES3.2B}
V_\alpha(x) \,\df\, 
\inf_{{U} \in \Uadm}\, {J}_\alpha(x,U) \qquad \forall\, \alpha>0\,, 
\end{equation}
where $\Exp_x^{{U}}$ denotes the expectation operator
corresponding to the process under the control ${U}$,
with initial condition $x\in\Rd$.
We introduce the following assumption for the discounted problem.

%%%%%%%%%%%%%%%%%%%%%%%%%%%%%%%%%%%%%%%%%%%%%%%%%%%%%%%%%%%%%%%%%%%%%%%%%%%%%%%%
\begin{assumption}\label{A3.1} 
There exists a constant $m_{A} \ge m\vee 2$
with $m$ as in \cref{ES3.2A} such that $\Exp[(G_i)^{m_A}] < \infty$,
for all $i\in\sI$,
and $\Exp[(d_1)^{m_A\vee(m+1)}] < \infty$.
\end{assumption}
We state the main result for the discounted problem in the next theorem, 
whose proof is given in \cref{S5.2}.

%%%%%%%%%%%%%%%%%%%%%%%%%%%%%%%%%%%%%%%%%%%%%%%%%%%%%%%%%%%%%%%%%%%%%%%%%%%%%%%%
\begin{theorem}\label{T3.2}
Grant the hypotheses in \cref{A2.1,A2.2,A3.1},
and suppose that $\Hat{X}^n(0)\rightarrow x\in\Rd$ as $n\rightarrow\infty$.
Then
\begin{equation}\label{ET3.2A}
\lim_{n\rightarrow\infty}\,\Hat{V}^n_\alpha\bigl(\Hat{X}^n(0)\bigr)
\,=\, V_{\alpha}(x) \,.
\end{equation}
\end{theorem}

%%%%%%%%%%%%%%%%%%%%%%%%%%%%%%%%%%%%%%%%%%%%%%%%%%%%%%%%%%%%%%%%%%%%%%%%%%%%%%%%
\begin{remark}
Note that in \cref{T3.2}, we do not need to impose any restrictions
on the limiting abandonment rates $\{\gamma_i\colon i\in\sI\}$.
\end{remark}
 
We define the ergodic control problem for the diffusion-scaled process by
\begin{equation*}
\varrho^n\bigl(\Hat{X}^n(0)\bigr) \,\df\, 
\inf_{Z^n\in\fZsm^n}\, \Hat{J}(\Hat{X}^n(0),{Z}^n)\,,
\end{equation*}
where the cost criterion $\Hat{J}$ is given by
\begin{equation*}
\Hat{J}(\Hat{X}^n(0),{Z}^n) \,\df\, 
\limsup_{T\rightarrow\infty}\,
\frac{1}{T}\Exp^{Z^n}\biggl[\int_0^{T}
\widetilde{\rc}\bigl(\Hat{Q}^n(s)\bigr)\,\D{s}\biggr] \,.
\end{equation*}
Here, the infimum is over all Markov scheduling policies,
since for the ergodic control problem, we work with Markov processes.
For the controlled jump diffusion in \cref{ET3.1A},
the ergodic cost criterion, and the optimal ergodic value are defined by 
\begin{equation*}
{J}(x,{U}) \,\df\,  
\limsup_{T\rightarrow\infty}\,
\frac{1}{T}\Exp_x^{{U}}\biggl[\int_0^{T}\rc({X}_s,{U}_s)\,\D{s}\biggr]\,,
\end{equation*}
and 
\begin{equation}\label{ES3.2C}
 \varrho_*(x) \,\df\, \inf_{{U} \in \Uadm}\, {J} (x,{U})\,, 
\end{equation}
respectively.
By Theorem~4.1 in \cite{APZ19a}, 
it follows that $\varrho_*$ is independent of $x$, 
and optimality is attained by a stationary Markov control.

We introduce the following assumption on $G_i$ and $d_1$ for the ergodic control problem.

%%%%%%%%%%%%%%%%%%%%%%%%%%%%%%%%%%%%%%%%%%%%%%%%%%%%%%%%%%%%%%%%%%%%%%%%%%%%%%%%
\begin{assumption}\label{A3.2}
The following hold.
\begin{enumerate}
\item[\textup{(i)}]
The right derivative of $F_i(t)$ is finite,
and $F_i(t) < 1$,
for all $t\ge 0$ and $i\in\sI$.
The distribution function $F^{d_1}$ of $d_1$ satisfies the same property.
\item[\textup{(ii)}]
The mean residual life function of $G_i$ and $d_1$ are 
bounded, that is, 
there exists some positive constant $\widehat{C}$ such that
\begin{equation}\label{EA3.2A}
\frac{\int_{t}^{\infty}\bigl(1 - F^{d_1}(y)\bigr)\,\D{y}}{1 - F^{d_1}(t)}
\,\le\, \widehat{C}\,,\quad\text{and\ \ }
\frac{\int_{t}^{\infty}\bigl(1 - F_i(y)\bigr)\,\D{y}}{1 - F_i(t)} \,\le\, \widehat{C}
\quad \forall\, i\in\sI\,,
\end{equation}
and for all $t\ge0$.
\end{enumerate}
\end{assumption}

\cref{A3.2} implies 
that all absolute moments of $G_i$, $i\in\sI$, and $d_1$ are finite.
The main result of the ergodic control problem is stated in the next theorem,
whose proof is given in \cref{S5.3}.
\begin{theorem}\label{T3.3}
Grant \cref{A2.1,A2.2,A3.2}. 
In addition, suppose that $m$ in \cref{ES3.2A} is larger than $1$,
and that $\Hat{X}^n(0)\rightarrow x\in\Rd$ as $n\rightarrow\infty$.
Then, we have
\begin{equation*}
\lim_{n\rightarrow\infty}\,\varrho^n\bigl(\Hat{X}^n(0)\bigr) \,=\, \varrho_*\,.
\end{equation*}
\end{theorem}

%%%%%%%%%%%%%%%%%%%%%%%%%%%%%%%%%%%%%%%%%%%%%%%%%%%%%%%%%%%%%%%%%%%%%%%%%%%%%%%%
\section{Ergodic properties}\label{S4}

%%%%%%%%%%%%%%%%%%%%%%%%%%%%%%%%%%%%%%%%%%%%%%%%%%%%%%%%%%%%%%%%%%%%%%%%%%%%%%%%
\subsection{The limiting controlled diffusion with compound Poisson jumps}
The controlled generator of the controlled limiting jump diffusion 
in \cref{ET3.1A} is given by
\begin{equation}\label{ES4.1A}
\Ag \varphi(x,u) \,=\, \sum_{i\in\sI}b_i(x,u)\partial_i\varphi(x) + 
\frac{1}{2}\sum_{i\in\sI}\lambda_i (1+ c^2_{a,i})\partial_{ii}\varphi(x) + 
\int_{\Rd} \bigl(\varphi(x + y) - \varphi(x)\bigr)\nu_{{L}}(\D{y})
\end{equation}
for $\varphi\in \Cc^2(\Rd)$, 
where the drift $b$ satisfies \cref{ET3.1B}, and
$\nu_L(A) \df \Pi_L\bigl(\bigl\{z\in\RR_* \colon \lambda z\in A\bigr\}\bigr)$
for any Borel measurable set $A$, with $\Pi_L$ as in \cref{ES2.1C}.
We refer the reader to \cite[Section~6]{MT-III} 
for the definition of exponential ergodicity.
The following proposition is a direct consequence of \cite[Theorem~3.5]{APS19}.

%%%%%%%%%%%%%%%%%%%%%%%%%%%%%%%%%%%%%%%%%%%%%%%%%%%%%%%%%%%%%%%%%%%%%%%%%%%%%%%%
\begin{proposition}\label{P4.1}
Under any constant control $v$ such that $\varGamma v \neq 0$,
the controlled limiting jump diffusion in \cref{ET3.1A} is exponentially ergodic.
\end{proposition}

%%%%%%%%%%%%%%%%%%%%%%%%%%%%%%%%%%%%%%%%%%%%%%%%%%%%%%%%%%%%%%%%%%%%%%%%%%%%%%%%
\begin{remark}
It is shown in \cite{AHPS19}*{Theorem 5} that the limiting controlled jump diffusion
is exponentially ergodic uniformly over all stationary Markov controls
resulting in a locally Lipschitz continuous drift, if $\varGamma>0$.  
\end{remark}

\cref{P4.1} implies that the optimal control problems 
for the limiting jump diffusion are well-posed.

%%%%%%%%%%%%%%%%%%%%%%%%%%%%%%%%%%%%%%%%%%%%%%%%%%%%%%%%%%%%%%%%%%%%%%%%%%%%%%%%
\subsection{Preliminaries}
We denote the scaled hazard rate function of $G_i$ by $r^n_i$.
This is defined by
\begin{equation*}
r^n_i(h_i) 
\,\df\, \frac{\lambda_i^n \dot{F}_i(\lambda^n_i h_i)}{1 - F_i(\lambda^n_ih_i)}\,,
\quad \forall\,h_i\in\RR_+\,, \quad \forall\, i\in\sI\,,
\end{equation*}
where $\dot{F}_i$ denotes the right derivative of $F_i$.
Recall $H^n$ in \cref{ES2.2B}.
The extended generator of $(A^n,H^n)$ associated with the renewal arrival processes, 
denoted by $\cH^n$, is given by
\begin{equation}\label{E-sH}
\cH^n f(x,h) \,=\,  \sum_{i\in\sI}\frac{\partial f(x,h)}{\partial h_i} +
\sum_{i\in\sI}r_i^n(h_i) \bigl(f(x+ e_i,h - h_ie_i) - f(x,h)\bigr)
\end{equation}
for $f\in\Cc_b(\Rd\times\RR^d_+)$.

%%%%%%%%%%%%%%%%%%%%%%%%%%%%%%%%%%%%%%%%%%%%%%%%%%%%%%%%%%%%%%%%%%%%%%%%%%%%%%%%
\begin{remark}\label{R4.1}
We sketch the derivation of \cref{E-sH};
see also \cite{Davis-84}*{Theorem~5.5}.
It is enough to consider one component $(A^n_i,H^n_i)$, $i\in\sI$.
We obtain
\begin{equation*}
\begin{aligned}
&\Exp_{x,h}\bigl[f\bigl(A^n_i(t+s),H^n_i(t+s)\bigr)\bigr] -f(x,h) \\
&\quad\,=\, \Exp_{x,h}\bigl[f\bigl(A^n_i(t+s),H^n_i(t+s)\bigr)\bigr] 
- \Exp_{x,h}\bigl[f\bigl(A^n_i(t+s),h\bigr)\bigr] 
+ \Exp_{x,h}\bigl[f\bigl(A^n_i(t+s),h\bigr)\bigr] -f(x,h) \\
&\quad\,=\, r^n_{i,0,s}(h)\bigl(f(x,h+s) - f(x,h)\bigr)
+  r^n_{i,1,s}(h)\bigl(f(x+1,h) - f(x,h)\bigr)
\\
&\qquad+ \sum_{j\in\NN}r^n_{i,j,s}(h)
\Exp_{x,h}\bigl[f\bigl(x+j,H^n_i(t+s)\bigr) - f(x+j,h) \bigm| A^n_i(t+s) =x+j\bigr] \\
&\qquad+ \sum_{j\in\NN,j\ge 2}r^n_{i,j,s}(h)\bigl(f(x+j,h) - f(x,h)\bigr)
\quad \forall\,f\in\Cc_b(\RR\times\RR)\,, \ \forall\, (x,h)\in \RR\times\RR_+\,,
\end{aligned}
\end{equation*}
where
\begin{equation*}
r^n_{i,j,s}(h) \,\df\, \Prob\bigl(A^n_i(t+s) = x + j\,|\, A^n_i(t) = x, H^n_i(t)
\,=\, h\bigr)
\,=\, \Prob\bigl(A^n_i(s+h)=j\,|\,G_i \ge \lambda^n_i h\bigr)
\end{equation*}
by the regenerative property of renewal process.
Since $\dot{F}_i(t)$ is finite for all $t\ge0$, 
it follows that 
\begin{equation*}
r^n_i(h) \,\equiv\, \lim_{s\searrow0}\,\frac{1}{s}\,r^n_{i,1,s}(h)  
\,=\, \frac{\lambda_i^n \dot{F}_i(\lambda^n_i h_i)}{1 - F_i(\lambda^n_ih_i)}\,, 
\quad \text{and} \quad  
\lim_{s\searrow0}\,\frac{1}{s}\,r^n_{i,j,s}(h) = 0 \quad \text{for } j\ge 2\,.
\end{equation*}
It is evident that $\lim_{s\searrow0}r^n_{i,0,s} = 1$ 
and $\lim_{s\searrow0}r^n_{i,j,s} = 0$ for $j\in\NN$.
Thus, we obtain \cref{E-sH}.
\end{remark}

We define (compare this with \cite{Takis-99})
\begin{equation}\label{E-eta}
\eta^n_i(h_i) \,\df\, 
1 - \frac{\int_{\lambda^n_ih_i}^{\infty}\bigl(1 - F_i(y)\bigr)\,\D{y}}
{1 - F_i(\lambda^n_i h_i)}\,, \quad h_i\in\RR_+\,,\ i\in\sI\,.
\end{equation}
Note that $\eta^n_i$ is bounded by \cref{EA3.2A}.
The following identity is frequently used throughout the paper.
\begin{equation}\label{ES4.2A}
\dot\eta^n_i(h_i) - \eta^n_i(h_i)r^n_i(h_i)
\,=\, \lambda^n_i -r^n_i(h_i)\,, 
\quad \forall\,h_i\in\RR_+\,, \quad \forall\,i\in\sI\,.
\end{equation}
Recall that $c^2_{a,i}$ denotes the squared coefficient of variation of $G_i$. 
Let 
\begin{equation}\label{E-kappa}
\kappa^n_i(h_i) \,\df\,  \frac{\int_{\lambda^n_i h_i}^{\infty}
\int_{t}^{\infty}\bigl(1 - F_i(x)\bigr)\,\D{x}\,\D{t}}{1- F_i(\lambda^n_i h_i)}
-\frac{c^2_{a,i} + 1}{2}\frac{\int_{\lambda^n_i h_i}^{\infty}
\bigl(1- F_i(x)\bigr)\,\D{x}}{1 - F_i(\lambda^n_i h_i)}
\end{equation}
for $h_i\in\RR_+$ and $i\in\sI$.
Note that the first term on the right-hand side  of \cref{E-kappa} is the second order
residual life function.
It follows by \cref{EA3.2A} that $\kappa^n_i$ is bounded. 
Using \cref{E-kappa}, we obtain $\kappa^n_i(0) = 0$, and
\begin{equation}\label{ES4.2B}
\dot{\kappa}^n_i(h_i) - r^n_i(h_i)\kappa^n_i(h_i) \,=\, 
\biggl(\eta^n_i(h_i) + \frac{c^2_{a,i}-1}{2}\biggr)\lambda^n_i\,, 
\quad h_i\in\RR_+\,, \ i\in\sI\,.
\end{equation} 

The scaled hazard rate function of $d_1$ is defined by
\begin{equation*}
\beta^n_{\mathsf{d}}(k) \,\df\,
\frac{\vartheta^n\dot{F}^{d_1}(\vartheta^n k)}
{1 - {F}^{d_1}(\vartheta^n k)}\,,
\quad k\in\RR_+\,.
\end{equation*}
Recall $K^n$ in \cref{ES2.2C}.
The extended generator of $(\Psi^n,K^n)$ associated with the alternating renewal process, 
denoted by $\sK^n$, is given by
\begin{equation}\label{E-sK}
\sK^n f(\psi,k) \,=\,  \psi\, 
\beta^n_{\mathsf{u}}\bigl(f(0,0) - f(1,0)\bigr)
+ (1-\psi)\biggl( \beta^n_{\mathsf{d}}(k)\bigl(f(1,0) - f(0,k)\bigr) + 
\frac{\partial f(0,k)}{\partial k} \biggr)
\end{equation}
for $f\in\Cc_b(\{0,1\}\times\RR_+)$,
with $\beta^n_{\mathsf{u}}$ as in \cref{A2.2}.
In analogy to \cref{ES4.2A}, 
we define
\begin{equation}\label{E-upalpha}
\upalpha^n(k) \,\df\, 1 - 
\frac{\int_{\vartheta^nk}^{\infty}\bigl(1 - F^{d_1}(x)\bigr)\,\D{x}}
{1 - F^{d_1}(\vartheta^n k)} \qquad \forall\,k\in\RR_+\,.
\end{equation}
The following identities hold: $\upalpha^n(0) = 0$, and
\begin{equation}\label{ES4.2C}
 \dot{\upalpha}^n(k) - \beta^n_{\mathsf{d}}(k)\upalpha^n(k) \,=\,
\vartheta^n - \beta^n_{\mathsf{d}}(k) \qquad \forall\,k\in\RR_+\,. 
\end{equation}
Let $\tilde{\upalpha}^n(\psi,k) \,\df\, 
\bigl(\psi + \upalpha^n(k)\bigr)(\vartheta^n)^{-1}$. 
It follows by \cref{ES4.2C} that 
\begin{equation}\label{ES4.2D}
\sK^n \tilde{\upalpha}^n(\psi,k) \,=\, 
- \frac{\beta^n_{\mathsf{u}}}{\vartheta^n} \psi
+ (1 - \psi) \,.
\end{equation}
Note that $\tilde{\upalpha}^n$ is bounded by \cref{EA3.2A}.

%%%%%%%%%%%%%%%%%%%%%%%%%%%%%%%%%%%%%%%%%%%%%%%%%%%%%%%%%%%%%%%%%%%%%%%%%%%%%%%%
\subsection{Diffusion-scaled processes}

Let $\sI_0 \df \{i\in\sI\colon \gamma_i = 0 \}$.
If $\sI_0\ne\varnothing$, then,
Without loss of generality, we assume that $\sI_0 = \{1,\dotsc,\abs{\sI_0}\}$,
where $\abs{\sI_0}$ denotes the cardinality of the set $\sI_0$.
In \cref{D4.1} below, we introduce a modified priority scheduling policy
which can be described as follows:
First, $\lfloor\nicefrac{n\rho_i}{\sum_{i\in\sI_0}\rho_i}\rfloor\wedge x_i$
servers are allocated to each class $i\in\sI_0$. 
Then, the remaining servers are allocated following the static priority rule.

%%%%%%%%%%%%%%%%%%%%%%%%%%%%%%%%%%%%%%%%%%%%%%%%%%%%%%%%%%%%%%%%%%%%%%%%%%%%%%%%
\begin{definition}\label{D4.1}
The Markov policy $\Check{z}^n$ is defined by
\begin{equation*}
\Check{z}^n_i(x) \,=\, \Biggl\lfloor\frac{n\rho_i}{\sum_{i\in\sI_0}\rho_i} + 
\Biggl(n - \sum_{j\in\sI_0}\biggl(x_j\wedge
\biggl\lfloor\frac{n\rho_j}{\sum_{i\in\sI_0}\rho_i}\biggr\rfloor\biggr) - 
\sum_{j=1}^{i-1}\biggl(x_j 
- \biggl\lfloor\frac{n\rho_j}{\sum_{i\in\sI_0}\rho_i}\biggr\rfloor\biggr)^+ \Biggr)^+  
\Biggr\rfloor \wedge x_i\,, \quad \forall\, i \in \sI_0\,, 
\end{equation*}
and
\begin{equation*}
\Check{z}^n_i(x) \,\df\, 
x_i \wedge \Biggl(n - \sum_{j=1}^{i-1}x_{j} \Biggr)^+\,, 
\quad \forall\, i\in\sI \setminus\sI_0\,.
\end{equation*}
We let $\Check{q}^n_i(x) \df x_i - \Check{z}^n_i(x)$, $i\in\sI$.
\end{definition}

We define the `unscaled'
process $\Breve{X}^n$ by 
\begin{equation}\label{E-checkX}
\begin{aligned}
\Breve{X}^n_i(t) 
&\,\df\, X^n(0) + A^n_i(t) - 
S^n_{i}(t) \\
& \qquad
- R^n_{*,i}\biggl(\gamma^n_i\int_0^{t}
\bigl(\Breve{X}^n_i(s) - n\mu^n_i\rho_i\cR^n(s) - Z_i^n(s)\bigr)\,\D{s}\biggr)
+ n\mu^n_i\rho_i\cR^n(t) \\
&\,=\, 
X^n_i(t) + n\mu^n_i\rho_i\cR^n(t) \quad \text{a.s.}
\end{aligned}
\end{equation}
for $i\in\sI$ and $t\ge 0$, 
where $\cR^n(t)$ is the residual time process for the system in the `down' state
given by
\begin{equation*}
\cR^n(t) \,=\, \sum_{k=1}^{N^n_{\mathsf{u}}(t)}d^n_k
- \int_0^{t}\bigl(1 - \Psi^n(s)\bigr)\,\D{s}\,,
\end{equation*}
and $N^n_{\mathsf{u}}(t)$ is the process counting the number of 
completed `up' periods by time $t$.
Here, the second equality in \cref{E-checkX} follows by the
fact that given $X^n(0)$, $\Psi^n$ and $Z^n$,
the evolution equation in \cref{E-dynamic} admits a unique solution.
Also, if $\Psi^n(t) = 1$, then $\cR^n(t) = 0$ and thus
$\Breve{X}^n(t) = X^n(t)$ a.s.
Note that under a Markov policy $z^n \in \fZsm^n$, the process
$(\Breve{X}^n,H^n,\Psi^n,K^n)$ is Markov with state space
\begin{equation*}
\fD \,\df\, \bigl\{(\Breve{x},h,\psi,k)\in\RR^d_+\times\RR^d_+\times\{0,1\}
\times\RR_+\colon
k \equiv 0\text{ if } \psi=1 \bigr\}\,,
\end{equation*}
and
\begin{equation*}
Z^n(t) \,=\, z^n\bigl(\Breve{X}^n(t) - n\mu^n_i\rho_i\cR^n(t),
H^n(t),\Psi^n(t), K^n(t)\bigr)\,.
\end{equation*}
Under $z^n\in\fZsm^n$,
the generator of $(\Breve{X}^n,H^n,\Psi^n,K^n)$ denoted by $\Breve{\Lg}_{n}^{z^n}$ 
is given by
\begin{equation}\label{ES4.3A}
\Breve{\Lg}^{z^n}_nf(\Breve{x},h,\psi,k) \,=\, 
\overline{\Lg}^{z^n}_{n,\psi}f(\Breve{x},h,\psi,k) + 
\cI_{n,\psi}f(\Breve{x},h,\psi,k) 
+ \cQ_{n,\psi}f(\Breve{x},h,\psi,k) 
\end{equation}
for $(\Breve{x},h,\psi,k)\in\mathfrak{D}$ and 
$f\in\Cc_b(\Rd\times\RR^d_+\times\{0,1\}\times\RR_+)$.
The operators on the right-hand side of \cref{ES4.3A} are defined by
\begin{align}
\overline{\Lg}^{z^n}_{n,\psi} f(\Breve{x},h,\psi,k) &\,\df\,
\sum_{i\in\sI}\frac{\partial f(\Breve{x},h,\psi,k)}{\partial h_i} +
\sum_{i\in\sI}r_i^n(h_i) \bigl(f(\Breve{x}+ e_i,h - h_ie_i,\psi,k) 
- f(\Breve{x},h,\psi,k)\bigr)
\nonumber\\
&\mspace{-80mu}
+ \psi\,\sum_{i\in\sI}\bigl(\mu_i^n z^n_i(\Breve{x},h,1,0) 
+ \gamma_i^n q^n_i(\Breve{x},z^n) \bigr)
\bigl(f(\Breve{x} - e_i,h,1,0) - f(\Breve{x},h,1,0)\bigr) \nonumber\\
&\mspace{-80mu} + (1 - \psi) \sum_{i\in\sI}\gamma^n_i 
\bigl(f(\Breve{x}- e_i,h,0,k) - f(\Breve{x},h,0,k)\bigr) 
\int_{\RR_*}q^n_i\bigl(\Breve{x} - n\upmu^n(y - k),z^n\bigr)\,
\Tilde{F}^{d^n_1}_{\Breve{x},k}(\D{y}) \nonumber\\
&\mspace{-80mu} - (1 - \psi)\sum_{i\in\sI}n\rho_i\mu^n_i
\frac{\partial f(\Breve{x},h,0,k)}{\partial \Breve{x}_i} \label{ES4.3B}
\intertext{with $q^n(\breve{x},z^n) = \breve{x} - z^n$,} 
\cI_{n,\psi} f(\Breve{x},h,\psi,k)&\,\df\, 
\psi\,\beta^n_{\mathsf{u}}\int_{\RR_*}
\biggl(f\Bigl(\Breve{x} + \frac{n}{\vartheta^n}\upmu^ny,h,0,0\Bigr) 
- f(\Breve{x},h,1,0)\biggr)\,F^{d_1}(\D{y})\,,\label{ES4.3C}
\intertext{and}
\cQ_{n,\psi}f(\Breve{x},h,\psi,k) &\,\df\, 
(1-\psi)\biggl(\beta^n_{\mathsf{d}}(k)
\bigl(f(\Breve{x},h,1,0) - f(\Breve{x},h,0,k)\bigr) + 
\frac{\partial f(\Breve{x},h,0,k)}{\partial k} \biggr)\,. \label{ES4.3D}
\end{align}
In \cref{ES4.3B}, $\upmu^n \df (\mu^n_1\rho_1,\dotsc,\mu^n_d\rho_d)'$,
$\Tilde{F}_{\Breve{x},k}^{d^n_1}$ denotes the conditional distribution of $d^n_1$
given $\{d^n_1 > k\}$,
and $\{n\mu^n_i\rho_i(d^n_1 - k) \le \Breve{x}_i \colon i\in\sI\}$.

The first two terms on the right-hand side  of \cref{ES4.3B} correspond to the
extended generator 
associated with the renewal arrival processes. Compare this to \cref{E-sH}.
Conditioning on the alternative renewal process $\Psi^n$ in the `up' state, 
the third term on the right-hand side  of \cref{ES4.3B} corresponds to the service
and abandonment processes,
and $\cI_{n,\psi}$ corresponds to the residual time process $\cR^n$
together with $\Psi^n$.
Similarly, conditioning on the alternative renewal process in the `down' state, 
the last two terms on the right-hand side  of \cref{ES4.3B} correspond to
the abandonment process and $\cR^n$, respectively, 
and $\cQ_{n,\psi}$ corresponds to $(\Psi^n,K^n)$.
The generators in \cref{ES4.3C,ES4.3D} are analogous to the extended generator
associated with 
the alternating renewal process in \cref{E-sK}.

%%%%%%%%%%%%%%%%%%%%%%%%%%%%%%%%%%%%%%%%%%%%%%%%%%%%%%%%%%%%%%%%%%%%%%%%%%%%%%%%
\begin{remark}
We sketch the derivation of $\cI_{n,\psi}$. 
The rest of the terms in \cref{ES4.3A} follow by the calculation below and \cref{R4.1}.
To simplify the calculation, we assume that the arrival processes are Poisson, 
and only consider the $i^{\mathrm{th}}$ component $(\Breve{X}^n_i,\Psi^n,K^n)$, $i\in\sI$.
Note that $K^n(t) = 0$ when $\Psi^n(t) = 1$.
Since there are no simultaneous jumps w.p.1., here
we only consider the jumps caused by $\Psi^n$, 
that is, we consider 
\begin{equation*}
\begin{aligned}
\sum_{j\in\NN}\Bigl(\Exp_{\Breve{x},1,0}\bigl[f(\Breve{X}^n_i(t+s),\Psi^n(t+s),K^n(t+s))\bigm| 
\Breve{N}^n(t+s)-\Breve{N}^n(t) = j\bigr] - f(\Breve{x},1,0)\Bigr) p^n_{j}(t,s)\,,
\end{aligned}
\end{equation*}
for $s,t\ge 0$,
where $\Breve{N}^n(t)$ denotes the number of jumps of $\Psi^n$ up to time $t$,
and $p^n_j(t,s) = \Prob\bigl(\Breve{N}^n(t+s)-\Breve{N}^n(t) = j\bigr)$, $j\in\NN$.
By the memoryless property of `up' times, and using the same calculation as 
in \cref{R4.1} for `down' times, it is straightforward to check that
\begin{equation*}
\lim_{s\searrow 0}\,\frac{1}{s}\, p^n_1(t,s) \,=\, \beta^n_{\mathsf{u}}\,, \quad 
\text{and} \quad \lim_{s\searrow 0}\,\frac{1}{s}\, p^n_j(t,s) \,=\, 0 
\quad \text{for } j\ge2\,,
\end{equation*}
and for any $t\ge 0$. 
By the continuity of $K^n$, we have
\begin{equation*}
\lim_{s\searrow 0}\,
\Prob\bigl(\Breve{N}^n(t+s) - \Breve{N}^n(t) = 1,K^n(t+s) = 0 \bigm| K^n(t)= 0\bigr)
\,=\, 1\,.
\end{equation*}
Thus,
\begin{equation*}
\begin{aligned}
&\lim_{s\searrow 0}\, \Exp_{\Breve{x},1,0}
\bigl[f(\Breve{X}^n_i(t+s),\Psi^n(t+s),K^n(t+s))\bigm| 
\Breve{N}^n(t+s)- \Breve{N}^n(t) = 1\bigr] \\
&\mspace{400mu} \,=\, \Exp_{\Breve{x},1,0}
\biggl[f\Bigl(\Breve{x} + n\mu^n_i\rho_i\frac{1}{\vartheta^n}d_1,0,0\Bigr)\biggr]\,.
\end{aligned}
\end{equation*}
This proves \cref{ES4.3C}.
\end{remark}

%%%%%%%%%%%%%%%%%%%%%%%%%%%%%%%%%%%%%%%%%%%%%%%%%%%%%%%%%%%%%%%%%%%%%%%%%%%%%%%%
\begin{definition}\label{D4.2}
We define $\Bar{x}^n_i(\Breve{x}) \df \Breve{x}_i - \rho_in$, 
$i\in\sI$, 
\begin{equation*}
\Bar{x} \,=\, \Bar{x}^n(\Breve{x}) 
\,\df\, \left(\Bar{x}^n_1(\Breve{x}),\dotsc,\Bar{x}^n_d(\Breve{x})\right)'\,,
\quad
\Tilde{x} \,=\,
\Tilde{x}^n(\Breve{x}) \,\df\, n^{-\nicefrac{1}{2}}\Bar{x}^n(\Breve{x})\,,
\quad\Breve{x}\in\RR^d\,,
\end{equation*}
and
\begin{equation*}
\fA^n_R \,\df\, \bigl\{x\in\Rd \colon \abs{x-\rho n}\,\le\,R \sqrt{n} \bigr\}
\end{equation*}
for a positive constant $R$.
\end{definition}

Let $\widetilde{\Lg}^{z_n}_n$ denote the generator 
of the scaled joint process $\widetilde{\Xi}^n 
\df (\widetilde{X}^n,H^n,\Psi^n,K^n)$
with $\widetilde{X}^n \df n^{\nicefrac{-1}{2}}(\Breve{X}^n - n\rho)$.
The state space of $\widetilde{\Xi}^n$ is given by
\begin{equation*}	
\widetilde{\mathfrak{D}}^n\,\df\,
\bigl\{(\Tilde{x}^n(\Breve{x}),h,\psi,k)\in\RR^d\times\RR^d_+\times\{0,1\}\times\RR_+ 
\colon \Breve{x}\in\RR^d_+,\  k \equiv 0\text{ if } \psi=1 \bigr\}\,.
\end{equation*}
Then, under any $z^n\in\fZsm^n$, we have
\begin{equation}\label{ES4.3E}
\widetilde{\Lg}^{z_n}_nf(\Tilde{x},h,\psi,k)
\,=\, \Breve{\Lg}^{z_n}_n 
f(\Tilde{x}^n(\Breve{x}),h,\psi,k)\,,
\end{equation}
for $f\in\Cc_b(\Rd\times\RR^d_+\times\{0,1\}\times\RR_+)$.

The next lemma concerns the ergodicity of the process $\widetilde{\Xi}^n$
under the modified priority policy in \cref{D4.1}.
Let $\Lyap_{\upkappa,\xi}(x) \df \sum_{i\in\sI}\xi_i\abs{x_i}^\upkappa$ for $x\in\Rd$, 
where $\upkappa>0$, and $\xi$ is a positive vector.
Define the function $\widetilde{\Lyap}^n_{\upkappa,\xi} \colon 
\Rd\times\RR^d_+\times\{0,1\}\times\RR_+ \to \RR$ by
\begin{equation}\label{ES4.3F}
\begin{aligned}
\widetilde{\Lyap}^n_{\upkappa,\xi}(x,h,\psi,k) &\,\df\, \Lyap_{\upkappa,\xi}(x) 
+ \sum_{i\in\sI} \eta^n_i(h_i)\bigl(\Lyap_{\upkappa,\xi}(x + n^{\nicefrac{-1}{2}}e_i) 
- \Lyap_{\upkappa,\xi}(x)\bigr) \\
&\mspace{30mu}+ \frac{\psi + \upalpha^n(k)}{\vartheta^n}\sum_{i\in\sI}
\mu_i^n\xi_i\Bigl(\Tilde{\Lyap}^n_{\upkappa,i}(x_i) + \eta^n_i(h_i)
\bigl(\Tilde{\Lyap}^n_{\upkappa,i}(x_i+ n^{\nicefrac{-1}{2}}) -
\Tilde{\Lyap}^n_{\upkappa,i}(x_i)\bigr)\Bigr)\,, 
\end{aligned}
\end{equation}
where $\eta^n_i$ and $\upalpha^n$ are as in \cref{E-eta} and \cref{E-upalpha},
respectively, 
and $\Tilde{\Lyap}^n_{\upkappa,i}(x_i) \df -\abs{x_i}^{\upkappa}$ for $x_i\in\RR_+$ and 
$i\in\sI\setminus\sI_0$,
and
\begin{equation*}
\Tilde{\Lyap}^n_{\upkappa,i}(x_i) \,\df\,
\begin{cases}
-\abs{x_i}^{\upkappa}\,, &\quad \text{for } x_i \,<\, 
\frac{\sqrt{n}\rho_i\sum_{j\in\sI\setminus\sI_0}\rho_i}
{\sum_{j\in\sI_0}\rho_i}\,, \\
- \frac{\sqrt{n}\rho_i\sum_{j\in\sI\setminus\sI_0}\rho_i}
{\sum_{j\in\sI_0}\rho_i}\abs{x_i}^{\upkappa-1}\,, &\quad \text{for } x_i \,\ge\, 
\frac{\sqrt{n}\rho_i\sum_{j\in\sI\setminus\sI_0}\rho_i}
{\sum_{j\in\sI_0}\rho_i}\,, 
\end{cases}
\quad \forall\,i\in\sI_0\,.
\end{equation*} 
The function $\widetilde{\Lyap}^n_{\upkappa,\xi}$ is constructed
in such a manner as to allow us to take advantage of the identities
in \cref{ES4.2A,ES4.2D}.
We define the set
\begin{equation*}
\cK_n(x) \,\df\, 
\biggl\{i\in\sI_0 \colon x_i \,\ge\, 
\frac{\sqrt{n}\rho_i\sum_{j\in\sI\setminus\sI_0}\rho_i}
{\sum_{j\in\sI_0}\rho_i} \biggr\}\,. 
\end{equation*}

Note that $\widetilde{\Lg}^{\Check{z}^n}_n$ denotes the generator of
$\widetilde{\Xi}^n$ under the modified priority scheduling policy in \cref{D4.1}.
We have the following lemma.

%%%%%%%%%%%%%%%%%%%%%%%%%%%%%%%%%%%%%%%%%%%%%%%%%%%%%%%%%%%%%%%%%%%%%%%%%%%%%%%%
\begin{lemma}\label{L4.1}
Grant \cref{A2.1,A2.2,A3.2}.
For any even integer $\upkappa\ge2$, there exist positive
constants $\widetilde{C}_0$ and $\widetilde{C}_1$, a positive vector $\xi\in\RR^d_+$, and
$\Tilde{n}\in\NN$ such that:
\begin{equation}\label{EL4.1B}
\widetilde{\Lg}^{\Check{z}^n}_n 
\widetilde{\Lyap}^n_{\upkappa,\xi}(\Tilde{x},h,\psi,k) \,\le\, 
\widetilde{C}_0 - \widetilde{C}_1
\sum_{i\in\sI\setminus\cK_n(\Tilde{x})}\Lyap_{\upkappa,\xi}(\Tilde{x})
- \widetilde{C}_1 \sum_{i\in\cK_n(\Tilde{x})}
\Lyap_{\upkappa-1,\xi}(\Tilde{x})
\end{equation}
for all $n> \Tilde{n}$, and $(\Tilde{x},h,y,k)\in\widetilde{\fD}^n$.
As a consequence, for all large enough $n$,
$\widetilde{\Xi}^n$ is positive Harris recurrent
under the modified priority scheduling policy $\Check{z}^n$.
\end{lemma}

The proof of \cref{L4.1} is given in \cref{AppB}.
We continue with the following theorem.

%%%%%%%%%%%%%%%%%%%%%%%%%%%%%%%%%%%%%%%%%%%%%%%%%%%%%%%%%%%%%%%%%%%%%%%%%%%%%%%%
\begin{theorem}\label{T4.1}
Grant \cref{A2.1,A2.2,A3.2}.
Under the scheduling policy $\Check{z}^n$ in \cref{D4.1}, 
and for any $\upkappa>0$, there exists $\Check{n}\in\NN$ such that
\begin{equation}\label{ET4.1A}
\sup_{n > \Check{n}}\,\limsup_{T\rightarrow\infty}\,\frac{1}{T}\,\Exp^{\Check{z}^n}
\left[\int_0^T\abs{\Hat{X}^n(s)}^\upkappa\,\D{s}\right] \,<\, \infty\,.
\end{equation}
\end{theorem}

\begin{proof}
Let $\upkappa \ge 2$ be an arbitrary even integer.
By \cref{EL4.1B}, we have 
\begin{equation}\label{PT4.1A}
\begin{aligned}
\Exp^{\Check{z}^n}\bigl[\widetilde{\Lyap}^n_{\upkappa,\xi}
\bigl(\widetilde{\Xi}^n(T)\bigr)\bigr]
- \Exp^{\Check{z}^n}\bigl[\widetilde{\Lyap}^n_{\upkappa,\xi}(\widetilde{\Xi}^n(0))\bigr]
&\,=\, 
\Exp^{\Check{z}^n}\biggl[\int_0^T\widetilde{\Lg}^{\Check{z}^n}_n
\widetilde{\Lyap}^n_{\upkappa,\xi} 
\bigl(\widetilde{\Xi}^n(s)\bigr)\,\D{s}\biggr] \\
&\,\le\, \widetilde{C}_0 T - \widetilde{C}_1
\Exp^{\Check{z}^n}\biggl[\int_0^T\Lyap_{\upkappa-1,\xi}
\bigl(\widetilde{X}^n(s)\bigr)\,\D{s}\biggr]\,.
\end{aligned} 
\end{equation}
Since $(\vartheta^n)^{-1}$ is of order $n^{-\nicefrac{1}{2}}$ by \cref{A2.2},
it follows by Young's inequality together with \cref{EA3.2A} 
that there exist some positive constants $c_0$ and $c_1$ such that
$ c_0(\Lyap_{\upkappa,\xi} - 1) \le
\widetilde{\Lyap}^n_{\upkappa,\xi} \le c_1(1 + \Lyap_{\upkappa,\xi})$
for all large $n$.
Note that $\Hat{X}^n(0) = \widetilde{X}^n(0)$.
Thus, by \cref{PT4.1A}, we obtain
\begin{equation}\label{PT4.1B}
\widetilde{C}_1
\Exp^{\Check{z}^n}
\biggl[\int_0^T\Lyap_{\upkappa-1,\xi}\bigl(\widetilde{X}^n(s)\bigr)\,\D{s}\biggr]
\,\le\,
(\widetilde{C}_0 + c_0)T 
+ c_1\bigl(1 + \Lyap_{\upkappa,\xi}\bigl(\Hat{X}^n(0)\bigr)\bigr)
\end{equation}
for some positive constants ${C}_3$ and ${C}_4$.
By dividing both sides of \cref{PT4.1B} by $T$, and taking $T\rightarrow\infty$,
we have
\begin{equation}\label{PT4.1C}
\sup_{n > \Check{n}}\,\limsup_{T\rightarrow\infty}\,\frac{1}{T}\,\Exp^{\Check{z}^n}
\left[\int_0^T\abs{\widetilde{X}^n(s)}^{\upkappa-1}\,\D{s}\right] \,<\, \infty\,.
\end{equation}
Let $\Exp \equiv \Exp^{U^n}$ for some admissible scheduling policy $U^n$.
We have
\begin{equation}\label{PT4.1D}
\frac{1}{T}\Exp\biggl[\int_0^{T}
\abs{\Hat{X}^n_i(s) - \widetilde{X}^n_i(s)}^{\upkappa-1}
\,\D{s}\biggr] \,=\, (\mu_i^n\rho_i)^{\upkappa-1}\frac{1}{T}
\Exp\biggl[\int_0^{T}\bigl(\sqrt{n}\cR^n(s)\bigr)^{\upkappa-1}\,\D{s}\biggr]
\quad \forall\,i\in\sI\,.
\end{equation}
We use the identity
\begin{equation}\label{PT4.1E}
\Exp\bigl[\bigl(\sqrt{n}\cR^n(s)\bigr)^{\upkappa-1}\bigr] \,=\, 
\Exp\bigl[\bigl(\sqrt{n}\cR^n(s)\bigr)^{\upkappa-1}\,|\, \cR^n(s)>0\bigr]
\Prob(\cR^n(s)>0)
\end{equation}
for any $s\ge 0$. 
Here $\cR^n(s)$ is the residual time of the system in the `down' state, and thus 
$\Exp[(\sqrt{n}\cR^n(s))^{\upkappa-1}|\cR^n(s)>0] 
\le \Exp[(\sqrt{n} d^n_1)^{\upkappa-1}] \le c_2$ 
for some positive constant $c_2$, by \cref{A2.2,EA3.2A}.
Also, $\Prob(\cR^n(s)>0) = \Prob(\Psi^n(s) = 0)$,
and it follows by \cite[Theorem~3.4.4]{Ross96} that
\begin{equation*}
\lim_{s\rightarrow\infty}\,\Prob(\Psi^n(s) = 0) \,=\, 
\frac{(\vartheta^n)^{-1}}{(\beta^n_{\mathsf{u}})^{-1}+ (\vartheta^n)^{-1}}\,,
\end{equation*} 
which is of order $n^{-\nicefrac{1}{2}}$ by \cref{A2.2}.
Therefore, applying \cref{PT4.1E}, we obtain
\begin{equation}\label{PT4.1F}
\lim_{(n,T)\rightarrow \infty}\frac{1}{T}\,\Exp\biggl[
\int_0^{T}\bigl(\sqrt{n}\cR^n(s)\bigr)^{\upkappa-1}\,\D{s}\biggr] \,=\, 0\,.
\end{equation}
It follows by \cref{PT4.1D,PT4.1F} that
\begin{equation}\label{PT4.1G}
\lim_{(n,T)\rightarrow\infty}\,\frac{1}{T}\,\Exp\biggl[ 
\int_0^{T}\norm{\Hat{X}^n(s) - \widetilde{X}^n(s)}^{\upkappa-1}\,\D{s}\biggr]
\,=\, 0\,.
\end{equation}
Thus  \cref{ET4.1A} follows by \cref{PT4.1C,PT4.1G}.
This completes the proof.
\end{proof}

%%%%%%%%%%%%%%%%%%%%%%%%%%%%%%%%%%%%%%%%%%%%%%%%%%%%%%%%%%%%%%%%%%%%%%%%%%%%%%%%
\begin{definition}\label{D4.3}
We define the \emph{quantization function} 
$\varpi \colon \RR^d_+ \to \ZZ^d_+$ by
\begin{equation*}
\varpi(x) \,\df\, \left(\lfloor x_1\rfloor,\dotsc,\lfloor x_{d-1}\rfloor, 
\lfloor x_{d}  \rfloor + \sum_{i=1}^d(x_i - \lfloor x_i \rfloor)\right)\,.
\end{equation*}
For a sequence $v^n\colon \Rd \to \cS$, $n\in\NN$, of
continuous functions satisfying
$v^n\bigl(\Tilde{x}^n(x)\bigr) = e_d$ if ${x}\notin\fA^n_R$, $R>1$,
with $\fA^n_R$ as in \cref{D4.2},
we define the map
\begin{equation*}
q^n[v^n](x) \,\df\, \begin{cases}\varpi\bigl(\bigl(\langle e,x \rangle - n\bigr)^+
v^n\bigl(\Tilde{x}^n(x)\bigr)\bigr) \quad  &\text{for }
\sup_{i\in\sI}\abs{\Tilde{x}^n(x)}\,\le\, \frac{1}{2d}\sqrt{n} 
\bigl(\min_{i}\, \rho_i\bigr)\,,\\[3pt] 
\Check{q}^n(x) \quad  &\text{for } 
\sup_{i\in\sI}\abs{\Tilde{x}^n(x)}\,>\, \frac{1}{2d}\sqrt{n}
\bigl(\min_{i}\, \rho_i\bigr)\,,
\end{cases}
\end{equation*}
and the scheduling policy $z^n[v^n](x) \df x - q^n[v^n](x)$
\end{definition}

The following corollary is used to prove the upper bound for the ergodic 
control problem in \cref{S5.3.2}.

%%%%%%%%%%%%%%%%%%%%%%%%%%%%%%%%%%%%%%%%%%%%%%%%%%%%%%%%%%%%%%%%%%%%%%%%%%%%%%%%
\begin{corollary}\label{C4.1}
Under the scheduling policy $z^n[v^n]$ in \cref{D4.3},
the conclusions in \cref{L4.1,T4.1} hold.  
\end{corollary}

\begin{proof}
For all sufficiently large $n$, we have
$q^n_i[v^n](\Breve{x}) \le 2dR\sqrt{n}$ for $\Breve{x}\in\fA^n_R$
(see also the proof of \cite[Lemma 5.1]{ABP15}).
If $\sup_{i\in\sI} \abs{\Tilde{x}^n_i(\Breve{x})} \le \frac{1}{d}
\sqrt{n}\bigl(\min_{i}\, \rho_i\bigr)$,
it is evident that $\sum_{i=1}^{d-1}\Breve{x}_i \le n$,
and thus $z^n[e_d]$ is equivalent to 
the modified priority policy on this set.
Therefore, the result follows
by the argument in \cref{L4.1,T4.1}.  
\end{proof}

%%%%%%%%%%%%%%%%%%%%%%%%%%%%%%%%%%%%%%%%%%%%%%%%%%%%%%%%%%%%%%%%%%%%%%%%%%%%%%%%
\section{Asymptotic Optimality}\label{S5}

%%%%%%%%%%%%%%%%%%%%%%%%%%%%%%%%%%%%%%%%%%%%%%%%%%%%%%%%%%%%%%%%%%%%%%%%%%%%%%%%
\subsection{Results concerning the limiting jump diffusion}%\label{S5.1}

Recall that a stationary
Markov control $v$ is called stable if the process under $v$
is positive recurrent, and the set of such controls is denoted by $\Ussm$.
Let $\eom$ denote the set of ergodic
occupation measures, that is,
\begin{equation}\label{E-eom}
\eom \,\df\, \biggl\{\uppi\in\cP(\Rd\times\Act) \colon 
\int_{\Rd\times\Act}\Ag f(x,u)\,\uppi(\D{x},\D{u})\,=\, 0 \quad 
\forall\, f\in\Cc^{\infty}_c(\Rd)\biggr\}\,.
\end{equation}
See Section~2.1 in \cite{ACPZ19} for more details.

We summarize the characterization of optimal controls for the limiting jump diffusion 
in the following theorems. Recall the definition of $d_1$ in \cref{A2.2}.

%%%%%%%%%%%%%%%%%%%%%%%%%%%%%%%%%%%%%%%%%%%%%%%%%%%%%%%%%%%%%%%%%%%%%%%%%%%%%%%%
\begin{theorem}\label{T5.1}
Assume that $\Exp[(d_1)^{m+1}] < \infty$ with $m$ as in \cref{ES3.2A}.
The following hold:
\begin{enumerate}
\item[\textup{(i)}]
For $\alpha >0$, $V_\alpha$ in \cref{ES3.2B} is the minimal nonnegative solution in
$\Cc^{2,r}(\Rd)$, $r\in(0,1)$, to the HJB equation
\begin{equation}\label{ET5.1A}
\min_{u\in\Act}\bigl[\Ag V_\alpha(x,u) + \rc(x,u) \bigr]\,=\, \alpha V_\alpha(x)
\quad \text{a.e. in }\Rd\,.
\end{equation}
In addition, $V_\alpha$ has at most polynomial growth with degree $m$.
Moreover, a stationary
Markov control $v$ is optimal for the $\alpha$-discounted problem if
and only if it is an a.e.\ measurable selector from the minimizer in \cref{ET5.1A}.

\item[\textup{(ii)}]
There exists a solution $V\in\Cc^{2,r}(\Rd)$, 
$r\in(0,1)$, 
to the HJB equation
\begin{equation}\label{ET5.1B}
\min_{u\in\Act}\bigl[\Ag V(x,u) + \rc(x,u) \bigr]\,=\, \varrho_*
\quad \text{a.e. in }\Rd\,.
\end{equation}
Moreover, a stationary
Markov control $v$ is optimal for the ergodic control problem if
and only if it is an a.e.\ measurable selector from the minimizer \cref{ET5.1B}.
\end{enumerate}
\end{theorem}

\begin{proof}
We first consider (i).
It follows by Remark~5.1 in \cite{APS19} and \cref{P4.1} that 
Assumptions~2.1 and 2.2 in \cite{APZ19a} hold with $\sV_{\circ}$ and $\sV$ 
having at most polynomial growth of degree $m$.
Since $\Exp[(d_1)^{m+1}] < \infty$, 
then $\cref{ES4.1A}$ satisfies Assumption~5.1 in \cite{APZ19a}.
Therefore, the results in part (i) follow by Theorems~5.1 and 5.3 in \cite{APZ19a}.
Note that by (5.4) in \cite{APZ19a}, $V_\alpha$ has at most polynomial growth of 
degree $m$.
Similarly, the claim in part (ii) follows by Theorems~5.2 and~5.3 of \cite{APZ19a}.
\end{proof}

%%%%%%%%%%%%%%%%%%%%%%%%%%%%%%%%%%%%%%%%%%%%%%%%%%%%%%%%%%%%%%%%%%%%%%%%%%%%%%%%

If we consider \cref{ES3.2C} over all stable Markov controls,
then the ergodic control problem is equivalent to 
$\min_{\uppi\in\eom}\int_{\Rd\times\Act}\rc(x,u)\,\uppi(\D{x},\D{u})$,
see, for example, \cite[Section 4]{APZ19a}.
We summarize a result on $\epsilon$-optimal controls for the ergodic 
problem in the next theorem, which follows directly by Corollary~7.1 in \cite{APZ19a}.
Note that the constant control $v \equiv e_d$ also satisfies \cref{P4.1}.
Recall that a stationary Markov control $v$ is called precise if
it is a measurable map from $\Rd$ to $\Act$.

%%%%%%%%%%%%%%%%%%%%%%%%%%%%%%%%%%%%%%%%%%%%%%%%%%%%%%%%%%%%%%%%%%%%%%%%%%%%%%%%
\begin{theorem}\label{T5.2}
Assume that $\Exp[(d_1)^{m}] < \infty$, with $m$ as in \cref{ES3.2A}. 
For any $\epsilon>0$, there exist a continuous precise control $v_\epsilon\in\Ussm$,
and $R \equiv R(\epsilon)\in\NN$ such that $v_\epsilon \equiv e_d$ on $\Bar{B}^c_R$,
and $v_\epsilon$ is $\epsilon$-optimal, that is,
\begin{equation*}
\int_{\Rd\times\Act}\rc(x,u)\,\uppi_{v_\epsilon}(\D{x},\D{u})
\,\le\, \varrho_* + \epsilon\,.
\end{equation*}
\end{theorem}

%%%%%%%%%%%%%%%%%%%%%%%%%%%%%%%%%%%%%%%%%%%%%%%%%%%%%%%%%%%%%%%%%%%%%%%%%%%%%%%%
\subsection{Proof of \texorpdfstring{\cref{T3.2}}{}}\label{S5.2}
To prove \cref{T3.2},
we use the approach developed in \cite{AMR04}.
We first establish a key moment estimate for the diffusion-scaled process $\Hat{X}^n$,
whose proof is similar to that of \cite[Lemma 3]{AMR04}.

%%%%%%%%%%%%%%%%%%%%%%%%%%%%%%%%%%%%%%%%%%%%%%%%%%%%%%%%%%%%%%%%%%%%%%%%%%%%%%%%
\begin{lemma}\label{L5.1}
Grant the hypotheses in \cref{T3.2}.
Then 
\begin{equation}\label{EL5.1A}
\Exp\bigl[\norm{\Hat{X}^n(t)}^{m_{A}}\bigr] \,\le\, 
c_1(1 + t^{m_1})(1 + \norm{x}^{m_1})\quad \forall\,t\ge0\,,
\end{equation}
where  
$c_1$ and $m_1$ are some positive constants independent of $n$, $x$ and $t$. 
\end{lemma}

\begin{proof}
Recall $\Hat{L}^n$ and $\Hat{X}^n$ in \cref{E-HatX}, 
and $\Hat{W}^n$ in \cref{ES3.1A}.
Let $\Hat{\Phi}^n$ be a $d$-dimensional process defined by
$\Hat{\Phi}^n_i(\cdot) \df \mu^n_i\int_0^{\cdot}
\Hat{Z}^n_i(s)\bigl(1 - \Psi^n(s)\bigr)\,\D{s}$, for $i\in\sI$.
Then, 
\begin{equation*}
\mu^n_i\int_0^{t}\Hat{Z}^n_i(s)\Psi^n(s)\,\D{s}
\,=\, -\Hat{\Phi}^n_i(t) + \mu^{n}_i\int_0^{t}\Hat{Z}^n_i(s)\,\D{s}
\qquad\forall\,t\ge0\,.
\end{equation*}
Thus, we obtain
\begin{equation*}
\Hat{X}^n_i(t) \,=\, \Hat{X}^n_i(0) + \ell^n_it + \Hat{W}^n_i(t) 
+ \Hat{\Phi}_i^n(t) + \Hat{L}^n_i(t)
-\mu^n_i\int_0^{t}\Hat{Z}^n_i(s)\,\D{s} - \gamma^n_i\int_0^{t}\Hat{Q}^n_i(s)\,\D{s}
\end{equation*}
for all $t\ge0$ and $i\in\sI$.
Following the same method as in \cite[Lemma~3]{AMR04}, we have
\begin{equation}\label{PL5.1A}
\begin{aligned}
\norm{\Hat{X}^n(t)} &\,\le\,
C\biggl[ 1 + t^2 + \norm{\Hat{X}^n(0)} + \norm{\Hat{W}^n(t)
+ \Hat{L}^n(t) + \Hat{\Phi}^n(t)}  \\
&\mspace{50mu}
+ \int_0^t\norm{\Hat{W}^n(s) + \Hat{L}^n(s)+ \Hat{\Phi}^n(s)}\,\D{s} 
+ \int_0^t\int_0^s\norm{\Hat{W}^n(r) + \Hat{L}^n(r)
+ \Hat{\Phi}^n(r)}\,\D{r}\,\D{s}\biggr]
\end{aligned}
\end{equation}
for some positive constant $C$.
Let
\begin{equation*}
\widehat{N}^n(t) \,\df\, \max\Biggl\{k\ge0\colon \sum_{i=1}^{k}u_i^n\le t\Biggr\}
\end{equation*}
with $u^n$ as in \cref{ES2.1A}.
By \cref{A2.2}, $\widehat{N}^n(t)$ is a Poisson process with rate
$\beta^n_{\mathsf{u}}$.
Then, we obtain 
\begin{equation}\label{PL5.1B}
\Exp\left[ \norm{\Hat{L}^n(t)}^{m_{A}}\right] 
\,\le\, C_1 \Exp\bigl[\bigl(\sqrt{n}C^n_{\mathsf{d}}(t)\bigr)^{m_{A}}\bigr]\,\le\, 
C_1\biggl(\frac{\sqrt{n}}{\vartheta^n}\biggr)^{m_{A}}\Exp
\Biggl[\Biggl(\sum_{i=1}^{\widehat{N}^n(t)+1} d_i \Biggr)^{m_{A}}\Biggr] 
\,\le\, C_2(1 +  t^{m_2}) 
\end{equation}
for some positive constants $C_1 = \sup\{\mu^n_i\rho_i\colon n\in\NN,i\in\sI\}$, 
$C_2$, and $m_2$.
The third inequality in \cref{PL5.1B}
 follows by the independence of $\widehat{N}^n$ and $d_i$,
and \cref{A3.1}.
On the other hand, for some positive constant $C_3$, we have
\begin{equation}\label{PL5.1C}
\abs{n^{-\nicefrac{1}{2}}\Hat{Z}_i^n(s)} \,\le\, C_3\bigl( 1 + n^{-1}A^n_i(s)\bigr)
\quad\text{a.s.}
\quad \forall\,s\,\ge\,0\,.
\end{equation}
Thus,
\begin{equation}\label{PL5.1D}
\begin{aligned}
\Exp\Bigl[\babs{\Hat{\Phi}^n_i(t)}^{m_A}\Bigr] &\,\le\, 
\mu^n_i\Exp\biggl[\biggl(\int_0^{t}\babs{n^{-\nicefrac{1}{2}}\Hat{Z}^n_i(s)}
\babs{\sqrt{n}\bigl(1 - \Psi^n(s)\bigr)}\,\D{s}\biggr)^{m_A}\biggr] \\
&\,\le\, \mu^n_i(C_3)^{m_{A}}
\Bigl(1 + \sup_{s\le t}\Exp\bigl[n^{-1}A^n_i(s)\bigr]\Bigr)^{m_A} 
\Exp\Bigl[\bigl(\sqrt{n}C^n_{\mathsf{d}}(t)\bigr)^{m_A}\Bigr] \\
&\,\le\, C_4(1 + t^{m_3})
\end{aligned}
\end{equation}
for some positive constant $C_4$, 
where the second inequality follows by \cref{PL5.1C}
and the independence of $A^n$ and $\Psi^n$,
and the third inequality follows by \cite[Theorem~4]{Krichagina} and \cref{PL5.1B}.
Therefore, following the argument in the proof of \cite[Lemma~3]{AMR04}, 
and using \cref{PL5.1A,PL5.1B,PL5.1D}, 
we establish \cref{EL5.1A}. This completes the proof.
\end{proof}

%%%%%%%%%%%%%%%%%%%%%%%%%%%%%%%%%%%%%%%%%%%%%%%%%%%%%%%%%%%%%%%%%%%%%%%%%%%%%%%%
\begin{proof}[Proof of \cref{T3.2}]
We first prove the lower bound:
\begin{equation*}
\liminf_{n\rightarrow\infty}\, \Hat{V}^n_{\alpha}\bigl(\Hat{X}^n(0)\bigr) 
\,\ge\, {V}_{\alpha}(x)\,.
\end{equation*}
By \cref{T5.1}, the partial derivatives of $V_{\alpha}(x)$ 
up to order two are locally H\"older continuous.
Let $V^{l}_\alpha \df \chi_l\comp V_\alpha= \chi_l(V_\alpha)$, where
$\chi_l\in\Cc^2(\RR)$ satisfies $\chi_l(x) = x$ for $x\le l$ and
$\chi_l(x) = l+1$ for $x\ge l+2$.
Let $\cL\colon \Cc^2(\Rd) \to \Cc^2(\Rd\times\cS)$ 
be the local operator defined by 
\begin{equation*}
\cL \varphi(x,u) \,\df\, \langle b(x,u),\grad \varphi(x)\rangle 
+ \frac{1}{2}\sum_{i\in\sI}\lambda_i (1+ c^2_{a,i})\,\partial_{ii}\varphi(x)\,,
\qquad\varphi\in \Cc^2(\Rd)\,.
\end{equation*}
Compare this to \cref{ES4.1A}.
We define $\sH(x,p) \df \min_{u\in\Act}[\langle b(x,u),p \rangle + \rc(x,u)]$,
for $(x,p)\in\Rd\times\Rd$.
By It\^{o}'s formula,
for any $l > \sup_{B_R} V_\alpha$,
it follows that 
\begin{equation*}
\begin{aligned}
&\E^{-\alpha(t\wedge\uptau_R)}V^l_\alpha({X}_{t\wedge\uptau_R}) 
\,=\, V^l_\alpha(x) 
- \int_0^{t\wedge\uptau_R}\alpha\E^{-\alpha s}\,V_\alpha({X}_s)\,\D{s}
+ \int_0^{t\wedge\uptau_R}\E^{-\alpha s}\,
\cL V_\alpha({X}_s,{U}_s)\,\D{s} \\
&\mspace{50mu}
+ \int_0^{t\wedge\uptau_R}
\langle \E^{-\alpha s}\,\grad V_\alpha({X}_s), \Sigma\,\D{{W}_s}\rangle
+ \int_0^{t\wedge\uptau_R}\int_{\RR_*}\E^{-\alpha s}\,
\bigl(V_\alpha^l({X}_{s-}+ \lambda y) - V_\alpha({X}_{s-})\bigr)
\cN_{{L}}(\D{s},\D{y})\,,
\end{aligned}
\end{equation*}
where
$\cN_{{L}}$ is the Poisson random measure of $\{L_t\colon t\ge0\}$ 
with the intensity $\Pi_L$.
Thus, applying \cref{ET5.1A}, we obtain 
\begin{align*}
\begin{aligned}
\E^{-\alpha(t\wedge\uptau_R)}V^l_\alpha({X}_{t\wedge\uptau_R}) \,=&\, V^l_\alpha(x) 
+ \int_0^{t\wedge\uptau_R}\E^{-\alpha s}\,\langle b({X}_s,{U}_s),
\grad V_\alpha({X}_s)\rangle\,\D{s} \\
&\,+ \int_0^{t\wedge\uptau_R}\langle \E^{-\alpha s}\,\grad V_\alpha({X}_s),
\Sigma\,\D{{W}_s}\rangle 
- \int_0^{t\wedge\uptau_R}\E^{-\alpha s}\,
\sH\bigl({X}_s,\grad V_\alpha({X}_s)\bigr)\,\D{s}
\\
&\,+ \int_0^{t\wedge\uptau_R}\int_{\RR_*}
\E^{-\alpha s}\,\bigl(V^l_\alpha({X}_{s-}+ \lambda y) - 
V_\alpha({X}_{s-})\bigr) \widetilde{\cN}_{{L}}(\D{s},\D{y}) \\
&\,+ \int_0^{t\wedge\uptau_R}\int_{\RR_*}
\E^{-\alpha s}\,
\bigl(V^l_\alpha({X}_{s-}+ \lambda y) - V_\alpha({X}_{s-}+ \lambda y)\bigr)
\Pi_{L}(\D{s},\D{y})\,,
\end{aligned}
\end{align*}
where $\widetilde{\cN}_{{L}}(t,A) = \cN_{{L}}(t,A) 
- t\,\Pi_{{L}}(A)$ for any Borel set $A\subset\RR$.
Repeating the same calculation as for the claim (71) in \cite{AMR04},
we obtain
\begin{equation}\label{PT3.2B}
\begin{aligned}
\E^{-\alpha (t\wedge\uptau_R)}V^l_\alpha({X}_t) \,\ge\, V^l_\alpha(x) &+
\int_0^{t\wedge\uptau_R}\langle 
\E^{-\alpha s}\,\grad V^l_\alpha({X}_s), \Sigma\,\D{{W}_s}\rangle 
- \int_0^{t\wedge\uptau_R}\E^{-\alpha s}\,\rc({X}_s,{U}_s)\,\D{s} \\
&+ \int_0^{t\wedge\uptau_R}\int_{\RR_*}
\E^{-\alpha s}\,\bigl(V^l_\alpha({X}_{s-}+ \lambda y) - 
V_\alpha({X}_{s-})\bigr) \widetilde{\cN}_{{L}}(\D{s},\D{y}) \\
&+ \int_0^{t\wedge\uptau_R}\int_{\RR_*}
\E^{-\alpha s}\,\bigl(V^l_\alpha({X}_{s-}+ \lambda y) 
- V_\alpha({X}_{s-}+ \lambda y)\bigr)
\Pi_{L}(\D{s},\D{y})\,.
\end{aligned}
\end{equation}
Note that $\widetilde{\cN}_{{L}}$ is a martingale measure
and $V_\alpha$ is nonnegative.
Taking expectations on the both sides of \cref{PT3.2B}, 
the second and fourth terms
on the right-hand side  of \cref{PT3.2B} vanish.
Thus,
first taking limits as $l\rightarrow\infty$, and then as $R\rightarrow\infty$,
it follows by the monotone convergence theorem that
\begin{equation*}
\Exp\biggl[\int_0^t\E^{-\alpha s}\,\rc(X_s,U_s)\,\D{s}\biggr]\,\ge\,
V_\alpha(x) - \Exp\bigl[\E^{-\alpha t}V_\alpha({X}_t)\bigr] \,.
\end{equation*}
Applying \cref{T5.1} 
it follows that solutions of \cref{ET5.1A} 
have at most polynomial growth of degree $m$,
which corresponds to \cite[Proposition~5\,(i)]{AMR04}.
Note that \cref{L5.1} corresponds to Lemma~3 in \cite{AMR04}.
The rest of the proof of the lower bound 
follows exactly the proof of \cite[Theorem~4\,(i)]{AMR04}.

To prove \cref{ET3.2A}, we construct 
a sequence of asymptotically optimal scheduling policies ${U}^n$.
Let ${v}_\alpha$ be an optimal control to \cref{ET5.1A}.
Recall the quantization function in \cref{D4.3}.
We define a sequence of scheduling policies
\begin{equation*}
\Bar{z}^n[{v}_\alpha](\Hat{x}) \,\df\, \begin{cases}
\varpi\bigl(\langle e,\Hat{x}\rangle^+v_\alpha(\Hat{x}) \bigr)\,, 
\quad &\text{if } \Hat{x}\in\Hat{\fX}^n\,, \\
\Check{z}^n(\sqrt{n}\Hat{x}+n\rho) \quad &\text{if } \Hat{x} \notin\Hat{\fX}^n\,,
\end{cases}
\end{equation*}
where $\Check{z}^n$ is the modified priority policy in \cref{D4.1},
and
\begin{equation*}
\Hat{\fX}^n \,\df\, \bigl\{n^{\nicefrac{-1}{2}}(x - n\rho) \colon x\in\Rd,\;
\langle e,x \rangle \le x_i~\forall\, i\in\sI \bigr\}\,.
\end{equation*}
Here the policy on $(\Hat{\fX}^n)^c$ may be chosen arbitrarily. 
Let $U^n[v_\alpha]$ be the equivalent parameterization of $\Bar{z}^n[v_\alpha]$.
Following the  proof of \cite[Theorem~2\,(i)]{AMR04},
we obtain
\begin{equation*}
\int_0^{\cdot}\E^{-\alpha s}\, \Upsilon^n(s)\,\D{s} \,\Rightarrow\, 0\,,
\end{equation*}
where 
\begin{equation*}
\Upsilon^n(s) \,\df\,
\bigl\langle b\bigl(\Hat{X}^n(s),{U}^n[v_\alpha](s)\bigr),
\grad V_\alpha\bigl(\Hat{X}^n(s)\bigr)\bigr\rangle
+ \rc\bigl(\Hat{X}^n(s),U^n[v_\alpha](s)\bigr) 
- \sH\bigl(\Hat{X}^n(s),\grad V_\alpha\bigl(\Hat{X}^n(s)\bigl)\bigr)\,.
\end{equation*}
Thus, by using the method  
in \cite[Theorem~4\,(ii)]{AMR04}, and repeating the above calculation,
we obtain 
\begin{equation*}
\limsup_{n\rightarrow\infty}\, \Hat{V}^n_{\alpha}\bigl(\Hat{X}^n(0)\bigr) 
\,\le\, {V}_{\alpha}(x)\,.
\end{equation*}
This completes the proof.
\end{proof}

%%%%%%%%%%%%%%%%%%%%%%%%%%%%%%%%%%%%%%%%%%%%%%%%%%%%%%%%%%%%%%%%%%%%%%%%%%%%%%%%
\subsection{Proof of \texorpdfstring{\cref{T3.3}}{}}\label{S5.3}
In this section, we prove \cref{T3.3} by establishing lower and upper bounds.

%%%%%%%%%%%%%%%%%%%%%%%%%%%%%%%%%%%%%%%%%%%%%%%%%%%%%%%%%%%%%%%%%%%%%%%%%%%%%%%%
\subsubsection{The lower bound}%\label{S5.3.1}
We show that
\begin{equation}\label{ET3.3A}
\liminf_{n\rightarrow\infty}\,
\varrho^n\bigl(\Hat{X}^n(0)\bigr) \,\ge\, \varrho_*\,.
\end{equation}
The proof is given at the end of this subsection.

We need the following lemma whose proof is similar to that of \cref{T4.1},
and is given in \cref{AppB}.

%%%%%%%%%%%%%%%%%%%%%%%%%%%%%%%%%%%%%%%%%%%%%%%%%%%%%%%%%%%%%%%%%%%%%%%%%%%%%%%%
\begin{lemma}\label{L5.2}
Grant the hypotheses in \cref{A2.1,A2.2,A3.2}.
For any $m > 1$, and 
any sequence $\{z^n\in\fZsm^n\colon n\in\NN\}$ with 
$\sup_n \Hat{J}(\Hat{X}^n(0),{z}^n)<\infty$,
there exists $n_\circ >0$ such that
\begin{equation}\label{EL5.2A}
\sup_{n>n_{\circ}}\,\limsup_{T\rightarrow\infty}\,\frac{1}{T}\Exp^{z^n}
\left[\int_0^T\abs{\Hat{X}^n(s)}^m\,\D{s}\right] \,<\, \infty\,.
\end{equation} 
\end{lemma}

%%%%%%%%%%%%%%%%%%%%%%%%%%%%%%%%%%%%%%%%%%%%%%%%%%%%%%%%%%%%%%%%%%%%%%%%%%%%%%%%

The main challenge in the proof
lies in approximating the 
generator of the diffusion-scaled process with the
generator of the limiting jump diffusion.
Recall the extended generator $\cH^n$ of $(A^n,H^n)$ in \cref{E-sH}. 
We define the function $\phi^n[f]$ by 
\begin{equation}\label{E-phi}
\begin{aligned}
\phi^n[f](x,h) &\,\df\, f(x) + \sum_{j\in\sI}\Hat{\phi}^n_{1,j}[f](x,h) 
+ \sum_{j\in\sI}\frac{c^2_{a,j}-1}{2\sqrt{n}}\partial_j f(x)  \\
&\mspace{50mu} + \sum_{j\in\sI}\Hat{\phi}^n_{2,j}[f](x,h)
+ \sum_{j\in\sI}\frac{\kappa^n_j(h_j)}{n}\partial_{jj}f(x)
+ \sum_{j=1}^{d-1}\Hat{\phi}^n_{3,j}[f](x,h)
\end{aligned}
\end{equation}
for any $f\in\Cc^{\infty}_c(\RR^d)$, and $n\in\NN$, 
where 
\begin{equation*}
\Hat{\phi}^n_{1,j}[f](x,h) \,\df\, \frac{1}{j\,!}
\sum_{i_j\in\sI} \sum_{i_{j-1}\neq i_j}
\dotsb\sum_{i_1\notin\{i_l\colon l > 1\}}
\prod_{r=1}^{j}\eta^n_{i_r}(h_{i_r})\bigl[f\bigr]^{1,n}_{i_1\dotsb i_j}(x)\,,
\end{equation*}
with 
\begin{equation}\label{E-phi1}
\begin{aligned}
\bigl[f\bigr]^{1,n}_{i_1\dotsb i_j}(x) \,&\df\,
\bigl[f\bigr]^{1,n}_{i_1\dotsb i_{j-1}}(x + n^{-\nicefrac{1}{2}}e_{i_j})
- \bigl[f\bigr]^{1,n}_{i_1\dotsb i_{j-1}}(x)\,, \\
\bigl[f\bigr]^{1,n}_{i_1}(x) \,&\df\, f(x + n^{-\nicefrac{1}{2}}e_{i_1}) - f(x)\,.
\end{aligned}
\end{equation}
The function $\Hat{\phi}^n_{2,j}[f]$ is defined analogously to \cref{E-phi1} with 
$\bigl[f\bigr]^{1,n}_{i_1\dotsb i_{j}}$ and $\bigl[f\bigr]^{1,n}_{i_1}$ replaced by
$\bigl[f\bigr]^{2,n}_{i_1\dotsb i_{j}}$ and
\begin{equation*}
\bigl[f\bigr]^{2,n}_{i_1}(x) \,\df\, 
\sum_{j\in\sI}\frac{c^2_{a,j}-1}{2\sqrt{n}}
\bigl(\partial_j f(x + n^{-\nicefrac{1}{2}}e_{i_1}) 
- \partial_jf(x)\bigr) \,,
\end{equation*}
respectively.
Also,
\begin{equation*}
\Hat{\phi}^n_{3,j}[f](x,h) \,\df\,  \frac{1}{j\,!} 
\sum_{i_j\in\sI} \sum_{i_{j-1}\neq i_j}
\dotsb\sum_{i_1\notin\{i_l\colon l > 1\}}
\prod_{r=2}^{j+1}\eta^n_{i_r}(h_{i_r})\frac{\kappa^n_{i_1}(h_{i_1})}{n}
\bigl[f\bigr]^{3,n}_{i_1\dotsb i_{j+1}}(x)
\end{equation*}
with $\bigl[f\bigr]^{3,n}_{i_1\dotsb i_{j+1}}(x)$ defined analogously to \cref{E-phi1},
and
\begin{equation*}
[f]^{3,n}_{i_1i_2}(x) \,\df\, \partial_{i_1i_1}f(x + n^{-\nicefrac{1}{2}}e_{i_2}) 
- \partial_{i_1i_1}f(x)
\quad\text{for\ } i_1, i_2,\dotsc,i_j,\ j\in\sI\,.
\end{equation*}
Note that $\phi^n[f]$ is bounded by \cref{A3.2}\,(i).

The extended generator $\widetilde{\cH}^n$ of the scaled process 
$(\Hat{A}^n,H^n)$ is given by 
$\widetilde{\cH}^nf(\Tilde{x},h) =\cH^n f(\Tilde{x}^n(x),h)$, 
for $f\in\Cc_b(\Rd\times\RR_+^d)$.
We have the following lemma.

\begin{lemma}\label{L5.3}
Grant \cref{A2.1} and \cref{A3.2}\,\ttup{i}. Then, 
\begin{equation}\label{EL5.3A}
\begin{aligned}
\widetilde{\cH}^n\phi^n[f](\Tilde{x},h) &\,=\,  
\sum_{i\in\sI}\frac{\lambda^n_i}{\sqrt{n}} \partial_i f(\Tilde{x}) 
+ \sum_{i\in\sI}\frac{\lambda^n_i c^2_{a,i}}{2n} \partial_{ii}f(\Tilde{x}) \\
&\qquad+ \sum_{i\in\sI}\frac{\lambda^n_i}{n} 
\sum_{j\in\sI}\biggl(\eta^n_j(h_j) + \frac{c^2_{a,j} - 1}{2}\biggr)
\partial_{ij}f(\Tilde{x}) + \order\Bigl( \frac{1}{\sqrt{n}}\Bigr)
\end{aligned}
\end{equation}
for all $f\in\Cc^{\infty}_c(\RR^d)$ and $(\Tilde{x},h)\in\Rd\times\RR^d_+$.
\end{lemma}
\begin{proof}
Note that 
\begin{equation*}
\begin{aligned}
\Hat{\phi}^n_{1,1}[f] &\,=\, \sum_{i\in\sI}\eta^n_i(h_i)
\bigl(f(\Tilde{x} + n^{-\nicefrac{1}{2}}e_{i}) - f(\Tilde{x}) \bigr)\,, \\
\Hat{\phi}^n_{2,1}[f] &\,=\, \sum_{i\in\sI}\eta^n_i(h_i)
\sum_{j\in\sI}\frac{c^2_{a,j}-1}{2\sqrt{n}}
\bigl(\partial_j f(\Tilde{x} + n^{-\nicefrac{1}{2}}e_{i_1}) 
- \partial_jf(\Tilde{x})\bigr)\,.
\end{aligned}
\end{equation*}
Using \cref{ES4.2A,ES4.2B}, and the Taylor expansion, we have 
\begin{equation}\label{PL5.3A}
\begin{aligned}
&\widehat{\cH}^n\Bigl(f + \Hat{\phi}^n_{1,1}[f]
+ \sum_{j\in\sI}\frac{c^2_{a,j}-1}{2\sqrt{n}}\partial_jf
+  \Hat{\phi}^n_{2,1}[f] 
+ \sum_{j\in\sI}\frac{\kappa^n_j(h_j)}{n}\partial_{jj}f\Bigr)(\Tilde{x},h) \\
&\quad \,=\, \sum_{i\in\sI}\frac{\lambda^n_i}{\sqrt{n}} \partial_i f(\Tilde{x}) 
+ \sum_{i\in\sI}\frac{\lambda^n_i c^2_{a,i}}{2n} \partial_{ii}f(\Tilde{x}) 
+ \sum_{i\in\sI}\frac{\lambda^n_i}{n}
\sum_{j\neq i}\frac{c^2_{a,j}-1}{2}\partial_{ij}f(\Tilde{x})
+ \order\biggl(\frac{1}{\sqrt{n}}\biggr) \\
&\mspace{50mu} + \sum_{i\in\sI}r^n_i(h_i)\sum_{j\neq i}\eta^n_j(h_j) 
\Bigl([f]^{1,n}_{ij}(\Tilde{x}) + [f]^{2,n}_{ij}(\Tilde{x})\Bigr) \\
&\mspace{100mu} + \sum_{i\in\sI}\frac{\lambda^n_i}{n} 
\biggl(\eta^n_i(h_i) + \frac{c^2_{a,i} - 1}{2}\biggr)
\partial_{ii}f(\Tilde{x}) + \sum_{i\in\sI}r^n_i(h_i)
\sum_{j\neq i}\frac{\kappa^n_j(h_j)}{n}[f]^{3,n}_{ij}(\Tilde{x})\,.
\end{aligned}
\end{equation}
It is straightforward to verify that
\begin{equation}\label{PL5.3B}
\begin{aligned}
&\widehat{\cH}^n \bigl(\Hat{\phi}^n_{1,2}[f] + \Hat{\phi}^n_{2,2}[f]
+  \Hat{\phi}^n_{3,1}[f] \bigr)(\Tilde{x},h) \\
&\quad \,=\, \sum_{i\in\sI}\bigl(\dot{\eta}^n_i(h_i) - \eta^n_i(h_i)r^n_i(h_i)\bigr)
\sum_{j\neq i}\eta^n_j(h_j)
\Bigl([f]^{1,n}_{ij}(\Tilde{x}) + [f]^{2,n}_{ij}(\Tilde{x})\Bigr) \\
&\quad +  \frac{1}{2}\sum_{i\in\sI}r^n_i(h_i)\sum_{j\neq i}\sum_{k\neq i,j} 
\eta^n_j(h_j)\eta^n_k(h_k)
\Bigl([f]^{1,n}_{ijk}(\Tilde{x}) + [f]^{2,n}_{ijk}(\Tilde{x})\Bigr) \\
&\quad + \sum_{i\in\sI}\biggl(
\bigl(\dot{\eta}^n_i(h_i) - \eta^n_i(h_i)r^n_i(h_i)\bigr)\sum_{j\neq i}
\frac{\kappa^n_j(h_j)}{n}
+  \bigl(\dot{\kappa}^n_i - r^n_i(h_i)
\kappa^n_i(h_i)\bigr)\sum_{j\neq i}\frac{\eta^n_j(h_j)}{n}\biggr)
[f]^{3,n}_{ij}(\Tilde{x}) \\
&\quad
+ \sum_{i\in\sI}r^n_i(h_i)\sum_{j\neq i}\eta^n_j(h_j)
\sum_{k\neq i,j}\frac{\kappa^n_{k}(h_k)}{n}[f]^{3,n}_{ijk}(\Tilde{x})
\end{aligned}
\end{equation}
for any $(\Tilde{x},h)\in\Rd\times\RR^d_+$.
Applying \cref{ES4.2A,ES4.2B}, and combining the first term on the right-hand side
of \cref{PL5.3B}
with the third, fifth and sixth terms on the right-hand side  of \cref{PL5.3A},
we obtain the third term on the right-hand side  of \cref{EL5.3A}.
We repeat this procedure until all the terms $r^n_i$ are canceled.
This proves \cref{EL5.3A}.
\end{proof}

%%%%%%%%%%%%%%%%%%%%%%%%%%%%%%%%%%%%%%%%%%%%%%%%%%%%%%%%%%%%%%%%%%%%%%%%%%%%%%%%
\begin{definition}\label{D5.1}
We define the operator 
$\Hat{\sA}^n \colon \Cc^2(\Rd) \to \Cc^2(\Rd\times\cS)$ by
\begin{equation*}
\Hat{\sA}^nf(x,u) \,\df\, \sum_{i\in\sI}\Bigl(\sA^n_{1,i}(x,u)\partial_if(x) 
+ \frac{1}{2}\sA^n_{2,i}(x,u)\partial_{ii}f(x) \Bigl)\,,
\end{equation*}
where $\sA^n_{1,i},\sA^n_{2,i}\colon \Rd\times\cS \to \RR$, $i\in\sI$, 
are given by
\begin{equation*}
\begin{aligned}
\sA^n_{1,i}(x,u) &\,\df\, \ell^n_i - \mu^n_i(x_i - \langle e,x \rangle^+u_i) 
- \gamma^n_i\langle e,x \rangle^+u_i \,,\\
\sA^n_{2,i}(x,u) &\,\df\, 
\frac{\lambda^n_i}{n}c^2_{a,i} + 
\rho_i\mu^n_i + \frac{\mu_i^n(x_i - \langle e,x \rangle^+u_i) 
+ \gamma^n_i\langle e,x \rangle^+u_i}{\sqrt{n}}\,,
\end{aligned}
\end{equation*}
respectively. 
Define the operator $\Hat{\cI}^n$ by 
\begin{equation*}
\Hat{\cI}^n f(x) \,\df\, 
\int_{\Rd}\bigl(f(x+ y) - f(x)\bigr)\,\nu^n_{d_1}(\D y) \,,
\end{equation*}
where 
\begin{equation*}
\nu^n_{d_1}(A) \,\df\, 
\Pi^n_{d_1}
\Bigl(\bigl\{y\in\RR_*\colon \bigl(\tfrac{\sqrt{n}}{\vartheta^n}\mu^n_1\rho_1y,
\dotsc,\tfrac{\sqrt{n}}{\vartheta^n}\mu^n_d\rho_d y\bigr)\in A\bigr\}\Bigr)\,,
\end{equation*} 
with $\Pi^n_{d_1}(\D{y})\,\df\, \beta^n_{\mathsf{u}}F^{d_1}(\D{y})$, and 
$\beta^n_{\mathsf{u}}$ as in \cref{A2.2}.
\end{definition}

Recall the generator $\widetilde{\Lg}^{z^n}_n$ 
of $\widetilde{\Xi}^n$ given in \cref{ES4.3E}.
The next lemma establishes the relation between the generator of 
the diffusion-scaled process and the operator in \cref{D5.1}.

%%%%%%%%%%%%%%%%%%%%%%%%%%%%%%%%%%%%%%%%%%%%%%%%%%%%%%%%%%%%%%%%%%%%%%%%%%%%%%%%
\begin{lemma}\label{L5.4}
Grant \cref{A2.1,A2.2,A3.2}. Then, 
\begin{equation}\label{EL5.4A}
\begin{aligned}
\widetilde{\Lg}^{z^n}_{n}\phi^n[f](\Tilde{x},h,\psi,k) &\,=\, 
\Hat{\sA}^nf\bigl(\Tilde{x},v^n(\Tilde{x},h,\psi,k)\bigr) 
+ \Hat{\cI}^n f(\Tilde{x}) \\
&\mspace{50mu}+
\order\Bigl(\frac{1}{\sqrt{n}}\Bigr)\bigl(\norm{\Tilde{x}} + \norm{\Tilde{q}^n}\bigr)
+ \order(1)(1 - \psi)\bigl(\norm{\Tilde{x}} + \norm{\Tilde{q}^n} + 1\bigr)\,,
\end{aligned}
\end{equation}
for any $f\in\Cc^{\infty}_c(\RR^d)$ and $z^n\in\fZsm^n$,
where $\Tilde{q}^n = n^{-\nicefrac{1}{2}}q^n$, and
\begin{equation}\label{EL5.4B}
v^n(\Tilde{x},h,\psi,k) = \begin{cases}
\frac{\Tilde{x} - \Tilde{z}^n(\sqrt{n}\Tilde{x}+n\rho,h,\psi,k)}
{\langle e,\Tilde{x} \rangle}\,, &\quad  
\text{if } \langle e,\Tilde{x} \rangle > 0\,, \\
e_d\,, &\quad  \text{if } \langle e,\Tilde{x} \rangle \le 0\,,
\end{cases}
\end{equation}
for $(\Tilde{x},h,\psi,k)\in\widetilde{\fD}^n$,
with $\Tilde{z}^n \df n^{\nicefrac{-1}{2}}(z^n - n\rho)$. 
\end{lemma}

\begin{proof}
Note that \cref{L5.3} concerns the renewal arrival process in the
diffusion-scale.
Recall that $z^n_i = \sqrt{n}(\Tilde{x}_i - \Tilde{q}^n_i) + n\rho_i$ for $i\in\sI$,
and $\Breve{x} = \sqrt{n}\Tilde{x} + n\rho$.
We let $q^n \equiv q^n(\sqrt{n}\Tilde{x} + n\rho,z^n)$ and 
$z^n\equiv z^n(\sqrt{n}\Tilde{x} + n\rho,h,\psi,k)$.
Applying \cref{L5.3} and the Taylor expansion, 
it follows by the definition of $\widetilde{\Lg}^{z^n}_n$ that
\begin{align}\label{PL5.4A}
&\widetilde{\Lg}^{z^n}_n \phi^n[f](\Tilde{x},h,\psi,k) \,=\, 
\sum_{i\in\sI}\biggl[ 
\biggl(\frac{(\lambda^n_i - n\rho_i\mu^n_i) }{\sqrt{n}} 
- \mu_i^n(\Tilde{x}_i - \Tilde{q}_i^n) - \gamma^n_i\Tilde{q}^n_i\biggr)
\partial_if(\Tilde{x}) \nonumber\\
& \mspace{10mu}
+ \frac{1}{2}\biggl(\frac{\lambda^n_ic^2_{a,i}}{n} + \rho_i\mu_i^n + 
\frac{\Tilde{x}_i + (\mu^n_i - \gamma^n_i)\Tilde{q}^n_i}
{\sqrt{n}}\biggr)\partial_{ii}f(\Tilde{x}) 
+ \frac{\lambda^n_i-n\rho_i\mu^n_i}{n}\sum_{j\in\sI}\biggl(\eta^n_j(h_j) 
+ \frac{c^2_{a,j} -1}{2}\biggr)\partial_{ij}f(\Tilde{x}) \nonumber\\
&\mspace{10mu} 
+ (1 - \psi)\gamma^n_i 
\bigl({\phi}^n[f](\Tilde{x} - n^{\nicefrac{-1}{2}}e_i,h) -
{\phi}^n[f](\Tilde{x},h)\bigr)
\int_{\RR_*}q^n_i\bigl(\sqrt{n}\Tilde{x} + n\rho - n\upmu^n(y - k),z^n\bigr)
\Tilde{F}^{d^n_1}_{\Breve{x},k}(\D{y}) \nonumber\\
&\mspace{10mu}
+ (\psi-1)(\mu^n_iz^n_i + \gamma^n_iq^n_i)
\bigl({\phi}^n[f](\Tilde{x} - n^{\nicefrac{-1}{2}}e_i,h) -
{\phi}^n[f](\Tilde{x},h)\bigr) \nonumber\\ 
&\mspace{10mu} 
- (1-\psi)\sqrt{n}\mu^n_i\rho_i
\frac{\partial{\phi}^n[f](\Tilde{x},h)}{\partial \Tilde{x}_i}
\biggr] + \psi\,\Hat{\cI}^n {\phi}^n[f](\Tilde{x},h) 
+ \order\Bigl(\frac{1}{\sqrt{n}}\Bigr)(\norm{\Tilde{x}} + \norm{\Tilde{q}^n})
\end{align}
for any $f\in\Cc^{\infty}_c(\RR^d)$,
where
\begin{equation*}
\Hat{\cI}^n {\phi}^n[f](\Tilde{x},h) \,=\, 
\int_{\Rd}\bigl({\phi}^n[f](\Tilde{x}+y,h) 
- {\phi}^n[f](\Tilde{x},h)\bigr)\,\nu^n_{d_1}(\D y) 
\end{equation*}
by a slight abuse of notation.
It is clear that
\begin{equation}\label{PL5.4B}
\lambda^n_i - n\mu^n_i\rho_i \,=\, \order(\sqrt{n})
\end{equation}
by \cref{A2.1}, 
and thus the third term in the sum on the right-hand side of \cref{PL5.4A} 
is of order $n^{-\nicefrac{1}{2}}$.
We next consider the fifth and sixth terms 
in the sum on the right-hand side of \cref{PL5.4A}.	
Using the fact that
\begin{equation*}
\phi^n[f](\Tilde{x} - n^{\nicefrac{-1}{2}}e_i,h) - \phi^n[f](\Tilde{x},h) \,=\,
-\frac{1}{\sqrt{n}} \frac{\partial\phi^n[f](\Tilde{x},h)}{\partial \Tilde{x}_i}
+ \order\biggl(\frac{1}{n}\biggr)\,, 
\end{equation*}
and $z^n_i = \sqrt{n}{\Tilde{x}}_i + n\rho_i - \sqrt{n}\Tilde{q}^n_i$, 
we obtain
\begin{equation*}
\begin{aligned}
(\psi-1)(\mu^n_iz^n_i + \gamma^n_iq^n_i)
 \bigl({\phi}^n[f](\Tilde{x} - n^{\nicefrac{-1}{2}}e_i,h) -
 {\phi}^n[f](\Tilde{x},h)\bigr) 
 - (1-\psi)\sqrt{n}\mu^n_i\rho_i\frac{\partial{\phi}^n[f](x,h)}{\partial \Tilde{x}_i} \\
= (\psi - 1)\bigl(\mu^n_i \Tilde{x}_i + (\mu^n_i -\gamma^n_i)\Tilde{q}^n_i\bigr)
\biggl(-\frac{\partial\phi^n[f](\Tilde{x},h)}{\partial x_i} 
+ \order\Bigl(\frac{1}{\sqrt{n}}\Bigr) \biggr)\,.
\end{aligned}
\end{equation*}
Recall the definition of $\Tilde{F}^{d^n_1}_{\Breve{x},k}$ in \cref{ES4.3B}.
Note that 
\begin{equation}\label{PL5.4C}
\int_{\RR_*}n\mu_i^n\rho_i(y - k)\,
\Tilde{F}^{d^n_1}_{\Breve{x},k}(\D{y})\,\le\,  
\frac{n}{\vartheta^n}\mu^n_i\rho_i
\Exp\bigl[d_1 - \vartheta^nk\,|\, d_1 > \vartheta^nk\bigr]
\,\in\, \order(\sqrt{n})\,,
\end{equation}
where the second equality follows by \cref{A2.2,EA3.2A}.
Note that $\Tilde{q}^n_i \le \langle e,\Tilde{x} \rangle^+$
for $i\in\sI$ and $(\Tilde{x},h,\psi,k)\in\widetilde{\fD}^n$. Thus,
the fourth term in the sum on the right-hand side  of \cref{PL5.4A} is bounded by 
$C(1 - \psi)(1 + \langle e,\Tilde{x} \rangle^+)$ for some positive constant $C$.
It is evident that $\phi^n[f] - f \in \order(n^{-\nicefrac{1}{2}})$, and
\begin{equation*}
\psi\,\Hat{\cI}^n {\phi}^n[f](\Tilde{x},h) \,=\, \Hat{\cI}^n f(\Tilde{x}) 
+ (\psi - 1)\,\Hat{\cI}^n f(\Tilde{x})
+ \psi\,\Hat{\cI}^n ({\phi}^n[f] -f)(\Tilde{x},h)\,.
\end{equation*}	
Therefore, \cref{EL5.4A} follows by the boundedness of $\phi^n[f]$
and \cref{PL5.4A}.
This completes the proof.
\end{proof}

%%%%%%%%%%%%%%%%%%%%%%%%%%%%%%%%%%%%%%%%%%%%%%%%%%%%%%%%%%%%%%%%%%%%%%%%%%%%%%%%
\begin{definition}\label{D5.2}
The mean empirical measure $\Hat{\zeta}^{z^n}_T\in\cP(\Rd\times\cS)$ 
associated with $\Hat{X}^n$ and a stationary Markov policy $z^n\in\fZsm^n$
is defined by
\begin{equation*}
\Hat{\zeta}^{z^n}_T(A\times B) \,\df\, \frac{1}{T}
\Exp\biggl[\int_{0}^{T}
\Ind_{A\times B}\bigl(\Hat{X}^n(s),
v^n\bigl(\Hat{X}^n(s),H^n(s),\Psi^n(s),K^n(s)\bigr)\bigr)\,\D{s}\biggr]
\end{equation*}
for any Borel sets $A\subset\Rd$ and $B\subset\cS$,
and with $v^n$ as in \cref{EL5.4B}.
\end{definition}

%%%%%%%%%%%%%%%%%%%%%%%%%%%%%%%%%%%%%%%%%%%%%%%%%%%%%%%%%%%%%%%%%%%%%%%%%%%%%%%%
The following theorem characterizes the limit points of mean empirical measures.

%%%%%%%%%%%%%%%%%%%%%%%%%%%%%%%%%%%%%%%%%%%%%%%%%%%%%%%%%%%%%%%%%%%%%%%%%%%%%%%%
\begin{theorem}\label{T5.3}
Grant the hypotheses in \cref{T3.3}.
Let  $\{z^n\in\fZsm^n \colon n\in \NN\}$ be a sequence of policies
satisfying \cref{EL5.2A}.
Then any limit point $\uppi\in\cP(\Rd\times\cS)$ of $\Hat{\zeta}^{z^n}_T$ as 
$(n,T)\rightarrow\infty$ lies in $\eom$.
\end{theorem}

\begin{proof}
It follows directly by \cref{A2.1,A2.2} that,
for any $f\in\Cc^{\infty}_c(\RR^d)$, we have 
\begin{equation}\label{PT5.3A}
\Hat{\sA}^nf(\Hat{x},u) + \Hat{\cI}^nf(\Hat{x}) 
\,\to\, \Ag f(\Hat{x},u) \quad 
\text{as\ } n\rightarrow\infty
\end{equation}
uniformly over compact sets of $\RR^d\times\cS$.
Thus, in view of \cref{PT5.3A,E-eom}, 
in order to prove the theorem, it is enough to show that
\begin{equation}\label{PT5.3B}
\lim_{(n,T)\rightarrow\infty}\,\int_{\Rd\times\cS}
\bigl(\Hat{\sA}^nf(\Hat{x},u) + 
\Hat{\cI}^n f(\Hat{x})\bigr)\,\Hat{\zeta}^{z^n}_{T}(\D{\Hat{x}},\D{u}) \,=\, 0
\qquad\forall\,f\in\Cc^{\infty}_c(\Rd)\,.
\end{equation}
Applying \cref{EL5.2A,PT4.1G}, we obtain
\begin{equation}\label{PT5.3C}
\sup_{n>n_\circ}\,\limsup_{T\rightarrow\infty}\,\frac{1}{T}\,\Exp^{z^n}
\left[\int_0^T\abs{\widetilde{X}^n(s)}^m\,\D{s}\right] \,<\, \infty\,.
\end{equation}
It follows by the same calculation as in \cref{PL5.1B} that,
for some positive constant $C_1$, we have
\begin{equation}\label{PT5.3D}
\Exp^{z^n}\biggl[\int_0^{T}\sqrt{n}(1 - \Psi^n(s))\,\D{s}\biggr]
\,\le\, C_1(1 + T) \quad \forall\,T\ge 0\,. 
\end{equation}
Using the facts that $\Tilde{q}^n_i\le \langle e,x \rangle^+$ and
$\Psi^n(s) \in \{0,1\}$, and
Young's inequality, we obtain
\begin{align}\label{PT5.3E}
\frac{1}{T}\Exp^{z^n}&\biggl[\int_0^{T}n^{\frac{m-1}{4m}}
\bigl(1 - \Psi^n(s)\bigr)n^{\frac{1-m}{4m}}
\Bigl(\norm{\widetilde{X}^n(s)} + 
\norm{\Tilde{q}^n\bigl(\sqrt{n}\widetilde{X}^n(s) + n\rho,z^n\bigr)}\Bigr)
\,\D{s}\biggr] \nonumber\\
&\,\le\, \frac{1}{T}\Exp^{z^n}\biggl[\int_0^{T}n^{\frac{1}{4}}\bigl(1 - \Psi^n(s)\bigr)
\,\D{s}\biggr] 
+ \frac{C_2}{T}\Exp^{z^n}\biggl[\int_0^{T}n^{\frac{1-m}{4}}
\abs{\widetilde{X}^n(s)}^m \,\D{s}\biggr] \nonumber\\
&\,\le\, \frac{1}{T n^{\frac{1}{4}}}C_1(1 + T) 
+ n^{\frac{1-m}{4}}\frac{C_2}{T}
\Exp^{z^n}\biggl[\int_0^{T}\abs{\widetilde{X}^n(s)}^m\,\D{s}\biggr]
 \,\longrightarrow\, 0 \quad \text{as\ } (n,T) \rightarrow\infty\,,
\end{align}
where $C_2$ is a positive constant.
In \cref{PT5.3E},
the second inequality follows by
\cref{PT5.3D}, and the convergence follows 
by \cref{PT5.3C} and the fact that $m>1$.
Applying It\^{o}'s formula to $\phi^n[f]$, 
and using \cref{L5.4,PT5.3C,PT5.3E},
it follows by the boundedness of $\phi^n[f]$ that
\begin{equation*}
\lim_{(n,T)\rightarrow\infty}\,\frac{1}{T}\Exp^{z^n}
\biggl[\int_0^{T}
\Hat{\sA}^nf\bigl(\widetilde{X}^n(s),v^n\bigl(\widetilde{\Xi}^n(s)\bigr)\bigr) 
+ \Hat{\cI}^n f\bigl(\widetilde{X}^n(s)\bigr)\,\D{s}\biggr] \,=\, 0\,.
\end{equation*} 
Therefore, using \cref{PT4.1G} again, we obtain \cref{PT5.3B}.
This completes the proof.
\end{proof}

%%%%%%%%%%%%%%%%%%%%%%%%%%%%%%%%%%%%%%%%%%%%%%%%%%%%%%%%%%%%%%%%%%%%%%%%%%%%%%%%
\begin{proof}[Proof of \cref{ET3.3A}]
Without loss of generality, suppose $\{n_j\}\subset \NN$ is an increasing sequence 
such that $z^{n_j}\in\fZsm$ and $\sup_{j}\Hat{J}(\Hat{X}^{n_j}(0),z^{n_j})<\infty$.
Recall $\Hat{\zeta}^{z^n}_T$ in \cref{D5.2}. 
There exists a subsequence of $\{n_j\}$, denoted as $\{n_l\}$,
such that $T_{l} \rightarrow\infty$ as $l\rightarrow\infty$, and
\begin{equation}\label{PT3.3A}
\liminf_{j\rightarrow\infty}\, \Hat{J}(\Hat{X}^{n_j}(0),z^{n_j}) + \frac{1}{l}
\,\ge\, \int_{\Rd\times\Act}\rc(\Hat{x},u)
\,\Hat{\zeta}^{z^{n_l}}_{T_{l}}(\D{\Hat{x}},\D{u})\,.
\end{equation}
Applying \cref{L5.2,T5.3}, any limit of $\Hat{\zeta}^{z^{n_l}}_{T_{l}}$
along some subsequence is in $\eom$. 
Choose any further subsequence of $(T_l,n_l)$, also denoted by $(T_l,n_l)$,
such that $(T_l,n_l) \rightarrow\infty$ as $l\rightarrow\infty$, and
$\Hat{\zeta}^{z^{n_l}}_{T_{l}}\rightarrow \uppi \in \eom$. 
By letting $l \rightarrow\infty$ and using \cref{PT3.3A}, 
we obtain
\begin{equation*}
\liminf_{j\rightarrow\infty}\,\Hat{J}(\Hat{X}^{n_j}(0),z^{n_j})
\,\ge\, \int_{\Rd\times\Act}\rc(\Hat{x},u)
\,\uppi(\D{\Hat{x}},\D{u}) \,\ge\, \varrho_*\,.
\end{equation*} 
This completes the proof.
\end{proof}

%%%%%%%%%%%%%%%%%%%%%%%%%%%%%%%%%%%%%%%%%%%%%%%%%%%%%%%%%%%%%%%%%%%%%%%%%%%%%%%%
\subsubsection{The upper bound}\label{S5.3.2}
In this subsection, we show that
\begin{equation}\label{ET3.3B}
\limsup_{n\rightarrow\infty}\,
\varrho^n\bigl(\Hat{X}^n(0)\bigr) \,\le\, \varrho_*\,.
\end{equation}

The following lemma concerns the convergence of mean empirical measures for 
the diffusion-scaled state processes under the scheduling policies in \cref{D4.3}.
Recall $\fA^n_R$ in \cref{D4.2} and $\Hat{\zeta}^{z^n}_T$ in \cref{D5.2}.

%%%%%%%%%%%%%%%%%%%%%%%%%%%%%%%%%%%%%%%%%%%%%%%%%%%%%%%%%%%%%%%%%%%%%%%%%%%%%%%%
\begin{lemma}\label{L5.5}
Grant the hypotheses in \cref{T3.3}.
For $\epsilon > 0$, let $v_\epsilon$ be a continuous $\epsilon$-optimal precise control, 
whose existence is asserted in \cref{T5.2},
and $\{z^n[v^n]\colon n\in\NN\}$ be as in \cref{D4.3}, and such
that $R\equiv R(\epsilon)$ and $v^n$ agrees with $v_\epsilon$ 
on $\fA^n_R$.  
Then, the ergodic occupation measure $\uppi_{v_\epsilon}$ 
of the controlled jump diffusion in \cref{ET3.1A} 
under the control $v_\epsilon$ is the unique limit point 
in $\cP(\Rd\times\cS)$ of $\Hat{\zeta}^{z^n[v^n]}_T$ as $(n,T)\rightarrow\infty$. 
\end{lemma}

\begin{proof}
Using \cref{C4.1,T5.3}, the proof of this lemma is the same as  
that of Lemma~7.2 in \cite{AP18}.
\end{proof}

%%%%%%%%%%%%%%%%%%%%%%%%%%%%%%%%%%%%%%%%%%%%%%%%%%%%%%%%%%%%%%%%%%%%%%%%%%%%%%%%
\begin{proof}[Proof of \cref{ET3.3B}] 
Let $\upkappa = 2\lfloor m \rfloor$ with $m$ as in \cref{ES3.2A}, and
$z^n[v^n]$ be the scheduling policy in \cref{L5.5}.
By \cref{C4.1}, there exist $\tilde{n}_\circ\in\NN$, and positive constants
$\widetilde{C}_0$ and $\widetilde{C}_1$ such that 
\begin{equation}\label{PT3.3B}
\widetilde{\Lg}^{{z}^n[v^n]}_n \widetilde{\Lyap}^n_{\upkappa,\xi}(\Tilde{x},h,\psi,k)
\,\le\, 
\widetilde{C}_0 -  \widetilde{C}_1 \Lyap_{\upkappa-1,\xi}(\Tilde{x})
\qquad\forall\,(\Tilde{x},h,\psi,k)\in\widetilde{\fD}^n\,,
\quad \forall\,n> \tilde{n}_\circ\,.
\end{equation}
Recall the definition of $\widetilde{\rc}$ in \cref{ES3.2A},
and let $\Hat{z}^n[v^n] = n^{\nicefrac{-1}{2}}(z^n[v^n] - n\rho)$.
Applying \cref{PT4.1G,PT3.3B}, we may select an increasing sequence $T_n$
such that
\begin{equation*}
\sup_{n\ge \tilde{n}_\circ}\,\sup_{T\ge T_n}\,\int_{\Rd\times\Act}
\Lyap_{\upkappa-1,\xi}(\Hat{x})\,
\Hat{\zeta}^{z^n[v^n]}_T(\D{\Hat{x}},\D{u})\,<\,\infty\,.
\end{equation*}
This implies that
$\widetilde{\rc}\bigl(\Hat{x} - \Hat{z}^n[v](\sqrt{n}\Hat{x}+n\rho)\bigr)$
is uniformly integrable.
By \cref{L5.5}, $\Hat{\zeta}^{z^n[v^n]}_T$ converges in $\cP(\Rd\times\cS)$
to $\uppi_{v_\epsilon}$ as $(n,T)\rightarrow\infty$.
Applying \cref{T5.2}, we deduce that $v_\epsilon$ is an $\epsilon$-optimal
control for the running cost function. 
Since $\epsilon$ is arbitrary, \cref{ET3.3B} follows.
\end{proof}

%%%%%%%%%%%%%%%%%%%%%%%%%%%%%%%%%%%%%%%%%%%%%%%%%%%%%%%%%%%%%%%%%%%%%%%%%%%%%%%

\appendix
%%%%%%%%%%%%%%%%%%%%%%%%%%%%%%%%%%%%%%%%%%%%%%%%%%%%%%%%%%%%%%%%%%%%%%%%%%%%%%%

\section{Proofs of \texorpdfstring{\cref{L3.1,T3.1}}{}}\label{AppA}

%%%%%%%%%%%%%%%%%%%%%%%%%%%%%%%%%%%%%%%%%%%%%%%%%%%%%%%%%%%%%%%%%%%%%%%%%%%%%%%%
\begin{proof}[Proof of \texorpdfstring{\cref{L3.1}}{}]
By \cite[Lemma~5.1]{PW09},  $\Hat{S}^n_i(t)$ 
and $\Hat{R}^n_i(t)$ in \cref{E-HatX}
are martingales with respect to the filtration $\cF^n_t$ in \cref{E-filtration},
having 
predictable quadratic variation processes given by
\begin{equation*}
\langle \Hat{S}^n_{i} \rangle(t) \,=\, \mu^n_i\int_0^tn^{-1}{Z}^n_i(s)\Psi^n(s)\,\D{s}
\quad \text{and}
\quad \langle \Hat{R}^n_{i} \rangle(t) \,=\, \gamma^n_i\int_0^tn^{-1}{Q}^n_i(s)\,\D{s}\,,
\quad t\ge 0\,,
\end{equation*}
respectively.
By \cref{E-dynamic}, we have the crude inequality
\begin{equation*}
0 \,\le\, n^{-1}{X}^n_i(t) \,\le\, n^{-1}{X}^n_i(0) + n^{-1}A^n_i(t)\,, \quad t\ge0\,.
\end{equation*}
Using the balance equation in \cref{ES2.1D}, we see that
the same inequalities hold for $n^{-1}{Z}_i^n$ and $n^{-1}{Q}_i^n$.
Since $\Psi^n(s) \in \{0,1\}$, 
it follows by Lemma~5.8 in \cite{PTW07} 
that $\{\Hat{W}^n_i\colon n\in\NN\}$ is stochastically bounded in $(\DD^d,J_1)$.
Also, $\{\Hat{L}^n_i\colon n\in\NN\}$ 
is stochastically bounded in $(\DD^d,M_1)$ by \cref{ES2.1C}.
On the other hand, it is evident that
\begin{equation*}
\Hat{Y}^n_i(t) \,\le\, C\int_0^t (1 + \norm{n^{-1}{X}^n(s)})\,\D{s}\,, \quad t\ge0\,,
\end{equation*}
where $C$ is some positive constant. Thus, we obtain
\begin{equation}\label{PL3.1A}
\norm{\Hat{X}^n(t)} \,\le\, \norm{\Hat{X}^n(0)}
+ \norm{\Hat{W}^n(t)} + \norm{\Hat{L}^n(t)}
+ C\int_0^t (1 + \norm{\Hat{X}^n(s)})\,\D{s}
\quad\forall\,t\ge0\,.
\end{equation}
Since $\Hat{X}^n(0)$ is uniformly bounded, 
applying Lemma~5.3 in \cite{PTW07} and Gronwall's inequality, 
we deduce that $\{\Hat{X}^n\colon n\in\NN \}$ is stochastically bounded in $(\DD^d,M_1)$.
Using Lemma~5.9 in \cite{PTW07}, we see that
\begin{equation*}
n^{-\nicefrac{1}{2}}\Hat{X}^n \,=\, n^{-1}{X}^n - \rho \;\Rightarrow\;  
\mathfrak{e}_0 \quad \text{in} \quad (\DD^d,M_1)
\quad\text{as\ } n\rightarrow\infty\,,
\end{equation*}
which implies that $n^{-1}{X}^n \Rightarrow \mathfrak{e}_\rho$ in $(\DD^d,M_1)$.
By \cref{ES2.1D}, and the fact 
$\langle e,n^{-1}{Q}^n \rangle = (\langle e,n^{-1}{X}^n\rangle - 1)^
+ \Rightarrow \mathfrak{e}_0$,
we have $n^{-1}{Q}^n \Rightarrow \mathfrak{e}_0$, 
and thus $n^{-1}{Z}^n \Rightarrow \mathfrak{e}_\rho$.
This completes the proof.
\end{proof}
%%%%%%%%%%%%%%%%%%%%%%%%%%%%%%%%%%%%%%%%%%%%%%%%%%%%%%%%%%%%%%%%%%%%%%%%%%%%%%%%

To prove \cref{T3.1}, we first consider a modified process.
Let $\Check{X}^n = (\Check{X}^n_1,\dotsc,\Check{X}^n_d)'$ 
be the $d$-dimensional process defined by
\begin{equation}\label{ES5.1A}
\begin{aligned}
\Check{X}^n_i(t) &\,\df\, \Hat{X}^n(0) + \ell^n_it+ \Hat{W}^n_i(t) + \Hat{L}^n_i(t) - 
\int_0^t\mu^n_i \bigl(\Check{X}^n_i(s) - \langle e,\Check{X}^n(s) \rangle^+
{U}^n_i(s)\bigr)\,\D{s} \\ 
&\qquad - \int_0^t\gamma^n_i\langle e,\Check{X}^n(s) \rangle^+ {U}^n_i(s)\,\D{s}\,, 
\quad \text{for } i\in\sI\,.
\end{aligned}
\end{equation}

%%%%%%%%%%%%%%%%%%%%%%%%%%%%%%%%%%%%%%%%%%%%%%%%%%%%%%%%%%%%%%%%%%%%%%%%%%%%%%%
\begin{lemma}\label{LA.1}
As $n\rightarrow\infty$, $\Check{X}^n$ and $\Hat{X}^n$ are asymptotically equivalent, 
that is, if either of them converges in distribution as
$n\rightarrow\infty$, then so does the other, 
and both of them have the same limit. 
\end{lemma}

\begin{proof}
Let $K = K(\epsilon_1)>0$ be the constant 
satisfying $\Prob(\norm{\Hat{X}^n}_T > K) < \epsilon_1$
for $T>0$ and any $\epsilon_1 > 0$, 
where $\norm{\Hat{X}^n}_T \df \sup_{0\le t\le T}\norm{\Hat{X}^n(t)}$.
Since $\Hat{U}^n(s)\in\cS$ for $s\ge0$, 
on the event $\{\norm{\Hat{X}^n}_T \le K\}$, we obtain
\begin{equation*}
\begin{aligned}
\norm{\Check{X}^n(t) - \Hat{X}^n(t)} & \,\le\,
C_1\int_0^{t} \norm{\Hat{X}^n(s)}\bigl(1-\Psi^n(s)\bigr)\,\D{s} +  
C_2\int_0^{t} \norm{\Check{X}^n(s) - \Hat{X}^n(s)}\,\D{s} \\
& \,\le\, C_1K C^n_{\mathsf{d}}(t) + C_2\int_0^{t} \norm{\Check{X}^n(s)
- \Hat{X}^n(s)}\,\D{s}
\quad \forall\,t\in[0,T]\,, 
\end{aligned}
\end{equation*}
where $C_1$ and $C_2$ are some positive constants.
Then, by Gronwall's inequality, on the event $\{\norm{\Hat{X}^n}_T \le K\}$, we have
\begin{equation*}
\norm{\Check{X}^n(t) - \Hat{X}^n(t)} \,\le\,  C_1K C^n_{\mathsf{d}}(t) \E^{C_2 T}
\quad\forall\,t\in[0,T]\,.
\end{equation*}
Thus, applying Lemma~2.2 in \cite{PW09}, we deduce that for any $\epsilon_2 > 0$, 
there exist $\epsilon_3>0$ and
$n_{\circ} = n_{\circ}(\epsilon_1,\epsilon_2,\epsilon_3,T)$ such that
\begin{equation*}
\norm{\Check{X}^n - \Hat{X}^n}_T \,\le\, \epsilon_2
\end{equation*}
on the event 
$\{\norm{\Hat{X}^n}_T \le K\}\cap\{\norm{C^n_{\mathsf{d}}}_T\le \epsilon_3\}$,
for all $n\ge n_{\circ}$,
which implies that
\begin{equation*}
\Prob(\norm{\Check{X}^n - \Hat{X}^n}_T > \epsilon_2) < \epsilon_1\,, 
\quad \forall\,n\ge n_{\circ}\,. 
\end{equation*}
As a consequence, $\norm{\Check{X}^n-\Hat{X}^n}_T\Rightarrow 0$, as $n\rightarrow\infty$,
and this completes the proof.
\end{proof}

%%%%%%%%%%%%%%%%%%%%%%%%%%%%%%%%%%%%%%%%%%%%%%%%%%%%%%%%%%%%%%%%%%%%%%%%%%%%%%%%
\begin{proof}[Proof of \cref{T3.1}]
We first prove (i). Define the processes
\begin{equation*}
\uptau^n_{1,i}(t) \,\df\, \frac{\mu^n_i}{n}\int_0^tZ^n(s)\Psi^n(s)\,\D{s}\,, \quad
\uptau^n_{2,i}(t) \,\df\, \frac{\gamma^n_i}{n}\int_0^t Q^n(s)\,\D{s}\,,  
\end{equation*}
$\Tilde{S}_i^n(t) \df n^{-\nicefrac{1}{2}}(S^n(nt) - nt)$, and 
$\Tilde{R}_i^n(t) \df n^{-\nicefrac{1}{2}}(R^n(nt) - nt)$, for $i\in\sI$. 
Then, since $\Psi^n(s)\in\{0,1\}$ for $s\ge 0$, 
applying \cref{L3.1} and Lemma~2.2 in \cite{PW09}, 
we have
\begin{equation*}
\uptau^n_{1,i}(\cdot) \,=\,
\mu^n_i\int_0^\cdot(n^{-1}{Z}_i^n(s) - \rho_i)\Psi^n(s)\,\D{s}
+ \mu^n_i\int_0^\cdot\rho_i\Psi^n(s)\,\D{s} 
\;\Rightarrow\; \lambda_i\mathfrak{e}(\cdot)\,.
\end{equation*}
in $(\DD, M_1)$, as $n\rightarrow\infty$, and that
$\uptau^n_{2,i}$ weakly converges to the zero process. 
Since $\{A_i^n,S_i^n,R^n_i,\Psi^n\colon i\in\sI, n\in\NN\}$ are independent processes, 
and $\tau^n_{1,i}$ and $\tau^n_{2,i}$ converge to deterministic functions,
we have joint weak convergence of 
$(\Hat{A}^n,\Hat{S}^n,\Hat{R}^n,\Hat{L}^n,\uptau^n_{1},\uptau^n_2)$,
where $\uptau^n_1\df (\uptau^n_{1,1},\dotsc,\uptau^n_{1,d})'$,
and $\uptau^n_2$ is defined analogously.
On the other hand, since the second moment of $A^n$ is finite,
it follows that $\Hat{A}^n$  converges weakly to a $d$-dimensional Wiener process
with mean $0$ and covariance matrix 
$\diag\bigl(\sqrt{\lambda_1c^2_{a,1}},\dotsc,\sqrt{\lambda_dc^2_{a,d}}\bigr)$
(see, e.g., \cite{IW-71}).
Therefore, by the FCLT for the Poisson processes $\Tilde{S}^n$
and $\Tilde{R}^n$, 
and using the random time change lemma in
\cite[Page 151]{Patrick-99}, we obtain (i).

Using \cref{PL3.1A} and \cref{T3.1}\,(i), 
the proof of (ii) is same as the proof of \cite[Lemma~4 (iii)]{AMR04}.

To prove (iii), we first show any limit of $\Check{X}^n$ 
in \cref{ES5.1A} satisfies \cref{ET3.1A}.
Following an argument similar to the proof of Lemma~5.2 in \cite{PW09}, 
one can easily show that the $d$-dimensional integral mapping 
$x = \Lambda(y,u) \colon \DD^d\times\DD^d \to \DD^d$ defined by 
\begin{equation*}
x(t) = y(t) + \int_0^{t} h\bigl(x(s),u(s)\bigr)\,\D{s}
\end{equation*}
is continuous in $(\DD^d,M_1)$, 
provided that the function $h\colon \Rd\times\Rd \to \Rd$ is Lipschitz continuous
in each coordinate.
Since
\begin{equation*}
\Check{X}^n \,=\, \Lambda(\Hat{X}^n(0) + \Hat{W}^n + \Hat{L}^n, {U}^n)\,,
\end{equation*}  
then, by the tightness of ${U}^n$ and the continuous mapping theorem, 
any limit of $\Check{X}^n$ satisfies \cref{ES5.1A}, 
and the same result holds for $\Hat{X}^n$ by \cref{LA.1}.

Recall the definition of $\breve{\tau}^n$ in \cref{ES2.2A}. 
It is evident that
\begin{equation}\label{PT3.1B}
\begin{aligned}
\Hat{L}^n_i(t+r) - \Hat{L}^n_i(t) &\,=\, \Hat{L}^n_i(\breve{\tau}^n(t) + r)
-  \Hat{L}^n_i\bigl(\breve{\tau}^n(t)\bigr) \\ 
&\mspace{50mu}+ \Hat{L}^n_i(t+r) 
-  \Hat{L}^n_i(\breve{\tau}^n(t) + r) + \Hat{L}^n_i\bigl(\breve{\tau}^n(t)\bigr)
-  \Hat{L}^n_i(t)\,.
\end{aligned}
\end{equation}
for all $t,r\ge0$ and $i\in\sI$.
By \cref{A2.2}, we have $\breve{\tau}^n(t) \Rightarrow t$ as $n\rightarrow\infty$,
for $t\ge 0$.
Then, by the random time change lemma  in \cite[Page 151]{Patrick-99},
we deduce that the last four terms on the right-hand side  of \cref{PT3.1B} converge
to $0$ in
distribution. It follows by \cref{T3.1}\,(i) and \cref{PT3.1B} that 
\begin{equation*}
\Hat{L}^n(\breve{\tau}^n(t) + r)
-  \Hat{L}^n\bigl(\breve{\tau}^n(t)\bigr) \,\Rightarrow\, 
\lambda L_{t+r} - \lambda L_{t} \quad \text{in }\Rd\,.
\end{equation*}
Repeating the same argument we establish convergence of $\Hat{S}^n$ and $\Hat{R}^n$. 
Proving that $U$ is non-anticipative 
follows exactly as in \cite{AMR04}*{Lemma 6}.
This completes the proof of (iii).
\end{proof}

%%%%%%%%%%%%%%%%%%%%%%%%%%%%%%%%%%%%%%%%%%%%%%%%%%%%%%%%%%%%%%%%%%%%%%%%%%%%%%%%
\section{Proofs of \texorpdfstring{\cref{L4.1,L5.2}}{}} \label{AppB}
In this section, we construct two functions,
which are used to show the ergodicity of $\widetilde{\Xi}^n$.
We provide two lemmas concerning the properties of these
functions, respectively.
The proofs of \cref{L4.1,L5.2} are given at the end of this section.

%%%%%%%%%%%%%%%%%%%%%%%%%%%%%%%%%%%%%%%%%%%%%%%%%%%%%%%%%%%%%%%%%%%%%%%%%%%%%%%%
\begin{definition}
For $z^n \in\fZsm^n$, 
define the operator $\Lg^{z^n}_n\colon \Cc_b(\Rd\times\Rd) \to \Cc_b(\Rd\times\Rd)$ by
\begin{equation}\label{EDB.1A}
\begin{aligned}
\Lg^{z^n}_nf(\Breve{x},h) &\,\df\, 
\sum_{i\in\sI}\frac{\partial f(\Breve{x},h)}{\partial h_i} +  
\sum_{i\in\sI}  r_{i}^n(h_i)\bigl(f(\Breve{x} +  e_i, h-h_i\,e_i) 
- f(\Breve{x},h)\bigr) \\
&\qquad+ \sum_{i\in\sI}\mu_i^n z^n_i\bigl(f(\Breve{x} - e_i,h) 
- f(\Breve{x},h)\bigr) 
+ \sum_{i\in\sI}\gamma_i^n q^n_i\bigl(f(\Breve{x} - e_i,h) - f(\Breve{x},h)\bigr)
\end{aligned}
\end{equation}
for $f\in\Cc_b(\Rd\times\Rd)$ and any $(\Breve{x},h)\in\RR^d_+\times\RR^d_+$, 
with $q^n \df \Breve{x} - z^n$. 
\end{definition}
Note that if $d_1^n \equiv 0$ for all $n$,
the queueing system has no interruptions.
In this situation, under a Markov scheduling policy,
the (infinitesimal) generator of $(X^n,H^n)$ takes the form of \cref{EDB.1A}.
Recall the scheduling policies $\Check{z}^n$ in \cref{D4.1}, and
 $\Bar{x} = \Breve{x} - n\rho$ in \cref{D4.2}.
We define the sets 
\begin{equation*}
\Tilde{\cK}_n(\Breve{x}) \,\df\, \biggl\{i\in\sI_0 \colon \Breve{x}_i \,\ge\, 
\frac{n\rho_i}{\sum_{j\in\sI_0}\rho_j}\biggr\}
\,=\, \biggl\{i\in\sI_0 \colon \Bar{x}_i \,\ge\, 
\frac{n\rho_i\sum_{j\in\sI\setminus\sI_0}\rho_j}{\sum_{j\in\sI_0}\rho_j}\biggr\}\,.
\end{equation*}
We have the following lemma.

%%%%%%%%%%%%%%%%%%%%%%%%%%%%%%%%%%%%%%%%%%%%%%%%%%%%%%%%%%%%%%%%%%%%%%%%%%%%%%%%
\begin{lemma}\label{LB.1}
Grant \cref{A2.1,A2.2,A3.2}.
For any even integer $\upkappa\ge2$, there exist 
a positive vector $\xi\in\RR^d_+$, $\Breve{n}\in\NN$, and 
positive constants  $\Breve{C}_0$ and $\Breve{C}_1$,
such that the functions $f_n$, $n\in\NN$, defined by 
\begin{equation}\label{ELB.1A}
f_n(\Breve{x},h) \,\df\, \sum_{i\in\sI}\xi_i \abs{\Bar{x}_i}^\upkappa
+ \sum_{i\in\sI}\eta^n_i(h_i) \xi_i
\bigl(\abs{\Bar{x}_i+1}^\upkappa - \abs{\Bar{x}_i}^\upkappa\bigr)
\quad \forall\, (\Breve{x},h)\in\RR^d_+\times\RR^d_+\,,
\end{equation}
with $\eta^n_i$ as defined in \cref{E-eta}, satisfy
\begin{equation}\label{ELB.1B}
\begin{aligned}
\Lg^{\Check{z}^n}_nf_n(\Breve{x},h) \,\le\, \Breve{C}_0
n^{\nicefrac{\upkappa}{2}} 
- \Breve{C}_1\sum_{i\in\sI\setminus\Tilde{\cK}_n(\Breve{x})}\xi_i \abs{\Bar{x}_i}^\upkappa
- \Breve{C}_1\sum_{i\in\Tilde{\cK}_n(\Breve{x})}\bigl(\mu^n_i(\Check{z}^n_i - n\rho_i)
+\gamma^n_i\Check{q}^n_i\bigr)\abs{\Bar{x}_i}^{\upkappa-1} \\
+ \sum_{i\in\sI}
\bigl(\order(\sqrt{n})\order\bigl(\abs{\Bar{x}_i}^{\upkappa-1}\bigr) 
+ \order(n)\order\bigl(\abs{\Bar{x}_i}^{\upkappa-2})\bigr)
\end{aligned}
\end{equation}
for all $n\ge \Breve{n}$ and $(\Breve{x},h)\in\RR^d_+\times\RR^d_+$.
\end{lemma}

\begin{proof}
Using the estimate
\begin{equation}\label{PLB.1A}
(a \pm 1)^m - a^{\upkappa} \,=\, \pm \upkappa a^{\upkappa-1} + \order(a^{\upkappa-2}) 
\qquad \forall\, a\in\RR\,,
\end{equation}
an easy calculation shows that
\begin{equation}\label{PLB.1B}
\begin{aligned}
\Lg^{\Check{z}^n}_n f_n(\Breve{x},h) &\,=\, \sum_{i\in\sI}
\dot\eta^n_i(h_i)
\xi_i\bigl(\abs{\Bar{x}_i+1}^\upkappa - \abs{\Bar{x}_i}^\upkappa\bigr)
+ \sum_{i\in\sI}r^n_i(h_i)\eta^n_i(0)\xi_i
\bigl((\Bar{x}_i+2)^{\upkappa} - (\Bar{x}_i+1)^{\upkappa}\bigr)\\
&\qquad 
- \sum_{i\in\sI}r^n_i(h_i)\eta^n_i(h_i)\xi_i
\bigl(\abs{\Bar{x}_i+1}^\upkappa - \abs{\Bar{x}_i}^\upkappa\bigr)\\
&\qquad + \sum_{i\in\sI} \eta^n_i(h_i)(\mu^n_i\Check{z}^n_i 
+ \gamma^n_i\Check{q}^n_i)\order(\abs{\Bar{x}_i}^{\upkappa-2})
+ \sum_{i\in\sI}r^n_i(h_i)\xi_i(\abs{\Bar{x}_i + 1}^\upkappa -
\abs{\Bar{x}_i}^\upkappa)\\
&\qquad
+ \sum_{i\in\sI}(\mu^n_i\Check{z}^n_i + \gamma^n_i\Check{q}^n_i)\xi_i
(\abs{\Bar{x}_i - 1}^\upkappa - \abs{\Bar{x}_i}^\upkappa)\,,
\end{aligned}
\end{equation}
where for the fourth term on the right-hand side we also used the fact that
\begin{equation*}
\bigl(\abs{\Bar{x}_i}^\upkappa - \abs{\Bar{x}_i - 1}^\upkappa\bigr) 
- \bigl(\abs{\Bar{x}_i+1}^\upkappa - \abs{\Bar{x}_i}^\upkappa \bigr) \,=\, 
\order(\abs{\Bar{x}_i}^{\upkappa-2})\,.
\end{equation*}
It is clear that $\eta^n_i(0) = 0$, since $F_i(0) = 0$ and $\Exp[G_i] = 1$.
On the other hand, $\eta^n_i(t)$ is bounded for all $n\in\NN$ and $t\ge 0$
by \cref{A3.2}. 
Thus, applying \cref{PLB.1A,PLB.1B,ES4.2A}, it follows that
\begin{equation}\label{PLB.1C}
\begin{aligned}
\Lg^{\Check{z}^n}_n f_n(\Breve{x},h) \,=\, \sum_{i\in\sI} \bigl[\xi_i
(\lambda^n_i - \mu_i^n\Check{z}^n_i - \gamma^n_i\Check{q}^n_i)
\bigl(\upkappa(\Bar{x}_i)^{\upkappa-1} + \order(\abs{\Bar{x}_i}^{\upkappa-2})\bigr)\\
+ \eta^n_i(h_i)(\mu^n_i\Check{z}^n_i + \gamma^n_i\Check{q}^n_i)
\order(\abs{\Bar{x}_i}^{\upkappa-2})\bigr]\,.
\end{aligned}
\end{equation}
Since $\eta^n_i(h_i)$ is uniformly bounded, and
$\Check{z}^n_i, \Check{q}^n_i \le \Bar{x}_i + n\rho_i$,
it follows that the last term in \cref{PLB.1C} is equal to 
$\order(n)\order(\abs{\Bar{x}_i}^{\upkappa-2}) 
+ \order(\abs{\Bar{x}_i}^{\upkappa-1})$.
Note that for $i\in\sI\setminus\sI_0$, $\Check{z}^n_i$ is equivalent to the 
static priority scheduling policy. 
Note also, that
\begin{equation}\label{PLB.1D}
\Bar{x}_i \,\ge\, \Check{z}^n_i - n\rho_i \,\ge\, 
\frac{n\rho_i\sum_{j\in\sI\setminus\sI_0}\rho_j}{\sum_{j\in\sI_0}\rho_j} \,>\, 0
\qquad \forall\,i\in\Tilde{\cK}_n(\Breve{x})\,,
\end{equation}
and for $i\in\sI_0 \setminus\Tilde{\cK}_n(\Breve{x})$, 
we have $\Check{z}^n_i - n\rho_i = \Bar{x}_i$ and $\Check{q}^n_i = 0$.
By using \cref{PLB.1C}, and the identity in \cref{PL5.4B}, 
we obtain
\begin{equation}\label{PLB.1E}
\begin{aligned}
\Lg^{\Check{z}^n}_n f_n(\Breve{x},h) &\,\le\, 
\sum_{i\in\sI\setminus\sI_0}\xi_i
\bigl(-\mu^n_i \Bar{x}_i
+ (\mu^n_i - \gamma^n_i) \Check{q}^n_i\bigr)m(\Bar{x}_i)^{\upkappa-1}\\
&\mspace{50mu} 
- \sum_{i\in\Tilde{\cK}_n(\Breve{x})}\xi_i\bigl(\mu^n_i
(\Check{z}^n_i - n\rho_i) 
+ \gamma^n_i\Check{q}^n_i\bigr) \abs{\Bar{x}_i}^{\upkappa-1} \\
&\mspace{100mu}
- \sum_{i\in\sI_0\setminus\Tilde{\cK}_n(\Breve{x})} 
\xi_i\mu^n_i\abs{\Bar{x}_i}^\upkappa
+ \sum_{i\in\sI}\bigl(\order(\sqrt{n})\order(\abs{\Bar{x}_i}^{\upkappa-1}) 
+ \order(n)\order(\abs{\Bar{x}_i}^{\upkappa-2})\bigr)\,.
\end{aligned}
\end{equation}
Let $\breve{c}_1\df \sup_{i,n}\{\gamma^n_i,\mu^n_i\}$,
and $\breve{c}_2$ be some constant such that
$\inf\{\mu_i^n,\gamma^n_j \colon i\in\sI,j\in \sI\setminus\sI_0, n \in\NN \} 
\ge \breve{c}_2 > 0$. 
We select a positive vector $\xi\in\RR^d_+$ such that
$\xi_1 \df 1$,
$\xi_{i}\df\frac{\kappa^m_1}{d^\upkappa}
\min_{i^{\prime}\le i-1}\xi_{i^{\prime}}$, $i\ge2$, with 
$\kappa_1 \df \frac{\breve{c}_1}{8\breve{c}_2}$. 
Compared to \cite[Lemma 5.1]{ABP15}, 
the important difference here is that, for $i\in\sI\setminus\sI_0$, we have
\begin{equation*}
\Check{q}^n_i \,=\, \Biggl(\Breve{x}_i - \biggl(n - \sum_{j\in\Tilde{\cK}_n(\Breve{x})}\Check{z}^n_j
- \sum_{j\in\sI_0\setminus\Tilde{\cK}_n(\Breve{x})}x_j
- \sum_{j=\abs{\sI_0}+1}^{i-1}x_j\biggr)^+\Biggr)^+\,.
\end{equation*}
Repeating the argument in the proof of \cite[Lemma 5.1]{ABP15},
it follows by \cref{PLB.1E} that
\begin{equation}\label{PLB.1H}
\begin{aligned}
\Lg^{\Check{z}^n}_n f_n(\Breve{x},h) &\,\le\, c_3n^{\nicefrac{\upkappa}{2}} 
- c_4 \sum_{i\in\sI\setminus\Tilde{\cK}_n(\Breve{x})}\xi_i\abs{\Bar{x}_i}^\upkappa
- c_5\sum_{i\in\Tilde{\cK}_n(\Breve{x})}\xi_i\bigl(\mu^n_i(\Check{z}^n_i -n\rho_i)
+ \gamma^n_i\Check{q}^n_i\bigr)\abs{\Bar{x}_i}^{\upkappa-1} \\
&\mspace{20mu}+ \frac{c_5}{2}
\sum_{i\in\Tilde{\cK}_n(\Breve{x})}\xi_i\mu^n_i
\bigl(\Check{z}^n_i -n\rho_i\bigr)^\upkappa
+ \sum_{i\in\sI}\bigl(\order(\sqrt{n})\order(\abs{\Bar{x}_i}^{\upkappa-1}) 
+ \order(n)\order(\abs{\Bar{x}_i}^{\upkappa-2})\bigr)
\end{aligned}
\end{equation}
for some positive constants $c_3$, $c_4$ and $c_5$.
Therefore, \cref{ELB.1B} follows by \cref{PLB.1D,PLB.1H},
and this completes the proof.
\end{proof}
 
%%%%%%%%%%%%%%%%%%%%%%%%%%%%%%%%%%%%%%%%%%%%%%%%%%%%%%%%%%%%%%%%%%%%%%%%%%%%%%%%

Let
\begin{equation}\label{ESBB}
\Tilde{g}_n(\Breve{x},h,\psi,k) \,\df\,
\frac{\psi + \upalpha^n(k)}{\vartheta^n}
\sum_{i\in\sI}\mu_i^n\xi_i\Bigl(\Tilde{g}_{n,i}(\Breve{x}_i)
+ \eta^n_i(h_i)\bigl(\Tilde{g}_{n,i}(\Breve{x}_i + 1) 
- \Tilde{g}_{n,i}(\Breve{x}_i)\bigr)\Bigr)
\end{equation}
for $(\Breve{x},h,\psi,k)\in\fD$,
where $\Tilde{g}_{n,i}(\Breve{x}_i) 
\df - \abs{\Bar{x}_i}^{\upkappa}$ for $i\in\sI\setminus\sI_0$, and
\begin{equation*}
\Tilde{g}_{n,i}(\Breve{x}_i) \,\df\, \begin{cases}
- \abs{\Bar{x}_i}^\upkappa\,, &\quad \text{if } \Bar{x}_i \,<\,
\frac{n\rho_i\sum_{j\in\sI\setminus\sI_0}\rho_j}{\sum_{j\in\sI_0}\rho_j}\,, \\
- \frac{n\rho_i\sum_{j\in\sI\setminus\sI_0}\rho_j}
{\sum_{j\in\sI_0}\rho_j}\abs{\Bar{x}_i}^{\upkappa-1}\,, &\quad \text{if } 
\Bar{x}_i \,\ge\,
\frac{n\rho_i\sum_{j\in\sI\setminus\sI_0}\rho_j}{\sum_{j\in\sI_0}\rho_j}\,.
\end{cases}
\quad \forall\,i\in\sI_0\,.
\end{equation*}
Recall $\overline{\Lg}^{z^n}_{n,\psi}$ in \cref{ES4.3B}. 
We also define 
\begin{equation*}
\overline{q}^{n,k}_i(\Breve{x},z^n) \,=\, \int_{\RR_*}
q^n_i\bigl(\Breve{x} - n\upmu^n(y - k),z^n\bigr)\,
\Tilde{F}^{d^n_1}_{\Breve{x},k}(\D{y})\,.
\end{equation*}

%%%%%%%%%%%%%%%%%%%%%%%%%%%%%%%%%%%%%%%%%%%%%%%%%%%%%%%%%%%%%%%%%%%%%%%%%%%%%%%%
\begin{lemma}\label{LB.2}
Grant \cref{A2.1,A2.2,A3.2},
and let $\xi\in\RR^d_+$ be as in \cref{ELB.1A}.
Then, for any even integer $\upkappa \ge 2$ 
and any $\varepsilon>0$,  
there exist a positive constant $\overline{C}$, and $\Bar{n}\in\NN$,
%and a positive sequence $\{\varepsilon_n\}_{n\in\NN}$, which
%converges to $0$ as $n\to\infty$,
such that
\begin{equation}
\begin{aligned}\label{ELB.2A}
\overline{\Lg}^{{z}^n}_{n,\psi}\, \Tilde{g}_n(\Breve{x},h,\psi,k) &\,\le\,
\overline{C} n^{\nicefrac{\upkappa}{2}} 
+\varepsilon 
\sum_{i\in\sI\setminus\Tilde{\cK}_n(\Breve{x})}\abs{\Bar{x}_i}^{\upkappa}
+ \sum_{i\in\Tilde{\cK}_n(\Breve{x})}
\order\bigl(\abs{\Bar{x}_i}^{\upkappa-1}\bigr) \\ 
&\mspace{-70mu}+ \frac{1}{\sqrt{n}}\sum_{i\in\Tilde{\cK}_n(\Breve{x})}
\bigl(\psi \mu^n_i(\abs{{z}^n_i - n\rho_i}) + \psi\gamma^n_i {q}^n_i 
+ (1 - \psi)\gamma^n_i\overline{q}^{n,k}_i\bigr)
\order\bigl(\abs{\Bar{x}_i}^{\upkappa-1}\bigr) 
\end{aligned}
\end{equation}
for any $z^n\in\fZsm^n$, and all $(\Breve{x},h,\psi,k)\in\fD$ and  $n>\Bar{n}$.
\end{lemma}

\begin{proof}
It is straightforward to verify that
\begin{equation}\label{PLB.2B}
\begin{aligned}
&\abs{{g}_{n,i}(\Breve{x}_i \pm 1) - {g}_{n,i}(\Breve{x}_i)} \,=\,
\order(\abs{\Bar{x}_i}^{\upkappa-1})\,, \\
&\abs{\bigl({g}_{n,i}(\Breve{x}_i) - {g}_{n,i}(\Breve{x}_i - 1)\bigr)
- \bigl({g}_{n,i}(\Breve{x}_i + 1) - {g}_{n,i}(\Breve{x}_i)\bigr)} 
\,=\, \order(\abs{\Bar{x}_i}^{\upkappa-2}) \,,
\end{aligned}
\end{equation} 
for $i\in\sI$.
Repeating the calculation in \cref{PLB.1B,PLB.1C},
and applying \cref{PLB.1A,PLB.2B}, we have
\begin{equation}\label{PLB.2C}
\begin{aligned}
&\overline{\Lg}^{{{z}}^n}_{n,\psi}\Tilde{g}_{n}(\Breve{x},h,\psi,k) \,\le\, 
\frac{\psi + \upalpha^n(k)}{\vartheta^n} \\
&\mspace{50mu}
\Biggl[\sum_{i\in\Tilde{\cK}_n(\Breve{x})}\mu^n_i\xi_i \Bigl[
\bigl(\abs{\lambda^n_i - n\mu^n_i\rho_i} + \psi\mu^n_i\abs{{z}^n_i - n\rho_i}
+ \psi\gamma_i^n {q}^n_i + (1 - \psi)\gamma^n_i\overline{q}^{n,k}_i\bigr)
\order(\abs{\Bar{x}_i}^{\upkappa-1}) \\
&\mspace{200mu}
+ \eta^n_i(h_i)\bigl(\psi\mu^n_i {z}^n_i + \psi\gamma^n_i {q}^n_i +
(1 - \psi)\gamma^n_i\overline{q}^{n,k}_i\bigr)
 \order(\abs{\Bar{x}_i}^{\upkappa-2})\Bigr] \\ 
&\mspace{50mu}
+ \sum_{i\in\sI\setminus\Tilde{\cK}_n(\Breve{x})}\mu^n_i\xi_i\Bigl[
\bigl(\lambda^n_i + (1 - \psi) n\mu^n_i\rho_i  \\
&\mspace{200mu}
+ \bigl(1 + \eta^n_i(h_i)\bigr)
(\psi\mu^n_i {z}^n_i + \psi\gamma^n_i {q}^n_i
+ (1 - \psi)\gamma^n_i\overline{q}^{n,k}_i \bigr)
\order\bigl(\abs{\Bar{x}_i}^{\upkappa-1}\bigr)\Bigr]\Biggr]\,.
\end{aligned}
\end{equation}
Note that $\overline{q}^{n,k}_i \le c(1 + \langle e,\Bar{x}\rangle^{+})$ 
for some positive constant $c$, by \cref{PL5.4C}.
Since ${z}^n_i, {q}^n_i \le \Bar{x}_i + n\rho_i$,
$(\vartheta^n)^{-1}$ is of order $n^{-\nicefrac{1}{2}}$ by \cref{A2.2},
and $\eta^n_i$ and $\upalpha^n$ are bounded,
it follows by \cref{PL5.4B,PLB.2C} that
\begin{equation*}
\begin{aligned}
&\overline{\Lg}^{ {z}^n}_{n,\psi}\Tilde{g}_{n}(\Breve{x},h,\psi,k) \,\le\, 
\sum_{i\in\sI\setminus\Tilde{\cK}_n(\Breve{x})}\frac{1}{\sqrt{n}}
\bigl(\order(n)\order(\abs{\Bar{x}_i}^{\upkappa-1})
+ \order(\abs{\Bar{x}_i}^\upkappa) \bigr) 
+ \sum_{i\in\Tilde{\cK}_n(\Breve{x})}\order(\sqrt{n})
\order(\abs{\Bar{x}_i}^{\upkappa-2})\\
&\mspace{50mu}
+ \sum_{i\in\Tilde{\cK}_n(\Breve{x})}
\frac{1}{\sqrt{n}}\bigl(\order(\sqrt{n}) + \psi\mu^n_i\abs{{z}^n_i - n\rho_i}
+ \psi\gamma_i^n {q}^n_i + (1 - \psi)\gamma^n_i\overline{q}^{n,k}_i\bigr)
\order(\abs{\Bar{x}_i}^{\upkappa-1})\,.
\end{aligned} 
\end{equation*}
Thus, applying Young's inequality, we obtain \cref{ELB.2A}, and this completes the proof.
\end{proof}

%%%%%%%%%%%%%%%%%%%%%%%%%%%%%%%%%%%%%%%%%%%%%%%%%%%%%%%%%%%%%%%%%%%%%%%%%%%%%%%%
\begin{proof}[Proof of \cref{L4.1}]
We define the function 
$\Tilde{f}_n \in \Cc(\Rd\times\RR^d_+\times\{0,1\}\times\RR_+)$ by
\begin{equation*}
\Tilde{f}_n(\Breve{x},h,\psi,k) \,\df\, 
f_n(\Breve{x},h) + \tilde{g}_n(\Breve{x},h,\psi,k)\,, 
\end{equation*}
with $f_n$ and $\Tilde{g}_n$ in \cref{ELB.1A} and \cref{ESBB}, respectively.
Recall $\widetilde{\Lyap}^n_{\upkappa,\xi}$ in \cref{ES4.3F}.
With $\xi\in\RR^d_+$ as in \cref{ELB.1A}, we have
\begin{equation*}
n^{\nicefrac{\upkappa}{2}}
\widetilde{\Lyap}^n_{\upkappa,\xi}(\Tilde{x}^n(\Breve{x}),h,\psi,k) 
\,=\, \Tilde{f}_n(\Breve{x},h,\psi,k) \qquad \forall\, (\Breve{x},h,\psi,k)\in\fD\,.
\end{equation*}
Hence, to prove \cref{EL4.1B},
it suffices to show that
\begin{equation}\label{PL4.1B} 
\Breve{\Lg}^{\Check{z}^n}_n \Tilde{f}_{n}(\Breve{x},h,\psi,k) 
\,\le\, \widetilde{C}_0  n^{\nicefrac{\upkappa}{2}} -  \widetilde{C}_1
\sum_{i\in\sI\setminus\Tilde{\cK}_n(x)}\xi_i\abs{\Bar{x}_i}^{\upkappa}
- \widetilde{C}_1 \sqrt{n}\sum_{i\in\Tilde{\cK}_n(\Breve{x})}
\xi_i\abs{\Bar{x}_i}^{\upkappa-1} \qquad \forall\, n > \Breve{n}\,,
\end{equation}
and all $(\Breve{x},h,\psi,k)\in\fD$,
where the generator $\Breve{\Lg}^{\Check{z}^n}_n$ is given in
\cref{ES4.3A}.
It is clear that $\cQ_{n,\psi}f_n(\Breve{x},h) = 0$.
Since $(\vartheta^n)^{-1}$ is of order $n^{\nicefrac{-1}{2}}$,
it follows by \cref{ES4.2D,ES4.3D} that
\begin{equation}\label{PL4.1G}
\begin{aligned}
\cQ_{n,0}\Tilde{g}_n(\Breve{x},h,0,k) &\,\le\, 
\sum_{i\in\sI\setminus\Tilde{\cK}_n(\Breve{x})}-\mu^n_i\xi_i\abs{\Bar{x}_i}^\upkappa 
+ \sum_{i\in\Tilde{\cK}_n(\Breve{x})}- 
\mu^n_i\xi_i\frac{n\rho_i\sum_{j\in\sI\setminus\sI_0}\rho_j}{\sum_{j\in\sI_0}\rho_j}
\abs{\Bar{x}_i}^{\upkappa-1} \\
&\mspace{100mu}
+ \epsilon_n\sum_{i\in\sI\setminus\Tilde{\cK}_n(\Breve{x})}
\order(\abs{\Bar{x}_i}^\upkappa)
+ \sum_{i\in\Tilde{\cK}_n(\Breve{x})}\order(\sqrt{n})
\order(\abs{\Bar{x}_i}^{\upkappa-1})\,,
\end{aligned}
\end{equation}
where $C$ is some positive constant and $\epsilon_n\to0$ as $n\to\infty$. 
Since all the moments of $d_1$ are finite by \cref{EA3.2A} and 
$(a+z)^\upkappa - a^\upkappa = \order(z)\order(a^{\upkappa-1}) 
+ \order(z^2)\order(a^{\upkappa-2})+ \dots + \order(z^\upkappa)$
for any $a,z\in\RR$, it is easy to verify that 
\begin{equation}\label{PL4.1D}
\cI_{n,1} \Hat{f}_n(\Breve{x},h,1,0) \,=\, \sum_{i\in\sI}
\sum^\upkappa_{j=1}\order(n^{\nicefrac{j}{2}})
\order(\abs{\Bar{x}_i}^{\upkappa-j})\,,
\end{equation}
 using also the fact that
\begin{equation*}
\beta^n_{\mathsf{u}}\int_{R_*} \biggl(\frac{n}{\vartheta^n}\mu^n_i\rho_iz\biggr)^j
F^{d_1}(\D{z}) \,=\, 
\beta^n_{\mathsf{u}}\biggl(\frac{n}{\vartheta^n}\biggr)^j(\mu^n_i\rho_i)^j
\Exp\bigl[(d_1)^j\bigr]
\,=\, \order(n^{\nicefrac{j}{2}})\quad \forall\,j>0\,,
\end{equation*}
which follows by
by \cref{A2.1,A2.2,EA3.2A}.
Then, for $\psi=1$, it follows by \cref{PL4.1D} and Young's inequality that
\begin{equation}\label{PL4.1H}
\begin{aligned}
\Breve{\Lg}^{\Check{z}^n}_n \Tilde{f}_n(\Breve{x},h,1,0) &\,\le\, 
\Lg^{\Check{z}^n}_n f_n(\Breve{x},h) + 
\overline{\Lg}^{\Check{z}^n}_{n,1}\Tilde{g}_n(\Breve{x},h,1,0) \\ 
&\mspace{50mu}+ Cn^{\nicefrac{\upkappa}{2}} 
+ \epsilon_n\sum_{i\in\sI\setminus\Tilde{\cK}_n(\Breve{x})}
\order(\abs{\Bar{x}_i}^\upkappa)
+ \sum_{i\in\Tilde{\cK}_n(\Breve{x})}\order(\sqrt{n})
\order(\abs{\Bar{x}_i}^{\upkappa-1})\,. 
\end{aligned}
\end{equation}
Note that the last two terms in \cref{ELB.1B} and
the last term in \cref{ELB.2A} are of smaller order than the second
and third terms on the right-hand side  of \cref{ELB.1B}, respectively.
Thus, applying \cref{LB.1,LB.2}, and using \cref{PL4.1H}, we obtain
\begin{equation}\label{PL4.1I}
n^{-\nicefrac{\upkappa}{2}} \Breve{\Lg}^{\Check{z}^n}_n \Tilde{f}_n(\Breve{x},h,1,0)
\,\le\, \widetilde{C}_0 
- \widetilde{C}_1\sum_{i\in\sI\setminus\Tilde{\cK}_n(\Tilde{x})}\abs{\Bar{x}_i}^\upkappa 
- \widetilde{C}_1\sum_{i\in\Tilde{\cK}_n(\Breve{x})}n^{-\nicefrac{1}{2}}
\bigl(\mu^n_i(\Check{z}^n_i - n\rho_i)
+\gamma^n_i\Check{q}^n_i\bigr)\abs{\Tilde{x}_i}^{\upkappa-1}
\end{equation}
for all large enough $n$,
where $\Tilde{x}$ is defined in \cref{D4.2}.
Since $\Check{q}^n_i \ge 0$ and $\Check{z}^n_i - n\rho_i>0$
for $i\in\Tilde\cK_n(\Breve{x})$, then by using \cref{PLB.1D} and \cref{PL4.1I},
we see that \cref{PL4.1B} holds when $y = 1$.

For $\psi = 0$, using \cref{PL4.1G}, Young's inequality,
and the fact that for $i\in\Tilde{\cK}_n(\Breve{x})$, $\Bar{x}_i > 0$,
we obtain
\begin{align*}
&\Breve{\Lg}^{\Check{z}^n}_n \Tilde{f}_{n}(\Breve{x},h,0,k) \,\le\,
\sum_{i\in\sI}\order(\sqrt{n})\order(\abs{\Bar{x}_i}^{\upkappa-1})
+ \sum_{i\in\sI}\order(n)\order(\abs{\Bar{x}_i}^{\upkappa-2})
+ Cn^{\nicefrac{\upkappa}{2}} \nonumber\\
& \mspace{50mu}
+ (\epsilon + \epsilon_n)\sum_{i\in\sI\setminus\Tilde{\cK}_n(\Breve{x})}
\xi_i\abs{\Bar{x}_i}^\upkappa
+ \sum_{i\in\sI\setminus\Tilde{\cK}_n(\Breve{x})}
\biggl(- \mu^n_i\xi_i\abs{\Bar{x}_i}^\upkappa
+ \gamma^n_i\xi_i\overline{q}^{n,k}_i\bigl(-\upkappa(\Bar{x}_i)^{\upkappa-1} 
+ \order(\abs{\Bar{x}_i}^{\upkappa-2})\bigr)\biggr) \nonumber\\
&\mspace{100mu} + \sum_{i\in\Tilde{\cK}_n(\Breve{x})}
-\frac{n\rho_i\sum_{j\in\sI\setminus\sI_0}\rho_j}{\sum_{j\in\sI_0}\rho_j}\mu^n_i\xi_i
\abs{\Bar{x}_i}^{\upkappa-1}
+ \overline{\Lg}^{\check{z}^n}_{n,0}\tilde{g}_n(\Breve{x},h,0,k)
\end{align*}
for some positive constant $C$ and sufficiently small $\epsilon>0$.
We proceed by invoking the argument in the proof of \cite[Lemma 5.1]{ABP15}.
The important difference here is that
\begin{equation*}
\Check{q}^n_i\bigl(\Breve{x} - n\upmu^n(z - k)\bigr)  \,=\,
\Tilde{\epsilon}_i\bigl(\Breve{x} - n\upmu^n(z - k)\bigr)
\bigl(\Bar{x}_i - n\mu_i\rho_i(z - k)\bigr)\\
+ \bar{\epsilon}_i\bigl(\Breve{x} - n\upmu^n(z - k)\bigr)
\sum_{j=1}^{i-1}\bigl(\Bar{x}_j - n\mu_j\rho_j(z - k)\bigr)\,,  
\end{equation*}
where the functions $\Tilde{\epsilon}_i, \bar{\epsilon}_i \colon \Rd \to [0,1]$,
for $i\in\sI$. Since $\Tilde{\epsilon}_i$ and $\bar{\epsilon}_i$
are bounded,
we have some additional terms which are bounded by 
$C \int_{\RR_*}n\mu_i\rho_i(y - k)\,\Tilde{F}^{d^n_1}_{\Breve{x},k}(\D{y})$
for some positive constant $C$.
Therefore, these are of order $\sqrt{n}$ by \cref{PL5.4C}.
Thus, repeating the argument in the proof of \cref{LB.1},
and applying \cref{LB.2}, we deduce that
\cref{PL4.1B} holds with $\psi=0$. This completes the proof.
\end{proof}

%%%%%%%%%%%%%%%%%%%%%%%%%%%%%%%%%%%%%%%%%%%%%%%%%%%%%%%%%%%%%%%%%%%%%%%%%%%%%%%%
\begin{proof}[Proof of \cref{L5.2}]
The proof mimics that of \cref{T4.1}.
We sketch the proof when $\sI_0$ is empty. 
Using the estimate
\begin{equation}\label{PL5.2A}
\order(q^n_i)\order(\abs{\Bar{x}_i}^{m-1}) \,\le\, 
\epsilon^{1-m}\bigl(\order(q^n_i)\bigr)^m + 
\epsilon\bigl(\order(\abs{\Bar{x}_i}^{m-1})\bigr)^{\nicefrac{m}{m-1}}
\end{equation}
for any $\epsilon>0$, which follows by Young's inequality,
we deduce that, for some positive constants $\{c_k \colon k=1,2,3\}$, we have
\begin{equation}\label{PL5.2B}
\Lg^{z^n}_n f_n(\Breve{x},h) \,\le\, c_1 n^{\nicefrac{m}{2}} + 
c_2(\langle e,q^n\rangle)^m - c_3\sum_{i\in\sI}\xi_i \abs{\Bar{x}_i}^m
\quad \forall\, (\Breve{x},h)\in\RR^d_+\times\RR^d_+\,,
\end{equation}
 and all large enough $n$.
Note that \cref{LB.2} holds for all $z^n\in\fZsm^n$.
Then, we may repeat the steps in the proof of \cref{L4.1}, except that
here we use
\begin{equation}\label{PL5.2C}
\begin{aligned}
&(\Tilde{x}_i)^{m-1}\int_{\RR_*}\Hat{q}^n_i\bigl(\Breve{x} - n\upmu^n(y - k),z^n\bigr)\,
\Tilde{F}^{d^n_1}_{\Breve{x},k}(\D{y}) \\
&\mspace{150mu}
\,\le\, \epsilon\abs{\Bar{x}_i}^{m}
 + \epsilon^{1-m}\Bigl(\Exp
\bigl[\Hat{q}^n_i\bigl(\Breve{x} - n\upmu^n(d^n_1 - k),z^n\bigr)
\,|\, d^n_1 > k\bigr]\Bigr)^m \,,
\end{aligned}
\end{equation}
where $\Hat{q}^n = n^{\nicefrac{-1}{2}}q^n$,
with $\epsilon>0$ chosen sufficiently small.
Since $\Hat{q}^n_i(\breve{x},z^n) \le \langle e,\Tilde{x} \rangle^+$,
it follows by \cref{PL5.4C} that
\begin{equation}\label{PL5.2E}
\Exp\bigl[\Hat{q}^n_i\bigl(\Breve{x} - n\upmu^n(d^n_1 - k),z^n\bigr)
\bigm| d^n_1 > k\bigr] \,\le\, c_4(1 + \langle e,\Tilde{x} \rangle^+)\,.
\end{equation}
Thus, by the same calculation in \cref{T4.1},
and using \cref{PL5.2A,PL5.2B,PL5.2C,PL5.2E}, we obtain 
\begin{equation}\label{PL5.2D}
\begin{aligned}
\Exp^{z^n}\biggl[\int_0^{T}\abs{\widetilde{X}^n(s)}^m\biggr]
\,\le\, C_1(T + \abs{\Hat{X}^n(0)}^m)
+ C_2\Exp^{z^n}
\biggl[\int_0^{T}\bigl(1 + \langle e,\widetilde{X}^n(s)\rangle^{+}\bigr)^m\,\D{s}\biggr] 
\end{aligned}
\end{equation}
for all large enough $n$, and $\{z^n\in\fZsm^n\colon n\in\NN\}$.
Since $\sup_n \Hat{J}(\Hat{X}^n(0),{z}^n)<\infty$,
it follows by \cref{PT4.1G} that
\begin{equation*}
\sup_{n}\,\limsup_{T\rightarrow\infty}\,\frac{1}{T}\Exp\biggl[
\int_0^{T}\bigl(\langle e,\widetilde{X}^n(s)\rangle^{+}\bigr)^m\,\D{s}\biggr]
\,<\, \infty\,.
\end{equation*}
Therefore, dividing both sides of \cref{PL5.2D} by $T$, 
taking $T\rightarrow\infty$ and using \cref{PT4.1G} again, 
we obtain \cref{EL5.2A}.
We may show that the result also holds when $\sI_0$ is nonempty by repeating 
the above argument and applying \cref{LB.2}. This completes the proof. 
\end{proof}

%%%%%%%%%%%%%%%%%%%%%%%%%%%%%%%%%%%%%%%%%%%%%%%%%%%%%%%%%%%%%%%%%%%%%%%%%%%%%%%%
\section*{Acknowledgments}
This research was supported in part by 
the Army Research Office through grant W911NF-17-1-001,
in part by the National Science Foundation through grants DMS-1715210,
CMMI-1538149 and DMS-1715875,
and in part by the Office of Naval Research through grant N00014-16-1-2956
and was approved for public release under DCN \#43-5442-19.

%%%%%%%%%%%%%%%%%%%%%%%%%%%%%%%%%%%%%%%%%%%%%%%%%%%%%%%%%%%%%%%%%%%%%%%%%%%%%%%%
\def\polhk#1{\setbox0=\hbox{#1}{\ooalign{\hidewidth
  \lower1.5ex\hbox{`}\hidewidth\crcr\unhbox0}}}


\begin{bibdiv}
\begin{biblist}

\bib{AMR04}{article}{
      author={Atar, Rami},
      author={Mandelbaum, Avi},
      author={Reiman, Martin~I.},
       title={Scheduling a multi class queue with many exponential servers:
  asymptotic optimality in heavy traffic},
        date={2004},
        ISSN={1050-5164},
     journal={Ann. Appl. Probab.},
      volume={14},
      number={3},
       pages={1084\ndash 1134},
      review={\MR{2071417}},
}

\bib{Atar05}{article}{
      author={Atar, R.},
       title={Scheduling control for queueing systems with many servers:
  asymptotic optimality in heavy traffic},
        date={2005},
        ISSN={1050-5164},
     journal={Ann. Appl. Probab.},
      volume={15},
      number={4},
       pages={2606\ndash 2650},
         url={https://doi.org/10.1214/105051605000000601},
      review={\MR{2187306}},
}

\bib{AMS09}{article}{
      author={Atar, R.},
      author={Mandelbaum, A.},
      author={Shaikhet, G.},
       title={Simplified control problems for multiclass many-server queueing
  systems},
        date={2009},
        ISSN={0364-765X},
     journal={Math. Oper. Res.},
      volume={34},
      number={4},
       pages={795\ndash 812},
         url={https://doi.org/10.1287/moor.1090.0404},
      review={\MR{2573496}},
}

\bib{ABP15}{article}{
      author={Arapostathis, Ari},
      author={Biswas, Anup},
      author={Pang, Guodong},
       title={Ergodic control of multi-class {$M/M/N+M$} queues in the
  {H}alfin-{W}hitt regime},
        date={2015},
        ISSN={1050-5164},
     journal={Ann. Appl. Probab.},
      volume={25},
      number={6},
       pages={3511\ndash 3570},
      review={\MR{3404643}},
}

\bib{AP18}{article}{
      author={Arapostathis, Ari},
      author={Pang, Guodong},
       title={Infinite-horizon average optimality of the {N}-network in the
  {H}alfin-{W}hitt regime},
        date={2018},
     journal={Math. Oper. Res.},
      volume={43},
      number={3},
       pages={838\ndash 866},
      review={\MR{3846075}},
}

\bib{AP19}{article}{
      author={Arapostathis, Ari},
      author={Pang, Guodong},
       title={Infinite horizon asymptotic average optimality for large-scale
  parallel server networks},
        date={2019},
     journal={Stochastic Process. Appl.},
      volume={129},
      number={1},
       pages={283\ndash 322},
      review={\MR{3906999}},
}

\bib{BGL14}{article}{
      author={Budhiraja, A.},
      author={Ghosh, A.},
      author={Liu, X.},
       title={Scheduling control for {M}arkov-modulated single-server
  multiclass queueing systems in heavy traffic},
        date={2014},
     journal={Queueing Syst.},
      volume={78},
      number={1},
       pages={57\ndash 97},
      review={\MR{3238008}},
}

\bib{RMH13}{article}{
      author={Kumar, R.},
      author={Lewis, M.~E.},
      author={Topaloglu, H.},
       title={Dynamic service rate control for a single-server queue with
  {M}arkov-modulated arrivals},
        date={2013},
     journal={Naval Res. Logist.},
      volume={60},
      number={8},
       pages={661\ndash 677},
      review={\MR{3146992}},
}

\bib{LQA17}{article}{
      author={Xia, L.},
      author={He, Q.},
      author={Alfa, A.~S.},
       title={Optimal control of state-dependent service rates in a {MAP}/{M}/1
  queue},
        date={2017},
     journal={IEEE Trans. Automat. Control},
      volume={62},
      number={10},
       pages={4965\ndash 4979},
      review={\MR{3708873}},
}

\bib{ADPZ19}{article}{
      author={Arapostathis, Ari},
      author={Das, Anirban},
      author={Pang, Guodong},
      author={Zheng, Yi},
       title={Optimal control of {M}arkov-modulated multiclass many-server
  queues},
        date={2019},
     journal={Stochastic Systems},
      volume={9},
      number={2},
       pages={155\ndash 181},
}

\bib{JMTW17}{article}{
      author={Jansen, H.~M.},
      author={Mandjes, M.},
      author={De~Turck, K.},
      author={Wittevrongel, S.},
       title={Diffusion limits for networks of {M}arkov-modulated
  infinite-server queues},
        date={2017},
     journal={ArXiv e-prints},
      volume={1712.04251},
      eprint={https://arxiv.org/abs/1712.04251},
}

\bib{PW09B}{article}{
      author={Pang, G.},
      author={Whitt, W.},
       title={Service interruptions in large-scale service systems},
        date={2009},
     journal={Management Science},
      volume={55},
      number={9},
       pages={1499\ndash 1512},
}

\bib{PW09}{article}{
      author={Pang, G.},
      author={Whitt, W.},
       title={Heavy-traffic limits for many-server queues with service
  interruptions},
        date={2009},
        ISSN={0257-0130},
     journal={Queueing Syst.},
      volume={61},
      number={2-3},
       pages={167\ndash 202},
         url={https://doi.org/10.1007/s11134-009-9104-2},
      review={\MR{2485887}},
}

\bib{LPZ16}{article}{
      author={Lu, H.},
      author={Pang, G.},
      author={Zhou, Y.},
       title={{$G/GI/N(+GI)$} queues with service interruptions in the
  {H}alfin-{W}hitt regime},
        date={2016},
        ISSN={1432-2994},
     journal={Math. Methods Oper. Res.},
      volume={83},
      number={1},
       pages={127\ndash 160},
         url={https://doi.org/10.1007/s00186-015-0523-z},
      review={\MR{3464192}},
}

\bib{LP17}{article}{
      author={Lu, H.},
      author={Pang, G.},
       title={Heavy-traffic limits for an infinite-server fork-join queueing
  system with dependent and disruptive services},
        date={2017},
        ISSN={0257-0130},
     journal={Queueing Syst.},
      volume={85},
      number={1-2},
       pages={67\ndash 115},
         url={https://doi.org/10.1007/s11134-016-9505-y},
      review={\MR{3604118}},
}

\bib{PZ16}{article}{
      author={Pang, G.},
      author={Zhou, Y.},
       title={{$G/G/\infty$} queues with renewal alternating interruptions},
        date={2016},
        ISSN={0001-8678},
     journal={Adv. in Appl. Probab.},
      volume={48},
      number={3},
       pages={812\ndash 831},
         url={https://doi.org/10.1017/apr.2016.29},
      review={\MR{3568893}},
}

\bib{APZ19a}{article}{
      author={Arapostathis, A.},
      author={Pang, G.},
      author={Zheng, Y.},
       title={Ergodic control of diffusions with compound {P}oisson jumps under
  a general structural hypothesis},
        date={2019},
     journal={ArXiv e-prints},
      volume={1908.01068},
      eprint={https://arxiv.org/abs/1908.01068},
}

\bib{Atar11}{article}{
      author={Atar, R.},
      author={Giat, C.},
      author={Shimkin, N.},
       title={On the asymptotic optimality of the {$c\mu/\theta$} rule under
  ergodic cost},
        date={2011},
        ISSN={0257-0130},
     journal={Queueing Syst.},
      volume={67},
      number={2},
       pages={127\ndash 144},
         url={https://doi.org/10.1007/s11134-010-9206-x},
      review={\MR{2771197}},
}

\bib{Takis-99}{article}{
      author={Konstantopoulos, T.},
      author={Last, G.},
       title={On the use of {L}yapunov function methods in renewal theory},
        date={1999},
        ISSN={0304-4149},
     journal={Stochastic Process. Appl.},
      volume={79},
      number={1},
       pages={165\ndash 178},
      review={\MR{1670534}},
}

\bib{MT-III}{article}{
      author={Meyn, Sean~P.},
      author={Tweedie, R.~L.},
       title={Stability of {M}arkovian processes. {III}. {F}oster-{L}yapunov
  criteria for continuous-time processes},
        date={1993},
     journal={Adv. in Appl. Probab.},
      volume={25},
      number={3},
       pages={518\ndash 548},
      review={\MR{1234295}},
}

\bib{Patrick-99}{book}{
      author={Billingsley, P.},
       title={Convergence of probability measures},
     edition={Second},
      series={Wiley Series in Probability and Statistics: Probability and
  Statistics},
   publisher={John Wiley \& Sons, Inc., New York},
        date={1999},
        ISBN={0-471-19745-9},
         url={https://doi.org/10.1002/9780470316962},
        note={A Wiley-Interscience Publication},
      review={\MR{1700749}},
}

\bib{WW-02}{book}{
      author={Whitt, W.},
       title={Stochastic-process limits},
      series={Springer Series in Operations Research},
   publisher={Springer-Verlag, New York},
        date={2002},
        ISBN={0-387-95358-2},
        note={An introduction to stochastic-process limits and their
  application to queues},
      review={\MR{1876437}},
}

\bib{Dai-95}{article}{
      author={Dai, J.~G.},
       title={On positive {H}arris recurrence of multiclass queueing networks:
  a unified approach via fluid limit models},
        date={1995},
        ISSN={1050-5164},
     journal={Ann. Appl. Probab.},
      volume={5},
      number={1},
       pages={49\ndash 77},
  url={http://links.jstor.org/sici?sici=1050-5164(199502)5:1<49:OPHROM>2.0.CO;2-4&origin=MSN},
      review={\MR{1325041}},
}

\bib{APS19}{article}{
      author={Arapostathis, Ari},
      author={Pang, Guodong},
      author={Sandri\'{c}, Nikola},
       title={Ergodicity of a {L}\'{e}vy-driven {SDE} arising from multiclass
  many-server queues},
        date={2019},
     journal={Ann. Appl. Probab.},
      volume={29},
      number={2},
       pages={1070\ndash 1126},
      review={\MR{3910024}},
}

\bib{AHPS19}{collection}{
      author={Arapostathis, Ari},
      author={Hmedi, Hassan},
      author={Pang, Guodong},
      author={Sandri\'{c}, Nikola},
      editor={Yin, George},
      editor={Zhang, Qing},
       title={Uniform polynomial rates of convergence for a class of
  {L}\'evy-driven controlled {SDE}s arising in multiclass many-server queues},
      series={{M}odeling, {S}tochastic {C}ontrol, {O}ptimization, and
  {A}pplications. The IMA Volumes in Mathematics and its Applications},
   publisher={Springer, Cham},
        date={2019},
      volume={164},
}

\bib{Davis-84}{article}{
      author={Davis, M. H.~A.},
       title={Piecewise-deterministic {M}arkov processes: a general class of
  nondiffusion stochastic models},
        date={1984},
        ISSN={0035-9246},
     journal={J. Roy. Statist. Soc. Ser. B},
      volume={46},
      number={3},
       pages={353\ndash 388},
        note={With discussion},
      review={\MR{790622}},
}

\bib{Ross96}{book}{
      author={Ross, Sheldon~M.},
       title={Stochastic processes},
     edition={Second},
   publisher={John Wiley \& Sons, Inc., New York},
        date={1996},
        ISBN={0-471-12062-6},
      review={\MR{1373653}},
}

\bib{ACPZ19}{article}{
      author={Arapostathis, Ari},
      author={Caffarelli, Luis},
      author={Pang, Guodong},
      author={Zheng, Yi},
       title={Ergodic control of a class of jump diffusions with finite
  {L}\'{e}vy measures and rough kernels},
        date={2019},
     journal={SIAM J. Control Optim.},
      volume={57},
      number={2},
       pages={1516\ndash 1540},
      review={\MR{3942851}},
}

\bib{Krichagina}{article}{
      author={Krichagina, E.~V.},
      author={Taksar, M.~I.},
       title={Diffusion approximation for {$GI/G/1$} controlled queues},
        date={1992},
        ISSN={0257-0130},
     journal={Queueing Systems Theory Appl.},
      volume={12},
      number={3-4},
       pages={333\ndash 367},
         url={https://doi.org/10.1007/BF01158808},
      review={\MR{1200872}},
}

\bib{PTW07}{article}{
      author={Pang, G.},
      author={Talreja, R.},
      author={Whitt, W.},
       title={Martingale proofs of many-server heavy-traffic limits for
  {M}arkovian queues},
        date={2007},
        ISSN={1549-5787},
     journal={Probab. Surv.},
      volume={4},
       pages={193\ndash 267},
         url={https://doi.org/10.1214/06-PS091},
      review={\MR{2368951}},
}

\bib{IW-71}{article}{
      author={Iglehart, D.~L.},
      author={Whitt, W.},
       title={The equivalence of functional central limit theorems for counting
  processes and associated partial sums},
        date={1971},
        ISSN={0003-4851},
     journal={Ann. Math. Statist.},
      volume={42},
       pages={1372\ndash 1378},
         url={https://doi.org/10.1214/aoms/1177693249},
      review={\MR{0310941}},
}

\end{biblist}
\end{bibdiv}
\end{document}